\documentclass[leqno,10pt]{amsart}
\usepackage{longtable,geometry}
\usepackage{graphicx}
\usepackage[english]{babel}
\usepackage[utf8]{inputenc}
\usepackage[T1]{fontenc}
\usepackage{amsmath,amsfonts,amsthm,amssymb}
\usepackage{bbm}
\usepackage{pict2e}
\usepackage[dvipsnames]{xcolor}
\usepackage{url}
\usepackage[normalem]{ulem}
\usepackage{array}

\numberwithin{equation}{section} 

\newtheorem{theorem}{Theorem}[section]
\newtheorem{corollary}[theorem]{Corollary}
\newtheorem{prop}[theorem]{Proposition}
\newtheorem{lemma}[theorem]{Lemma}

\newtheorem{remark}{Remark}[section]\theoremstyle{remark}
\newtheorem{definition}{Definition}[section]\theoremstyle{definition}

\newcommand{\ind}{\mathbbm{1}}
\newcommand{\gr}[1]{{\mathbf{#1}}}
\newcommand{\ui}{{\underline{i}}}
\newcommand{\ts}{{\tau}_{\text{stop}}}
\newcommand{\Zt}{{\text{Z}}}
\newcommand{\Zc}{{\mathsf{Z}}}
\newcommand{\N}{\mathcal{N}}
\newcommand{\so}{\overset{\circ}{\sim}}
\newcommand{\e}{\varepsilon}

\makeindex

\title[Long time validity of the linearized Boltzmann equation]{Long time validity of the linearized Boltzmann equation for hard spheres: a proof without billiard theory}

\begin{document}
	
	\author{Corentin Le Bihan}
	\email{corentin.le-bihan@ens-lyon.fr}
	\address{UMPA (UMR CNRS 5669), \'Ecole Normale Superieur de Lyon, 46 allée d'Italie, 69364 LYON, FRANCE}
	
	\maketitle
	\begin{abstract}
		We study space-time fluctuations of a hard sphere system at thermal equilibrium, and prove that the covariance converges to the solution of a linearized Boltzmann equation in the low density limit, globally in time. This result has been obtained previously in \cite{BGSS1}, by using uniform bounds on the number of recollisions of dispersing systems of hard spheres (as provided for instance in \cite{BFK}). We present a self-contained proof with substantial differences, which does not use this geometric result. This can be regarded as the first step of a program aiming to derive the fluctuation theory of the rarefied gas, for interaction potentials different from hard spheres.
	\end{abstract}
	\tableofcontents
	
	\section{Introduction}
	Consider a system of $N$ particles in a box $\Lambda\subset\mathbb{R}^d$ ($d\geq 3$), interacting by means of a two-body potential $\mathcal{V}_\e(\cdot):=\mathcal{V}(\cdot/\e)$. We are interested in the behavior of the system as the number of particles goes to infinity and the interaction length scale $\e$ is fixed by the Boltzmann-Grad scaling $N\e^{d-1} = 1$. It is a limit of low density where the mean free path of a particle between two collisions is of order $O(1)$.
	
	Away from equilibrium, it is expected that the system is governed by the Boltzmann equation in the low density limit. However most of the existing rigorous results are valid for short time, such that only a small fraction of the particles actually interact. The first convergence proofs were provided in the fundamental work of Lanford \cite{Lanford} for hard spheres and by King \cite{king_bbgky_1975} for different finite range potentials (see also \cite{CIR,Spohn3}). Quantitative convergence bounds have been obtained later (see \cite{GTS,PSS}).
	
	Illner and Pulvirenti proved a first long time convergence result in \cite{IP} (see also \cite{Delinger}), but only for a very diluted gas in the whole space, where dispersion is the dominant phenomena. Other long time results have been obtained later on for a system of one labeled particle evolving in a background at equilibrium (see \cite{vBHLS} for arbitrary kinetic times and \cite{BGS2} for diffusive times). The law of the tagged particle follows then the linear Boltzmann equation. See also \cite{A,Catapano} for  adaptations of the proof to interaction potentials different from hard core.
	
		Looking at a tagged particle in a background at equilibrium can be seen as a perturbation of order $O(1)$ of the equilibrium measure. The next natural step is to study small fluctuations around equilibrium which can be seen as perturbations of order $O(N)$ (we are interested in the square of the small fluctuations). Note that a "final step" would be to understand on long time non equilibrium factorized measures, which are $O(C^N)$ pertubations. 
		
		In the low density limit, the fluctuations behave like a Gaussian field with covariance governed by the linearized Boltzmann equation, as predicted in \cite{Spohn,Spohn2}. The rigorous proof is separated in two main parts: first the convergence of the covariance and second checking asymptotically the Wick's rules characterizing the higher order moments; treated first for short times, respectively in \cite{Spohn}, and \cite{BGSS3,BGSS} in the more general context of non equilibrium states. Concerning the global in time result, the Wick's rule has been treated recently in the case of hard spheres in \cite{BGSS2}. Convergence of the covariance has been obtained first for hard disks in dimension $2$ in the canonical ensemble (see \cite{BGS}), using that the partition function is uniformly bounded (independently of $\e$), which is a specificity of dimension $2$. Later on a proof has been given for dimension $3$ in \cite{BGSS1}, in the grand canonical ensemble.
	
	The purpose of the present work is to propose a different method of proof for the result in \cite{BGSS1}. As known, a crucial part of the argument leading to the Boltzmann equation amounts to showing that dynamical memory effects (called recollisions) are vanishing in the limit. The long time result is based then on a sampling checking the trajectories carefully and eliminating the recollisions on very small time scales (of order $\delta$, a power of $\e$). On these scales, it is used in \cite{BGSS1} that the dynamics is decomposed on independent clusters of finite size, each of which behaves as a dispersing billiard with uniformly bounded number of collisions. The latter property is unproved (possibly false) for arbitrary potentials with compact support (defining, say, a collision as a two-by-two interaction at distance $\e$). Even in the case of hard spheres, the property is delicate: explicit bounds have been provided in \cite{BFK} by means of refined geometric techniques.
	
	This motivates us to develop a different argument circumventing any uniform control on recollision numbers. The main ingredients are a subtle conditioning of the initial data forbidding explosions of the number of recollisions, together with a suitable dynamical cumulant decomposition method, inspired by \cite{BGSS}.
	
	\subsection{Definition of the system}
	Let $\Lambda:=\mathbb{R}^d/\mathbb{Z}^d$ (with $d\geq3$) be the domain. We denote $\mathbb{D}=\Lambda\times\mathbb{R}^d$ its tangent bundle and $\mathcal{D}_\e^n\subset\mathbb{D}^n$ the $n$-particle canonical phase space:
	\begin{equation}
	\mathcal{D}^n_\e:=\Big\{Z_n:=(x_1,v_1,\cdots,x_n,v_n)\in\mathbb{D}^n,~\text{for~} 1\leq i<j\leq n,~|x_i-x_j|>\e\Big\}.
	\end{equation}
	Here and in the following, we use the notation 
	\[X_n= (x_1,\cdots,x_n),~V_n=(v_1,\cdots,v_n),~\text{and}~z_i=(x_i,v_i).\]
	
	On each $\mathcal{D}^n_\e$ we construct the hard sphere dynamics as the Hamiltonian dynamics associated with the Hamiltonian
	\begin{equation}
	\mathcal{H}^\e_n(Z_n):= \frac{1}{2}\|V_n\|^2+\mathcal{V}_n^\e(X_n),~~\mathcal{V}_n^\e(X_n):=\sum_{1\leq i<j\leq n} \mathcal{V}\left(\frac{\|x_i-x_j\|}{\e}\right)
	\end{equation}
	where $\mathcal{V}$ is the hard core interaction potential
	\begin{equation}
	\mathcal{V}(r):= \left\{\begin{array}{l}0\text{~if~}|r|>1\\\infty\text{~else}\end{array}\right. .
	\end{equation}
	In this dynamics particles move along straight lines until they meet each other. If at time $\tau$ we have $|x_q(\tau)-x_{q'}(\tau)|=\e$, the outgoing velocities are given by the following scattering law:
	\begin{equation}\label{scattering}
	\left\{\begin{array}{rl}
	\displaystyle v_{q}(\tau^+)&\displaystyle = v_{q}(\tau^-) - \frac{x_{q'}(\tau)-x_q(\tau)}{|x_{q'}(\tau)-x_q(\tau)|}\cdot \Big(v_{q}(\tau^-)-v_{q'}(\tau^-)\Big)\frac{x_{q'}(\tau)-x_q(\tau)}{|x_{q'}(\tau)-x_q(\tau)|}\\
	\displaystyle v_{q'}(\tau^+)&\displaystyle= v_{q'}(\tau^-) + \frac{x_{q'}(\tau)-x_q(\tau)}{|x_{q'}(\tau)-x_q(\tau)|}\cdot \Big(v_{q}(\tau^-)-v_{q'}(\tau^-)\Big)\frac{x_{q'}(\tau)-x_q(\tau)}{|x_{q'}(\tau)-x_q(\tau)|}.
	\end{array}\right.
	\end{equation}
	This process is well defined for all times, almost everywhere in $\mathcal{D}_\e^n$ with respect to the Lebesgue measure (see \cite{Alexander}).
	
	We denote in the following $\mathcal{D}_\e:=\bigsqcup_{n\geq 0} \mathcal{D}_\e^n$ the grand canonical phase space and $\mathcal{N}$ the random number of particles.We can then extend the Hamiltonian dynamics to $\mathcal{D}_\e$ and denote $\gr{Z}_{\N}(t)$ the realization (defined almost surely) of the hard sphere flow on $\mathcal{D}_\e$ with random initial data $\gr{Z}_{\mathcal{N}}(0)$: for $\N = n$, $\gr{Z}_{\N}(t)$ follows the Hamiltonian dynamics on $\mathcal{D}^n_\e$.	
	
	The initial data is sampled according to the stationary measure introduced now. The \emph{grand canonical Gibbs measure} $\mathbb{P}_\e$ (and its expectation $\mathbb{E}_\e$) are defined on $\mathcal{D}_\e$ as follows. An application $G:\mathcal{D}_\e\rightarrow \mathbb{R}$ is a test function if there exists a sequence $(g_n)_{n\geq 0}$ with $g_n\in L^\infty(\mathbb{D}^n)$ and
	\[\text{for}~\mathcal{N}=n,~  G(\gr{Z}_{\N}):=g_n(\gr{Z}_{\N}).\]
	Then we define $\mathbb{E}_\e$ as
	\begin{equation}
	\mathbb{E}_\e[G(\gr{Z}_{\N})]:=\frac{1}{\mathcal{Z}_\e}\sum_{n\geq0} \frac{\mu_\e^n}{n!}\int_{\mathbb{D}^n}g_n(Z_n)\,\frac{e^{-\mathcal{H}^\e_n(Z_n)}}{(2\pi)^{nd/2}}dZ_n,
	\end{equation}
	where  $\mathcal{Z}_\e$ is a normalisation constant and $\mu_\e$ is tuned to respect the Boltzmann-Grad scaling $\mu_\e\e^{d-1} = 1$.
	
	The empirical distribution at time $t$ is defined as the average configuration of particles at time $t$: for $g$ some test function on $\mathbb{D}$,
	\begin{equation}
	\pi_t^\e(g):=\frac{1}{\mu_\e}\sum_{i= 1}^{\N} g(\gr{z}_i(t)).
	\end{equation}
	
	At equilibrium, we have the following law of large numbers. Denote
	\begin{equation}
	M(v) := (2\pi)^{-d/2}e^{-\|v\|^2/2}.
	\end{equation}

	\begin{theorem}
		For any continuous and bounded test function $g:\Lambda\times\mathbb{R}^d\to \mathbb{R}$, for all $t\in \mathbb{R}$ and for any $a>0$, 
		\begin{equation}
		\lim_{\e\to 0}\mathbb{P}_\e\left[\left|\pi_\e^t(g)-\int g(z) M(v) dz\right|\geq a\right]=0.
		\end{equation}
	\end{theorem}
	\begin{remark}
		The previous result is a simple corollary of the Lanford theorem and of the invariance of the measure (see \cite{Lanford}).
	\end{remark}
	
	\subsection{Convergence to the linearized Boltzmann equation}
	
	The aim of this article is to investigate the next order, namely the fluctuation field  
	\begin{equation}
	\zeta_\e^t(g):=\mu_\e^{1/2}\bigg(\frac{1}{\mu_\e}\sum_{1\leq i\leq \N}g(\gr{z}_i(t))-\mathbb{E}_\e[\pi_0^\e(g)]\bigg).
	\end{equation}
	
	When $\e$ tends to $0$, collisions become rare and we expect that particles can see each other only a finite number of times in any bounded time interval. We define the linearized Bolzmann operator as
	\begin{equation}\label{linearized Boltzmann}
	\mathcal{L} g(v) := \int_{\mathbb{S}^{d-1}\times\mathbb{R}^d} \big(g(v')+g(v_*')-g(v)-g(v_*)\big)((v-v_*)\cdot\eta)_+M(v_*)d\eta\,dv_*,
	\end{equation}
	where $(v',v'_*)$ are given by the scattering of $(v,v_*,\eta)$
	\begin{equation}\label{scatering}
	\left\{\begin{array}{rl}
	v' &:= v- \eta\cdot(v-v_*)\eta\\
	v_*' &:= v_*+ \eta\cdot(v-v_*)\eta.
	\end{array}\right. 
	\end{equation}
	This operator describes the variation of mass due to changes of velocity of colliding particles. The operator $\mathcal{L}$ is a self-adjoint negative operator on $L^2(M(v)dz)$. We want to prove the following result
	\begin{theorem}\label{Main theorem}
		Let $f,g\in L^2(M(v)dz)$ be two test functions. Then we have the following convergence result: for all $t\geq 0$,
		\[\mathbb{E}_\e\left[\zeta_\e^t(h)\zeta_\e^0(g)\right]\underset{\e\rightarrow0}{\longrightarrow} \left<h,e^{t(-v\cdot\nabla_x+\mathcal{L})}g\right>\]
		where $<,>$ is the Hermitian product on $L^2(M(v)dz)$.
	\end{theorem}
	
	Since the two bilinear operators
	\begin{align*}
	(h,g)\mapsto \mathbb{E}_\e\left[\zeta_\e^t(h)\zeta_\e^0(g)\right],~(h,g)\mapsto \left<h, e^{t(-v\cdot \nabla_x+\mathcal{L})}g\right>
	\end{align*}
	are both continuous on $L^2(M(v)dz)$ (see \cite{BGSS1}), it sufficient to prove Theorem \ref{Main theorem} in a dense subset. This also allows to have a quantitative version of the theorem, which we state for completeness.
	 
	We define for $g$ smooth the norm
	\begin{equation}
		\|g\|:= \sup_{(x,v)\in\mathbb{D}} \left| M^{-1}(v) g(x,v)\right|
	\end{equation}
	and we consider test functions $g$ such that \begin{equation}\label{norme forte}\|g\|+\|\,|\nabla_x g|\,\|<\infty.\end{equation}
	\begin{theorem}
		Let $g$ and $h$ two $C^1(\mathbb{D})$ functions satisfying condition \eqref{norme forte}. Then there exist three constants $C>1$, $C'>1$ and $\alpha\in(0,1)$ independent of $g,h$ such that for any $\e$ small enough, $T>1$, $\theta<\tfrac{1}{C'T^2}$ 
		\begin{equation}\label{Borne principale}
		\begin{split}
		\sup_{t\in[0,T]}\bigg|\mathbb{E}_\e\Big[\zeta^t_\e(h)\zeta^0_\e(g)\Big] -& \Big<h,e^{t(-v\cdot\nabla_x+\mathcal{L})}g\Big>\bigg|\\
		&\leq C \Big( C T^{3/2} \theta^{1/2} + (CT)^{2^{T/\theta}}\e^\alpha\Big)\|h\|\big(\|g\|+\|\nabla g\|\big).
		\end{split}
		\end{equation}
		In particular we can choose $T=o((\log|\log\e|)^{1/3})$ and $\theta = \frac{1}{\beta\log|\log\e|}$, $\beta\in(0,1)$ small enough.
	\end{theorem}

	\noindent\textbf{Notations.} From now on we will use the following notations.
	
	We denote for $m<n$ two integers, $[m,n] := \{m,m+1,\cdots,n\}$ and $[n]:=[1,n]$.
	
	For $Z_n\in\mathbb{D}^n$, and $\omega\subset[n]$, we denote 
	\[Z_{\omega}:=(z_{\omega(1)},\cdots, z_{\omega(|\omega|)})\]
	where $\omega(i)$ is the $i$-th element of $\omega$ counted in increasing order. For $1\leq l<m\leq n$, $Z_{l,m}:=Z_{[l,m]}$.
	
	Given a family  particles indices $\{i_1,\cdots,i_n\}$, the notation $(i_1,\cdots,i_n)$ indicates the ordered sequence in which $\forall k\neq l$, $i_k\neq i_l$. In addition
	\begin{itemize}
		\item $\ui_n := (i_1,\cdots,i_n)$,
		\item for $m\leq n$, $\ui_m=(i_1,\cdots,i_m)$, and more generally for $\omega \subset [1,n]$, $\ui_\omega := (i_{\min \omega}, \cdots,i_{\max \omega})$,
		\item for $0\leq m< n$ and $(i_1,\cdots,i_m)$, $\underset{(i_{m+1},\cdots,i_n)}{\sum}$ denotes the sum over every family such that for $k<l\leq n$, $i_k\neq i_l$,
		\item $\gr{Z}_{\ui_n}:= (\gr{z}_{i_1},\cdots,\gr{z}_{i_n})$, as ordered sequence.
	\end{itemize}

	We also precise the sense of Landau notation: $A=B+O(D)$ means that there exists a constant $C$ depending only on the dimension such that $|A-B|< C \,D$.
	
	Finally let $h_n$ be a function on $\mathbb{D}^n$. We denote
	\[\mathbb{E}_\e\big[h_n\big]:=\mathbb{E}_\e\Bigg[\frac{1}{\mu_\e^n}\sum_{(i_1,\cdots,i_n)}h_n\big(\gr{Z}_{\ui}\big)\Bigg]\]
	and the associated truncated function defined on $\mathcal{D}_\e$
	\[\hat{h}_n(\gr{Z}_{\N}):=\frac{1}{\mu_\e^n}\sum_{(i_1,\cdots,i_n)}h_n\big(\gr{Z}_{\ui_n}\big)-\mathbb{E}_\e\big[h_n\big].\]

	\subsection{Strategy of the proof}
	We explain now the main ideas of the proof and of the improvement with respect to \cite{BGSS1}. 
	
	Because $\zeta_\e^t(g)$ is a centered random variable, 
	\begin{equation}
	\mathbb{E}_\e\left[\zeta_\e^t(h)\zeta_\e^0(g)\right] =\mathbb{E}_\e\left[\mu_\e^{-1/2}\sum_{i=1}^\N h(\gr{z}_i(t))\,\zeta_\e^0(g)\right].
	\end{equation}
	
	The first step is to find a family of functionals $\Phi_{\e,n}^t:L^\infty(\mathbb{D})\to L^\infty(\mathbb{D}^n)$ corresponding to the pullback of the test function $h$ at time $0$
	\begin{equation}\label{pull-back de h}
	\mathbb{E}_\e\left[\mu_\e^{-1/2}\sum_{i=1}^\N h(\gr{z}_i(t))\,\zeta_\e^0(g)\right] = \sum_{n\geq 1}\mathbb{E}_\e\left[\mu_\e^{-1/2}\sum_{(i_1,\cdots,i_n)}\Phi_{\e,n}^t[h]\big(\gr{Z}_{\ui_n}(0)\big)\,\zeta_\e^0(g)\right].
	\end{equation}
	 It turns out that the $\Phi_{\e,n}^t[h]$ are a sum over \emph{histories}. Loosely speaking, a history is defined as a way to remove (or not) particles at each collision, so that at time $t$ there remains only one particle (see the picture below). Then 
	 \begin{equation} 
	 \Phi_{\e,n}^t[h](Z_n) := \frac{1}{n!}\sum_{\text{history}} h(z_k(t))\ind_{\text{history}}\sigma_{\text{history}}
	 \end{equation}
	 where $\sigma_{\text{history}}=\pm1$ and $z_k(t)$ is the position of the last particle $k$ ($k$ depends of the history). This formula will be explained precisely in section \ref{development along pseudotrajectories and time sampling}. For the moment we mention that the signs $\sigma_{\text{history}}$ are related to a splitting of the collision operators  in positive and negative part (as in \eqref{linearized Boltzmann}).
	  \begin{figure}[h!]
	 	\centering\includegraphics[scale=0.2]{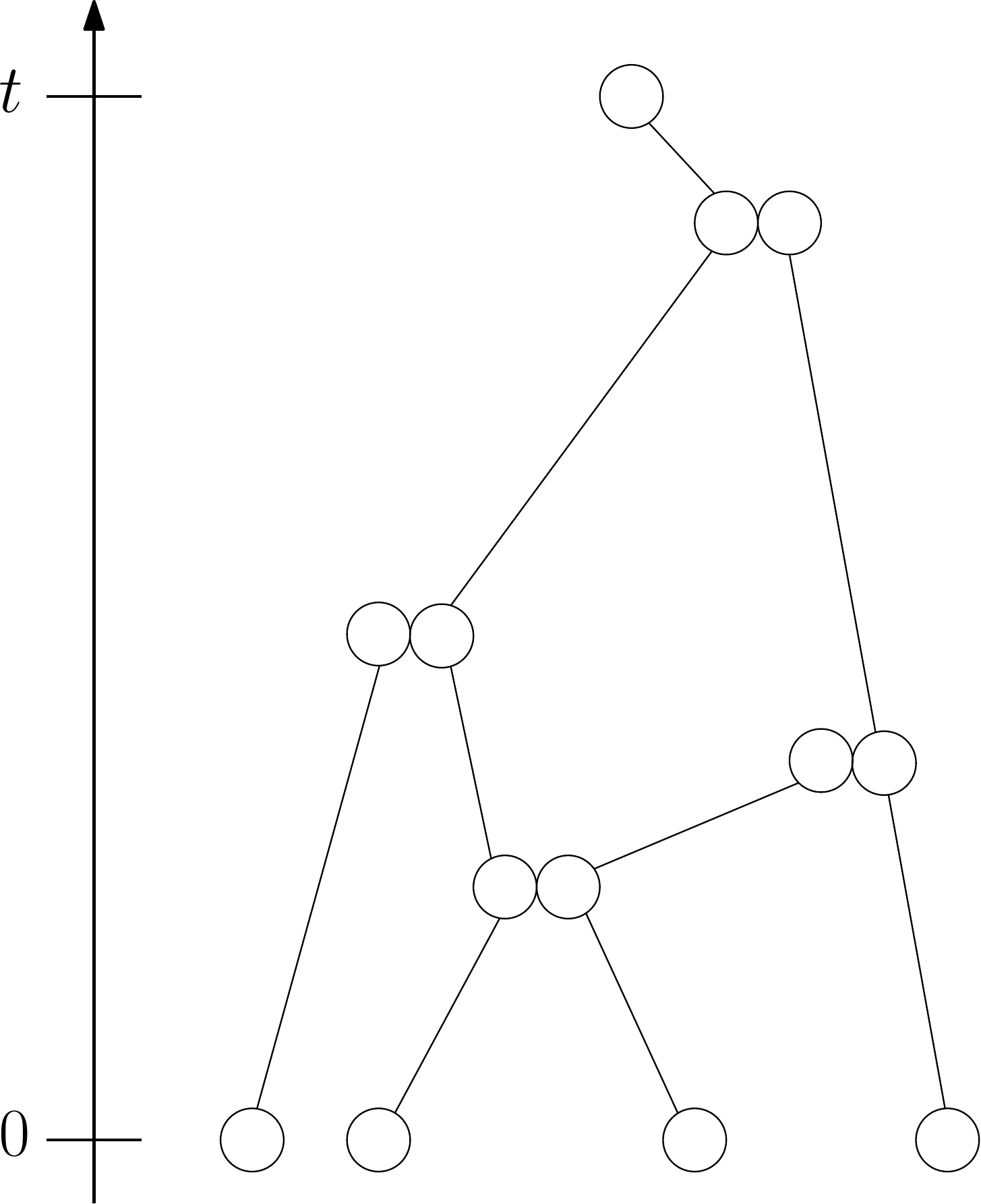}
	 	\caption{Exemple of history for four particles.}
	 \end{figure}

	The classical method to prove convergence of a hard sphere system to the Boltzmann equation (and here to the linearized equation) amounts to show that each term of the sum \eqref{pull-back de h} converges to its formal limit. This is the the way we compare the hard sphere process with the limit punctual process. In this procedure, one naturally separates a principal part containing a controlled number of collisions, from some rest terms encoding ill-behaved trajectories (for instance trajectories with more than $n-1$ collisions, which do not have a counterpart in the limit process).
	
	For the argument to be rigorous, we then need a bound on the rest terms of the sum. In usual derivations of the Boltzmann equation (see for instance \cite{Lanford,king_bbgky_1975,GTS,BGS2,PSS}) one resorts to $L^\infty$ bounds (and to a dual representation of the sum \eqref{pull-back de h}). In contrast  here we rely on the above pullback formula, together with suitable stopping times $t_{\text{s}}$ truncating the formula when the number of histories becomes uncontrolled. To implement this idea it is convenient to consider $L^2$ bounds as in \cite{BGS,BGSS1}. Indeed using notation introduced at the end of the previous section, because $\zeta_\e^0(g)$ is centered
	\begin{equation*}\begin{split}
	\left|\mathbb{E}_\e\left[\sum_{\ui_n}\Phi_{\e,n}^{t-t_{\text{stop}}}[h](\gr{Z}_{\ui_n}(t_{\text{stop}}))\zeta_\e^0(g)\right]\right|&\leq \mathbb{E}_\e\left[
\left(\widehat{\Phi_{\e,n}^{t-t_{\text{s}}}[h]}(\gr{Z}_{\ui_n}(t_{\text{s}}))\right)^2\right]^{\frac{1}{2}}\mathbb{E}_\e\left[\zeta_\e^0(g)^2\right]^{\frac{1}{2}}\\
	&\leq \mathbb{E}_\e\left[\left(\widehat{\Phi_{\e,n}^{t-t_{\text{s}}}[h]}(\gr{Z}_{\ui_n}(0))\right)^2\right]^{\frac{1}{2}}\mathbb{E}_\e\left[\zeta_\e^0(g)^2\right]^{\frac{1}{2}}
	\end{split}\end{equation*}
	using Cauchy-Schwarz and the invariance of the Gibbs measure. By virtue of such estimates, we do not need to take into account what happens for pathological histories  before $t_{\text{stop}}$. 

	 Unfortunately, in the bound for $\Phi_{\e,n}^t[h]$, we do not know how to take into account the cancellations due to the signs in $\sigma_{\text{history}}$. Thus we have to count the number of possible histories and collisions. We then need to distinguish two kinds of collisions: those where one particle is removed, and those where both particles are kept, called \emph{recollisions}. The second type is harder to control. 
	
	We need two different samplings to control each type of collision separately. The first sampling has a relative large step $\theta = \beta\log|\log\e|$ (with $\beta\in(0,1)$ set later) and enables to control a moderate growth of collisions with removal. The method (already used in \cite{BGSS1}) is an adaptation of \cite{BGS2,BGS} (and reminiscent of \cite{Erdos} in the context of the quantum Lorentz gas). This will be the source of the slow speed of convergence in \eqref{Borne principale}.
	
	The second sampling, which has a shorter step $\delta = \e^{\beta'}$ (with $\beta'\in(0,1)$ set later) is used to control possibly many recollisions on the short time scale. These collisions will be allowed only on the last time interval $[t_{\text{stop}},t_{\text{stop}}+\delta]$. In the present paper, two conditionings on initial data are used. The first one is symmetric on all the particles and forbids a group of more than a fixed integer $\gamma>0$ to interact altogether on each small time interval $[k\delta,(k+1)\delta]$ (for $k\in\mathbb{N}$). At this point, the paper \cite{BGSS1} uses the  billiard theory developed in \cite{BFK} to control the histories in clusters of  $\gamma$ particles. Notice that such result has no known analogue for other interaction potentials, even with compact support.
	
	The main goal of this paper is to avoid this geometrical argument. We defined the collision graph of a trajectory on a time interval $[\tau,\tau']$ as the graph where the vertices are the set of particles and to each collision happening on $[\tau,\tau']$ correspond an edges between the colliding particles. A trajectory on the time interval $[t_{\text{stop}},t]$ is said \emph{non-pathological} if
	\begin{itemize}
		\item its collision graph restricted to $[t_{\text{stop}}+\delta,t]$ is a tree (at each collision, one particle is removed),
		\item on $[t_{\text{stop}},t_{\text{stop}}+\delta]$ the collision graph has no cycle (but there can be recollisions).
	\end{itemize}
	Due to the symmetric conditioning, one particle can meet at most $\gamma$ other particles on $[t_{\text{stop}},t_{\text{stop}}+\delta]$, and thus there are at most $\gamma$ recollisions per particle. 
	Therefore the number of non-pathological trajectories is controlled by construction.
	
	\begin{figure}[h]
		\centering\includegraphics[scale=0.2]{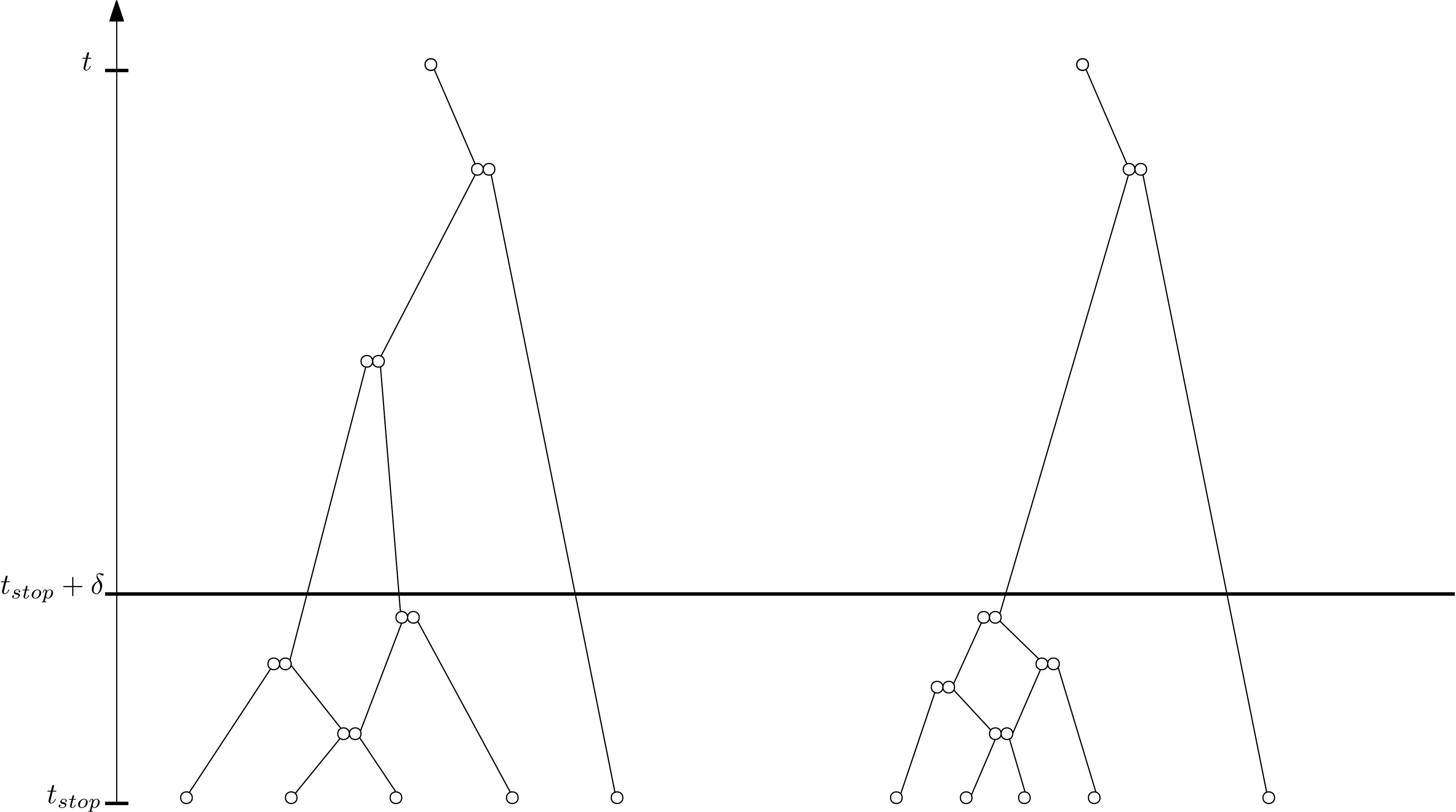}
		\caption{An example of one non-pathological pseudotrajectory (on the left) and a pathological one (on the right).}
	\end{figure}
	
	We therefore introduce a second conditioning forbidding pathological trajectories. One difficulty is that this conditioning will introduce asymmetry. Indeed, since there are approximately $\mu_\e$ particles in the system, choosing one particle costs roughly $\mu_\e$, and in a  symmetric conditioning the choice of $k$ particles would cost $\mu_\e^k$. However in the sum $\sum_{\ui_n}\Phi_n[h](\gr{Z}_{\ui_n}),$
	there are already $n$ chosen particles. We are interested in particles of the background which can influence these $n$ particles. Hence, it is sufficient to impose an asymmetric conditioning where one of the $k$ particles is chosen in $\ui_n$ and at most $k-1$ new particles have to be chosen. Such procedure  provides a gain of $\mu_\e^{-1}$ which turns out to be enough to control the error term, by means of a cumulant expansion.
	
	We conclude by describing the asymmetric conditioning. Let $\chi(Z_r)$ be the indicator function which takes value $1$ if there exist history parameters such that the graph with initial data $Z_r$ at time $t_{\text{stop}}$ has a cycle and is connected. Because the indicator function involves a bounded number of particles, its  weight $\|\chi(Z_r)\|_{L^1}$ is small. We then introduce an asymmetric conditioning $\mathcal{X}_{\ui_n}(\gr{Z}_{\N}(t_{\text{stop}}))$ imposing the existence of a set of particles $\omega$ containing at least one particle of $\{i_1,\cdots,i_n\}$ such that $\chi(\gr{Z}_{\omega}(t_{\text{stop}})$ is equal to $1$, \textit{id est} one trajectory containing a particle of $\ui_n$ is pathological. 
	
	Let us give an idea on how to bound $\mathcal{X}_{\ui_n}(\gr{Z}_{\N}(t_{\text{stop}}))$. We develop the constraint in cumulants over finite numbers of variables
	\[\mathcal{X}_{\ui_n}(\gr{Z}_{\N}(t_{\text{stop}})) =\sum_{p\geq n}\sum_{(i_{n+1},\cdots,i_{p})} \mathfrak{X}_{n,p}\big(\gr{Z}_{\ui_p}(t_{\text{stop}})\big).\]
	By definition, the $\mathfrak{X}_{n,p}\big(\gr{Z}_{\ui_p}(t_{\text{stop}})\big)$ are sums over families of particles $(\omega_1,\cdots,\omega_k)$, where $\omega_i$ is a subset of $\ui_p$, of terms \[(-\chi(\gr{Z}_{\omega_1}(t_{\text{stop}})))(-\chi(\gr{Z}_{\omega_2}(t_{\text{stop}})))\cdots(-\chi(\gr{Z}_{\omega_r}(t_{\text{stop}}))).\]
	The $\omega_i$ can intersect, hence the number of terms in $\mathfrak{X}_{n,p}$ is huge. But the (first) symmetric conditioning permits to bound the number of intersecting sets. If $\omega_1,\cdots,\omega_k$ intersect, all the particles in their union are close. Hence the size of $\omega_1\cup\cdots\cup\omega_k$ is bounded by $\gamma$ and $k$ is smaller than $2^\gamma$. This is sufficient to bound $\mathfrak{X}_{n,p}$.
	
	The paper is organized as follows. In section \ref{development along pseudotrajectories and time sampling} we give a proper definition of history and we use it to construct the functionals $\Phi_{\e,n}^t$. Then the two samplings mentioned above are constructed. This allows to decompose $\mathbb{E}_\e\left[\zeta_\e^t(h)\zeta_\e^0(g)\right]$ in a main term plus error terms of different nature: one is a development on trajectories without recollisions (bounded in Section \ref{Clustering estimations}), a second is a development on non pathological trajectories (actually called below pseudotrajectories) with recollisions (bounded in section \ref{Estimation of the long range recollisions.}) and the last part deals with pathological recollisions (bounded in Section \ref{Estimation of the long range recollisions.}). These estimations need standard $L^2(\mathbb{P}_\e)$ estimates based on static cumulant decompositions. They are given in Section \ref{Quasi-orthogonality estimates}. Finally, the convergence of the main term is proved in Section \ref{Treatment of the main part}.
	
	\section{Development along pseudotrajectories and time sampling}\label{development along pseudotrajectories and time sampling}
	
	\subsection{Definition of (forward) pseudotrajectories}
	Consider $n$ particles. To lighten notation for trajectories we will drop the dependence on $\e$.

	For  $m\leq n$, fix a family of \emph{pseudotrajectory parameters} (which was called history in the introduction)
	\[((s_i,\bar{s}_i)_{1\leq i\leq n-m},(\kappa_j)_{1\leq j \leq n})\in \{(\pm1,\pm1)\}^{n-m}\times\mathbb{N}^n,\]
	and an initial data $Z_n\in\mathcal{D}^n_\e$.
	
	We construct iteratively the pseudotrajectories $\Zt_n(\tau,((s_i,\bar{s}_i)_{1\leq i\leq n-m},(\kappa_j)_{1\leq j \leq n}),Z_n)$, the collision indices $\iota(\tau)$ and recollision indices $(\kappa_j(\tau))_{1\leq j \leq n}$. At time $\tau=0$, we set $\iota(0):=1$ and for all $j$, $\kappa_j(0):=\kappa_j$. Moreover at $\tau = 0$, $\Zt_n(0) = Z_n\in\mathcal{D}^n_\e$. The number of particles decreases with time and is equal to $(n+1-\iota(\tau))$.
	
	If $\iota(\tau)<n-m+1$, the remaining particles move along freely until there is a new collision between two of them at time $\tau$ (say $q$ and $q'$ with $q<q'$). If $s_{\iota(\tau^-)}=1$ (respectively $-1$) we look at $\kappa_q(\tau^-)$ (respectively $\kappa_{q'}(\tau^-)$):
	\begin{itemize}
		\item if it is strictly positive, we have a \emph{recollision}. The two particles scatter as in \eqref{scattering} and $\kappa_{q}(\tau^+) = \kappa_q(\tau)-1$ (respectively $\kappa_{q'}(\tau^+) = \kappa_{q'}(\tau)-1$),
		\item if it is $0$ we have an \emph{annihilation}: we remove the particle $q$ (and in the case where $s_{\iota(\tau^-)}=-1$ remove $q'$). The other particle scatters if $\bar{s}_{\iota(\tau^-)} =1$ or  continue in straight line else. Finally we increment $\iota(\tau)$.
	\end{itemize}
	When $\iota(\tau)=n-m+1$ (there are $m$ particles left), all the annihilations have been performed and particles evolve along the Hamiltonian flow.
	
	\begin{figure}[h]
		\centering
		\includegraphics[scale=0.35]{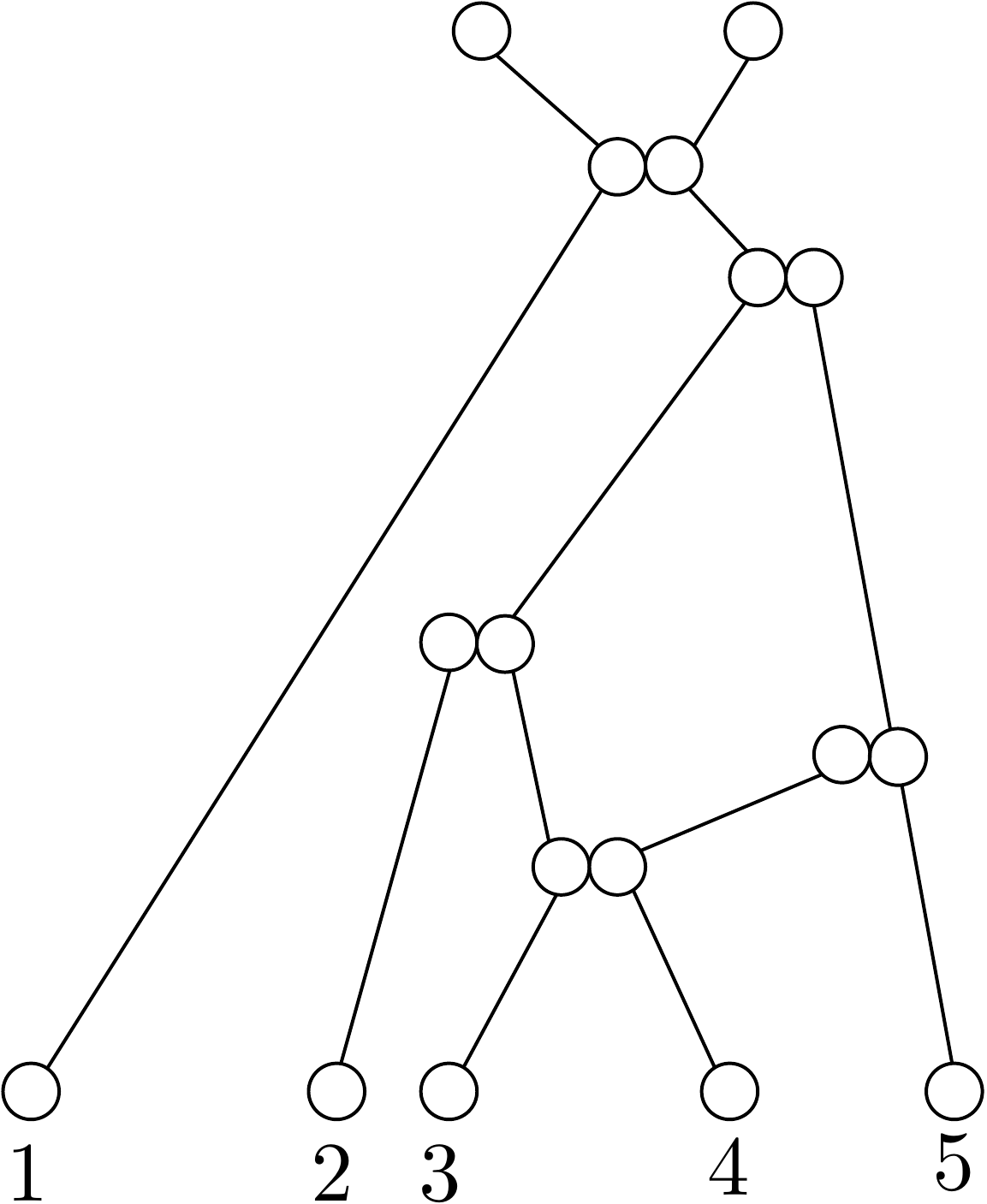}
		\caption{Pseudotractory associated with $(((-1,-1),(-1,1),(1,1)),(0,0,1,0,0))$.}\label{construction-pseudo}
	\end{figure}
	
	Let $\omega$ be a finite subset of $\mathbb{N}^*$. We will denote $\Zt_\omega(t,Z_\omega,((s_i,\bar{s}_i)_{i\leq |\omega|-m},(\kappa_j)_{j\in\omega}))$ the pseudotrajectory with particles of $\omega$ and $\Zt_\omega(t)$ when there is no ambiguity on the parameters. Note that $\gr{Z}_\omega(t)$ is the configuration of the particles $\omega$ in the dynamics of $\mathcal{D}_\e$ (the real trajectories).
	
	\begin{definition}[Collision graph]
		For $Z_r\in\mathcal{D}^r_\e$  and parameters $((s_i,\bar{s}_i),(\kappa_j)_j)$, we contruct the \emph{collision graph} $\mathcal{G}_r^{[0,t]}$ as the couple $(E,V)$, with $V := \{1,\cdots,r\}$ and
		\[E\subset \{(i,j)_\tau,\,\text{where}~(i,j)\in V^2,\, i<j,\,\tau\in[0,t]\}\]
		such that $(i,j)_\tau\in E$ if and only if there is a collision at time $\tau$ in the pseudotrajectory between particle $i$ and $j$. By standard properties of the hard sphere dynamics (see \cite{Alexander}), for almost all $Z_r$, $\mathcal{G}_r^{[0,t]}$ has a finite number of edges. We can order $\tau_1<\tau_2<\cdots<\tau_k$ the collision times of $G$.
		
		In the following we denote $E\left(\mathcal{G}_r^{[0,t]}\right):=E$. 
	\end{definition}
	
	\begin{figure}[h]
		\centering\includegraphics[scale=0.2]{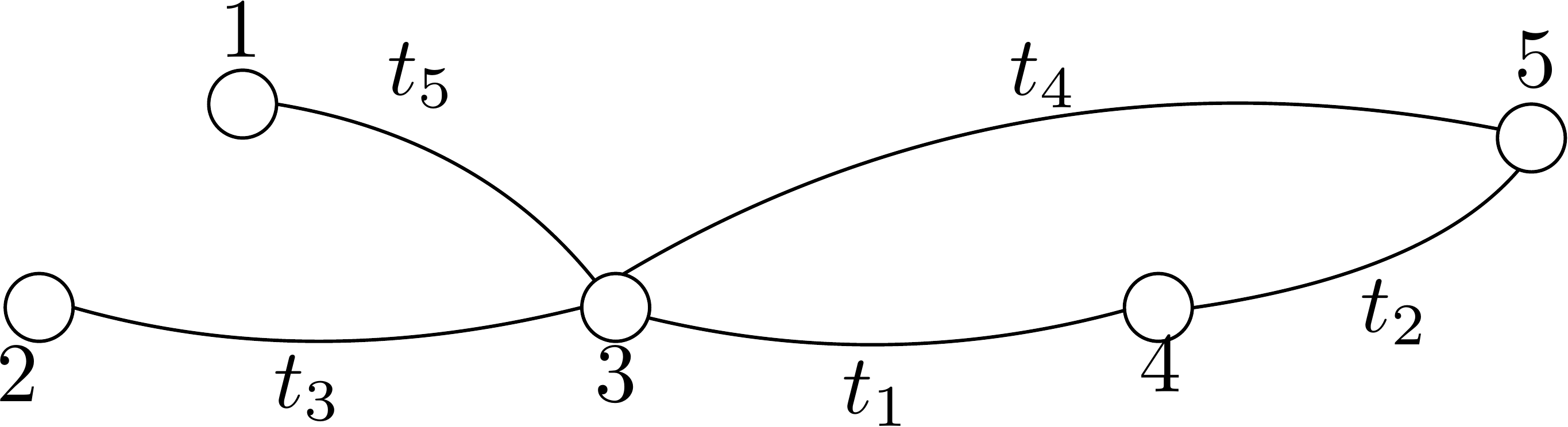}
		\caption{Collision graph associated with pseudotrajectory of figure \ref{construction-pseudo}, with $t_1<t_2<t_3<t_4<t_5$ the collision times.}
	\end{figure}
	\subsection{Developement along pseudotrajectories}
	The pseutrajectories are used to pull back a function evaluated at time $t$ to a previous time $0$.
	
	Let $m\leq n$ be two integers, and $((s_i,\bar{s}_i)_{1\leq i\leq n-m}),(\kappa_j)_{1\leq j\leq n})$ be collision parameters and $t>0$ the finite time. In order to not count twice the same pseudotrajectory, all parameters have to be take in to account. We define $\mathcal{R}_{((s_i,\bar{s}_i),(\kappa_j))}^{m\leftarrow n,t} \subset \mathcal{D}^n_\e$  as the set of initial parameters $Z_m$ such that at time $t$ the remaining parameter terms in the pseudory $\Zt_n(\tau,)$ verifies the following condition:
	$\{1,\cdots,m\}$ are the remaining particles of $\Zt_n(t,Z_n)$, and the recollision indices define in the previous section vanish at time $t$: for all $j$, $\kappa_j(t)$ vanishes.
	
	Let $h_m \in L^\infty(\mathbb{D}^m)$ some test function (not supposed symmetric under permutation of its parameters). We define the \emph{pseudotrajectory development} as the functional $\Phi^t_{m\leftarrow n} :L^\infty(\mathbb{D}^m)\to L^\infty(\mathbb{D}^n)$ with
	\begin{equation}
	\Phi^t_{m\leftarrow n}[h_m](Z_n) :=\frac{1}{(n-m)!} \sum_{\substack{(s_i,\bar{s}_i)_{i\leq n-m}\\(\kappa_j)_j}} \left(\prod_{i= 1}^{n-m} \bar{s}_i\right) \ind_{\mathcal{R}_{((s_i,\bar{s}_i),(\kappa_j))}^{m\leftarrow n,t}}h_m(\Zt_n(t,Z_n)).
	\end{equation}
	
	We have the following semigroup property:
	\begin{prop}
		 Consider $m \leq n$ two integers, $t>t'>0$ two evaluation times and $\ui_m$ a family of particles. Then for any function $h_m \in L^{\infty}(\mathbb{D}^m)$ and almost all initial data $Z_n\in\mathcal{D}^n_\e$, 
		 \begin{equation}
		 \sum_{(i_{m+1},\cdots i_n)}\Phi_{m\leftarrow n}^t[h_m](Z_n)=\sum_{n'=m}^n \sum_{(i_{m+1},\cdots,i_{n})} \Phi_{n'\leftarrow n}^{t'}\left[\Phi^{t-t'}_{m\leftarrow n'}[h_m]\right](Z_n).
		 \end{equation}
	\end{prop}
	\begin{proof}
		Fix collision parameters $((s_i,\bar{s}_i)_{i\leq n_m},(\kappa_j)_{j\leq n})$ and an initial data. In pseudotrajectory $\Zt_n(\tau,((s_i,\bar{s}_i),(\kappa_j)),Z_n)$, we consider $\omega$ the set of the remaining particles at time $t'$ and $t_l$ the time of the last annihilation before $t'$. We construct two set of collision parameters
		\[((s'_i,\bar{s}'_i)_{i\leq n-|\omega|},(\kappa'_j)_{j\leq n}):=((s_i,\bar{s}_i)_{i\leq n-|\omega|},(\kappa_j-\kappa_j(t_a))_{j\leq n}),\]
		\[((s''_i,\bar{s}''_i)_{i\leq |\omega|-m},(\kappa''_j)_{j\in \omega}):=((s_i,\bar{s}_i)_{i> |\omega|-m},(\kappa_j-\kappa_j(t'))_{j\in\omega}).\]
		
		We first prove the equality 
		\begin{multline*}\Zt_n\left(t,((s_i,\bar{s}_i)_{n-m},(\kappa_j))_n,Z_n\right)\\
		=\Zt_{|\omega|}\left(t-t',((s''_i,\bar{s}''_i)_{|\omega|-m},(\kappa''_j)_{\omega}),\Zt_n\left(t',((s'_i,\bar{s}'_i)_{ n-|\omega|},(\kappa'_j)_n),Z_n\right)\right).\end{multline*}
		Until time $t_l$ we have
		\[\Zt_n(\tau,((s_i,\bar{s}_i)_{n},(\kappa_j)_n),Z_n)=\Zt_n(\tau,((s'_i,\bar{s}'_i)_{ n-|\omega|},(\kappa'_j)_n),Z_n).\]
		Then on the time interval $[t_l,t']$ there is only recollisions (no annihilation) in the two pseudotrajectories. In the first one they are threaten by the parameters $\kappa_j(\tau)$ which decrease. In the second one because we have threaten all the annihilation, particles evolves along the Hamiltionian flow. Note that after time $t_l$, the $\kappa'_j(\tau)$ vanish. Hence at time $t'$
		\[\Zt_n(t',((s_i,\bar{s}_i),(\kappa_j)),Z_n)=\Zt_n(t',((s'_i,\bar{s}'_i),(\kappa'_j)),Z_n).\]
		and denoting $\iota(\tau)$ the collision indices associated with the first pseudotrajectory
		\[\iota(t') = n-|\omega|,~\forall j,~\kappa_j(t') = \kappa_j-\kappa_j(t').\]
		We can pursue the two pseudotrajectory constructions and finally 
		\[\Zt_n(t,((s_i,\bar{s}_i),(\kappa_j)),Z_n)=\Zt_{\omega}(t-t',((s''_i,\bar{s}''_i),(\kappa''_j)),\Zt_n(t',((s'_i,\bar{s}'_i),(\kappa'_j)),Z_n))\]
		which will be written shorter as
		\[\Zt_n(t,Z_n)=\Zt_{\omega}(t-t',\Zt_n(t',Z_n)).\]
		
		For each initial data, we have constructed an onto map 
		\[((s_i,\bar{s}_i)_{n},(\kappa_j)_n)\mapsto (\omega,((s'_i,\bar{s}'_i)_{ n-|\omega|},(\kappa'_j)_n),((s''_i,\bar{s}''_i)_{|\omega|-m},(\kappa''_j)_{|\omega|})\]
		with in addition 
		\[\prod_{i=1}^{n-m}s_i=\prod_{i=1}^{n-|\omega|}\bar{s}'_i\prod_{i=1}^{|\omega|-m}\bar{s}_i''.\]
		Hence denoting $\mathcal{R}_{((s'_i,\bar{s}'_i),(\kappa'_j))}^{\omega\leftarrow n,t}$ the set of initial data such that the set of remaining particles of $\Zt_n(t',((s'_i,\bar{s}'_i),(\kappa'_j)),Z_n)$ is $\omega$, and the corresponding recollision parameters $\kappa_i'(t')$ vanish, we have
		\[\begin{split}\sum_{\substack{(s_i,\bar{s}_i)_{i\leq n-m}\\(\kappa_j)_j}}  \prod_{i=1}^{n-m}s_i\, \ind_{\mathcal{R}_{((s_i,\bar{s}_i),(\kappa_j))}^{m\leftarrow n,t,\omega}}h_m(\Zt_n(t)) = \sum_{[m]\subset\omega\subset[n]}\sum_{\substack{(s'_i,\bar{s}'_i)_{i\leq n-|\omega|}\\(\kappa''_j)_{j\leq n}}} \prod_{i=1}^{n-|\omega|}\bar{s}'_i\, \ind_{\mathcal{R}_{((s'_i,\bar{s}'_i),(\kappa'_j))}^{\omega\leftarrow n,t}}\\
		\times\sum_{\substack{(s''_i,\bar{s}''_i)_{i\leq |\omega|-m}\\(\kappa''_j)_{j\leq|\omega|}}}\prod_{i=1}^{|\omega|-m}\bar{s}_i''\, \ind_{\mathcal{R}_{((s''_i,\bar{s}''_i),(\kappa''_j))}^{m\leftarrow \omega,t}}\big(\Zt_n(t')\big)h_m\big(\Zt_{\omega}\big((t-t',\Zt_n(t'))\big)\big).
		\end{split}\]
		We prove that for $Z_n\in\mathcal{D}^n_\e$,
		\[\begin{split}(n-m)!&\Phi_{m\leftarrow n}^t[h_m](Z_n)\\
		&=\sum_{n'=m}^n ~\sum_{\substack{\omega\subset[m+1,m+n]\\|\omega|=m-n'}} (n-n')!\Phi_{n'\leftarrow n}^{t'}\left[(n'-m)!\Phi^{t-t'}_{m\leftarrow n'}[h_m]\right](Z_m,Z_\omega,Z_{\omega^c}).\end{split}\]
		
		Then summing on all family of particles
		\[\begin{split}
		\sum_{(i_{m+1},\cdots i_n)}&\Phi_{m\leftarrow n}^t[h_m](\gr{Z}_{\ui_n}(0))\\
		&=\sum_{(i_{m+1},\cdots i_n)}\sum_{n'=m}^n\frac{1}{\binom{n-m}{n'-m}} ~\sum_{\substack{\omega\subset[m+1,m+n]\\|\omega|=m-n'}} \Phi_{n'\leftarrow n}^{t'}\left[\Phi^{t-t'}_{n\leftarrow n'}[h_m]\right]\big((\gr{Z}_{\ui_m},\gr{Z}_{\ui_\omega},\gr{Z}_{\ui_{\omega^c}})(0)\big)\\
		&=\sum_{(i_{m+1},\cdots,i_{n})}\sum_{n'=m}^n  \Phi_{n'\leftarrow n}^{t'}\left[\Phi^{t-t'}_{n\leftarrow n'}[h_m]\right](\gr{Z}_{\ui_n}(0)).
		\end{split}\]
	\end{proof}
	
	Now we can write the pseudotrajectory development of the marginal (see \cite{SP}):	
	\begin{theorem}
		Let $(i_1,\cdots,i_m)$ be a family of particles,with $i_{\max} := \max\{i_1,\cdots,i_n\}$. For almost all $\gr{Z}_\N\in\mathcal{D}_\e\cup\{\N\geq i_{\max}\}$ we have
		\begin{equation}
		h_m\left((\gr{Z}_{(i_1,\cdots,i_m)}(t)\right) = \sum_{n\geq m} \sum_{(i_{n+1},\cdots,i_{m})} \Phi^t_{m\leftarrow n}[h_m](\gr{Z}_\ui(0)).
		\end{equation}
		
		In addition if we do not fix $(i_1,\cdots,i_m)$ we have
		\begin{equation}
		\sum_{\ui_m}h_m\left((\gr{Z}_{\ui_m}(t)\right) = \sum_{n\geq m} \sum_{\ui_m} \Phi^t_{m\leftarrow n}[h_m](\gr{Z}_{\ui_n}(0)).
		\end{equation}
	\end{theorem}
	\begin{proof}
		The proof is an adaptation of \cite{SP}.
		
		We rewrite semi-group equality: For $m\leq n_1 \leq n$, $t>t'>0$,
		\[\begin{split}
		\sum_{(i_{m+1},\cdots i_n)}&\Phi_{m\leftarrow n}^t[h_m](\gr{Z}_{\ui_n}(0))=\sum_{n'=m}^n \sum_{(i_{m+1},\cdots,i_{n'})} \sum_{(i_{n'+1},\cdots,i_n)} \Phi_{n'\leftarrow n}^{t'}\left[\Phi^{t-t'}_{n\leftarrow n'}[h_m]\right](\gr{Z}_{\ui_n}(0)).
		\end{split}\]
		
		Thanks to Alexender's proof of wellposedness of the the hard sphere dynamic (see \cite{Alexander}), outside a set of zero measure the number of collision in finite on a finite interval. Hence $[0,t]$ can be cut into small time intervals  time interval $[t_k,t_{k+1}]$ such that on each one there is at most one collision between two particles $i$ and $j$, and if $i$ is removed there is no more collision. Using the semigroup property, one needs to prove the result only on each $[t_k,t_{k+1}]$.
		
		We fix the number of particle $\N$ and the initial configuration $\gr{Z}_\N$ and we consider a small time $t$ such that the preceding conditions are check. Let $\ui_n$ be a family of particle. We distinguish t cases. First on $[0,t]$ none of the particles in $\ui_m$ have a collision. Thus for any $n>m$ all the $\ind_{\mathcal{R}^{m\leftarrow n,t}_{((s_i,\bar{s}_i)_i,(\kappa_j)_j)}}(\gr{Z}_{\ui_n})$ vanish and
		\[h_m\left((\gr{Z}_{\ui_m}(t)\right) =\Phi^t_{m\leftarrow m}[h_m](\gr{Z}_{\ui_n}(0)).\]
		In the same way if the collision occurs between two particles of $\ui_m$, the $\ind_{\mathcal{R}^{m\leftarrow n,t}_{((s_i,\bar{s}_i)_i,(\kappa_j)_j)}}(\gr{Z}_{\ui_n})$ also vanish and the same equality holds. The last case is when the collision happens between one particle of $\ui_m$ and an other particle $i_{n+1}$. Up to a permutation of the indices, the collision happens between $i_1$ and $i_{m+1}$. Removing all the vanishing terms,
		\[\begin{split}\sum_{n\geq m} \sum_{(i_{n+1},\cdots,i_{m})} \Phi^t_{m\leftarrow n}[h_m](\gr{Z}_{\ui_m}) =&h_m\Big(\Zt_m(t,\big(t,\big(()_0,(0)_{m})\big)\gr{Z}_{\ui_m})\Big)\\
		&-h_m\Big(\Zt_{m+1}\big(t,\big(((1,-1)),(0)_{m+1}\big),\gr{Z}_{\ui_{m+1}}\big)\Big)\\
		&+h_m\Big(\Zt_{m+1}\big(t,\big(((1,1)),(0)_{m+1}\big),\gr{Z}_{\ui_{m+1}}\big)\Big).
		\end{split}\]
		In the two first terms particles move along straight lines because there is no scattering at the collision. Thus these two terms compensate. For the first one note that because particles in $\ui_m$ are deviated at each of their collisions,
		\[\Zt_{m+1}\big(t,\big(((1,1)),(0)_{m+1}\big),\gr{Z}_{\ui_{m+1}}\big)=\gr{Z}_{\ui_{m}}(t)\]
		and 
		\[\begin{split}h_m(\gr{Z}_{\ui_m})(t) =h_m\Big(\Zt_m(t,\big(t,\big((),(0)_{m})\big)\gr{Z}_{\ui_m})\Big)-h_m\Big(\Zt_{m+1}\big(t,\big(((1,-1)),(0)_{m+1}\big),\gr{Z}_{\ui_{m+1}}\big)\Big)\\
		+h_m\Big(\Zt_{m+1}\big(t,\big(((1,1)),(0)_{m+1}\big),\gr{Z}_{\ui_{m+1}}\big)\Big)
		\end{split}\]
		which conclude the proof.
	\end{proof}
	
	Applied to the covariance, it gives
	\begin{equation*}
	\mathbb{E}_\e\left[\zeta_\e^t(h)\zeta^0_\e(g)\right]=\sum_{n\geq 1}\frac{1}{{\mu_\e}^{\frac{1}{2}}}\mathbb{E}_\e\left[\sum_{(i_1,\cdots,i_n)}\Phi^t_{1\leftarrow n}[h](\gr{Z}_{\ui}(0))\zeta_\e^0(g)\right].
	\end{equation*}
	
	Because we do not know how take account of the vanishing due to the $\bar{s}_i$, we bound $\Phi^t_{1\leftarrow n}[h]$ by counting the number of collision parameter. The problem is that there no \textit{a priori} bound on the number of recollision parameters needen. To overcome this difficulty we have to introduce some conditioning in order to bound the number of recollision.
	
	\subsection{Conditioning}\label{Conditioning}
	We need two conditioning on the data.
	
	\paragraph{The first one is a symmetric conditioning on all the parameters.}
	
	\begin{definition}[Distance cluster]
		Let $L$ be a positive real number and $Z_n\in\mathbb{D}^n$ be a particle configuration. We consider the undirected graph of vertices $\{1,\cdots,n\}$ and of edges 
		\[\{(i,j)\in[1,m]^2,~d(x_i,x_j)<L\}.\]
		A \emph{$L$-distance cluster} is one of its connected component. In the further we only look at $2\delta\mathbb{V}$-distance distance cluster so we drop the "$2\delta\mathbb{V}$".
	\end{definition}

	Let $\gamma>0$ be an integer depending only on the dimension, a time scale $\delta$ (which will be a power of $\e$) and $\mathbb{V}$ a velocity bound. We construct $\Upsilon_\e\subset\mathcal{D}_\e$ the set of particles configurations such that for any time $\tau\in\{0,\delta,2\delta,\cdots,t\}$, there is no distance cluster of size bigger than $\gamma$ at time and inside any subset of particles $\omega\subset[1,\N]$ whith less than $\gamma$ element, $\tfrac{1}{2}\|\gr{V}_\omega(\tau)\|^2$ is bounded by $\tfrac{1}{2}\mathbb{V}^2$. We have the following bound on the measure of the complement of $\Upsilon_\e$:
	\begin{prop}
		There exists a constant $C_\gamma$ depending only on $\gamma$ and on the dimension such that
		\begin{equation}
		\mathbb{P}_\e\left(\Upsilon^c_\e\right)\leq C_\gamma\frac{t}{\delta}\left(\mu_\e\delta^{-1}\left(\mu_\e\delta^d\mathbb{V}^d\right)^\gamma+\mu_\e^\gamma e^{-\mathbb{V}^2/4}\right).
		\end{equation}
	\end{prop}
	\begin{proof}
		\begin{align*}
		\mathbb{P}_\e\Big(&\Upsilon_{\e}^c\Big)\leq \sum_{k=0}^{t/\delta}\mathbb{E}_\e\left[\sum_{(i_1,\cdots,i_{\gamma+1})}\ind_{\gr{X}_{\ui_{\gamma+1}}(k\delta)\text{\,form\,a\,distance\,cluster}}+\sum_{k'= 1}^\gamma\sum_{(i_1,\cdots,i_{k'})}\ind_{\|\gr{V}_{\ui_{k'}}(k\delta)\|\geq\mathbb{V}}\right]\\
		&\leq \frac{t}{\delta}\left(\mu_\e^{\gamma+1}\int\ind_{{X}_{\gamma+1}\text{\,form\,a\,distance\,cluster}}M^{\otimes (\gamma+1)}dZ_{\gamma+1}+\sum_{k'= 1}^\gamma\mu_\e^\gamma\int\ind_{\|V_k\|\geq\mathbb{V}}M^{\otimes (\gamma+1)}dZ_{\gamma+1}\right)\\
		&\leq \frac{t}{\delta}\left(\mu_\e^{\gamma+1}\left(\gr{c}_{\gr{d}}(\gamma\delta\mathbb{V})^d\right)^\gamma+\gamma 2^{\gamma/2}\mu_\e^\gamma e^{-\frac{\mathbb{V}^2}{2}}\right)
		\end{align*}
		where $\gr{c}_{\gr{d}}$ is the volume of sphere of diameter $\gamma$. We used that Gibbs measure is time invariant and that particles $X_\gamma$ have to be at distance less than $\gamma\delta\mathbb{V}$ of $x_{\gamma+1}$ in order to form a distance cluster.
	\end{proof}
	
	Thus for $\delta:= \e^{1-\frac{1}{2d}}$, $\mathbb{V}:=|\log\e|$ and $\gamma$ large enough, $\mathbb{P}_\e(\Upsilon_{\e}^c)$ is smaller than $\e^d$.
	
	\paragraph{The second conditioning is an asymmetric conditioning.}
	
	We look only at a finite number of particle $Z_r$. For fix pseudotrajectory parameters $((s_i,\bar{s}_i)_{1\leq i\leq n-1},(\kappa_j)_{1\leq j \leq n})$, the configuration $Z_r\in\mathcal{D}^r_\e$ form a \emph{collision cluster} if the collision graph of $\Zt_r(\tau,((s_i,\bar{s}_i)_{i},(\kappa_j)_{j}),Z_r)$ on the time interval $[0,\delta]$ is connected. We define \emph{local recollision} of $\Zt_r(\tau)$ as the first collision (in time order) which create a loop in the collision graph. 
	
	We define function $\chi_r:\mathcal{D}_\e^r \mapsto \{0,1\}$ the indicator function of:
	\[\left\{Z_r\in \mathcal{D}^r_\e,~\exists ((s_i,\bar{s}_i)_{1\leq i\leq n-1},(\kappa_j)_{1\leq j \leq n}),~\Zt_r(\tau)\,\text{form\,a\,collision\,cluster\,with\,local\,recollision}\right\}.\]
	
	We have the following $L^1$ bound on $\chi_r$:
	
	\begin{prop}
	There exists a positive constant $C_r$ depending only on the number of particle such that and some $a>0$ depending only on the dimension such that
	\begin{equation}\label{estimation recollision}
	\int_{\Lambda^r\times B_r(\mathbb{V})} \chi(Z_r)M^{\otimes r}(V_r)dV_rdX_{2,r}\leq C_r\mu_\e^{-r+1}\delta^2\left(\mu_\e \delta^d \mathbb{V}^d\right)^{r-3}\e^\alpha.
	\end{equation}
	\end{prop}
	
	\begin{proof}
		First if the pseudotrajectories $\Zt_r(\tau)$ form a collision cluster for some collision parameters, the initial position need to be close enough. Because the speed of each particles is globaly bounded by $\mathbb{V}$, there exists for any couple of particle $(i,i')$ a finite sequence $i=j_1,\,j_2,\cdots,j_k=i'$ two by two distinct with \[|x_{j_l}-x_{j_{l+1}}|\leq 2\mathbb{V}\delta.\] Thus the distance between two different particles of $Z_r$ is bounded by $2r\mathbb{V}\delta$. We need a more precise geometric conditioning.
		
		Let $Z_r\in \mathcal{D}_\e^r$ such that $\chi(Z_r)$ is non zero. Then there exists a set of pseudotrajectory parameters $((s_i,\bar{s}_i)_i,(\kappa_j)_j)$ such that the pseudotrajectory $\Zt_r(\tau,((s_i,\bar{s}_i)_i,(\kappa_j)_j))$ has a local recollision. We define $\ts$ the time of the first local recollision. We construct an other set of recollision parameter:
		\[\kappa'_j = \kappa_j(0)-\kappa_j(t).\]
		Then for any $\tau\in[0,\ts]$,
		\[\Zt_r(\tau,((s_i,\bar{s}_i)_i,(\kappa_j)_j)) = \Zt_r(\tau,((s_i,\bar{s}_i)_i,(\kappa'_j)_j))\]
		and for all $j$, $\kappa'_j(\ts)=0$. Secondly on $[0,\ts)$, the pseudotrajectory have no local recollision. Thus $\kappa'_j$ is lower than $r-1$. Indeed after $r$ collision a particles have meet an other one twice. 

		Let $\varpi\subset[1,r]$ be the connected component of the collision graph $\mathcal{G}^{[0,\ts]}_\e$ which contains the particles involve in the local recollision. Because particles in $\varpi$ do not interact with particles in $\varpi^c$, $\Zt_r(t,Z_r)$ restricted to particles $\varpi$ and to time interval $[0,\tau_{\text{stop}}]$ can be represented by a pseudotrajectory $Z_{|\varpi|}(\tau,((s''_i,\bar{s}''_i)_i,(\kappa_j'')_j),Z_\varpi)$ for some collision parameters $((s''_i,\bar{s}''_i)_i,(\kappa_j'')_j)$. Note that we can take the $\kappa_j''$ smaller than the maximum of the $\kappa'_j$ and thus smaller than $r-1$. This gives a more precise  constraints on  $Z_\varpi$: $\{Z_r,\,\chi_r(Z_r):=1\}$ is include in
		\[\bigcup_{\varpi\subset\{1,\cdots,r\}}\bigcup_{\substack{(s_i,\bar{s}_i)_{i\leq|\omega|}\\ (\kappa_j)\in[0,r-1]^{|\varpi|}}}\left\{Z_r \left|\begin{array}{c}\text{the~collision~graph~ of}\\\Zt_{|\varpi|}(\tau,((s_i,\bar{s}_i)_i,(\kappa_j)_j),Z_\varpi)\\\text{is connected}\end{array}\right. \right\}.\]
		Note this union is done on a finite set of parameters.
		
		We can now fix the pseudotrajectory's parameters. Usual estimation on the pseudotrajectories development gives that for some collision parameters (see estimation of Section \ref{Estimation of the long range recollisions.}),
		\[\begin{split}\int_{\Lambda^{|\varpi|}\times B_{|\varpi|}(\mathbb{V})}\ind_{\substack{\,\Zt_{|\varpi|}(\tau)\,\text{forms\,a\,cluster}\\ \text{with\,local\,recollision}}}e^{-\frac{1}{2}\|V_r\|^2}dV_\varpi dX_{\varpi\setminus\{\min \varpi\}}&\leq C_|\varpi| \mu_\e^{-|\varpi|+1}\delta^{|\varpi|-2}\e|\log\e|\\
		&\leq C_{|\varpi|} \mu_\e^{-|\varpi|+1}\delta^{2}\left(\mu_\e\delta^d\mathbb{V}^d\right)^{|\varpi|-3}\e^a\end{split}\]
		where we use that $|\varpi|\geq 2$ and that $\e|\log\e|/\delta = O(\e^\alpha)$ for some $\alpha$.
		
		Using the distance constraints on $Z_{\varpi^c}$ and summing on all possible parameters, we obtain the expected bound.
	\end{proof}
	
	Finally we denote $\mathcal{X}_{(i_1,\cdots,i_n)}:\{\gr{Z}_\N\in \mathcal{D}_\e,~\N\geq \max \ui_n\}\mapsto\{0,1\}$ the indicator of the set
	\[\left\{\gr{Z}_{\N}\in\mathcal{D}_\e\Big\vert \,\exists \varpi,\, \varpi\cap(i_1,\cdots,i_n)\neq\emptyset,\, \chi(\gr{Z}_\varpi)=1\right\}.\]
	Note that $\mathcal{X}_{\ui_n }$ depends on the background.
	
	We have:
	\begin{equation}
	\mathcal{X}_{(i_1,\cdots,i_n)}(\gr{Z}_{\N})~=~ 1~~~~ - \prod_{\substack{\varpi\subset\{1,\cdots,\N\\\varpi\cap\{i_1,\cdots,i_n\}\neq\emptyset}} \big(1-\chi(\gr{Z}_\varpi)\big).
	\end{equation}
%
	
	\paragraph{The two conditioning allow us to bound the number of recollision.} Let $Z_n\in\mathcal{D}^n_\e$ be an initial position such that there is no distance cluster of size bigger than $\gamma$ (first conditioning) and for any $\sigma\subset\{1,\cdots,n\}$, $\chi(Z_\sigma)=0$ (second conditioning). We fix collision parameters $((s_i,\bar{s}_i)_i,(\kappa_j)_j)$ such that pseudotrajectory $\Zt_n(t)$ has no recollision on $[\delta,t]$. Due to the symmetric conditioning a particle can only meet $\gamma-1$ different particles on $[0,\delta]$. Due to the asymetric conditioning there is no local recollision on $[0,\delta]$ and a particle has at most $\gamma-1$ collision on $[0,\delta]$. Finally there are at most $\gamma-1$ recollisions by particles.

	Thus any pseudotrajectories of this type can be parameterized by collision parameters \[((s_i,\bar{s}_i)_{1\leq i\leq n-1}(\kappa_j)_{1\leq j\leq n})\in\{\pm1\}^{2(n-1)}\times[0,\gamma-1]^n.\] 
	
	\subsection{Sampling}
	Using the two conditioning of the previous  part, for $\gr{Z}_{\N}\in \Upsilon_{\e}\cap\left\{\mathcal{X}_{\ui_m}\left(\gr{Z}_\N\right)=0\right\}$ we have
	\begin{equation}
	h_m\left((\gr{Z}_{\ui_m}(\tau)\right) = \sum_{n\geq m} \sum_{(i_{m+1},\cdots,i_{n})} \Phi_{m\leftarrow n}^{\gamma,\tau}[h_m](\gr{Z}_{\ui_n}(0))
	\end{equation}
	where
	\begin{equation}
	\Phi_{m\leftarrow n}^{\gamma,\tau}[h_m](Z_n) :=\frac{1}{(n-m)!} \sum_{\substack{(s_i,\bar{s}_i)_{i\leq n-m}\\(\kappa_j)_j\in[1,\gamma-1]^n}} \prod_{i= 1}^{n-m} \bar{s}_i \ind_{\mathcal{R}_{((s_i,\bar{s}_i),(\kappa_j))}^{m\leftarrow n,\gamma,\tau}}h_m(\Zt_n(\tau)),
	\end{equation}
	and ${\mathcal{R}_{((s_i,\bar{s}_i),(\kappa_j))}^{m\leftarrow n,\gamma,\tau}}\subset{\mathcal{R}_{((s_i,\bar{s}_i),(\kappa_j))}^{m\leftarrow n,\tau}}$ with no local recollision.
	
	We can then do the following decomposition on $\Upsilon_{\e}$
	\begin{equation}
	\begin{split}
	\sum_{\ui_m}h_m\left(\gr{Z}_{\ui_m}(t)\right)=&\sum_{\ui_m}h_m(\gr{Z}_{\ui_m}(t))\mathcal{X}_{\ui_m}(\gr{Z}_{\N}(t-\delta))+\sum_{\ui_m}h_m(\gr{Z}_{\ui_m}(t))\big(1-\mathcal{X}_{\ui_m}(\gr{Z}_{\N}(t-\delta))\big)\\
	=&\sum_{\ui_m}h_m(\gr{Z}_{\ui_m}(t))\mathcal{X}_{\ui_m}(\gr{Z}_{\N}(t-\delta))\\
	&+\sum_{n\geq m} \sum_{\ui_n}\Phi_{m\leftarrow n}^{\gamma,\tau}[h_m](\gr{Z}_{\underline{i}_n}(t-\delta))\big(1-\mathcal{X}_{\ui_m}(\gr{Z}_{\N}(t-\delta))\big)
	\end{split}
	\end{equation}
	
	Let $\Phi_{m, n}^{\gamma,\tau}[h_m]$ be the symmetrization  $\Phi_{m\leftarrow n}^{\gamma,\tau}[h_m]$:
	\[\Phi_{m, n}^{\gamma,\tau}[h_m](Z_n):=\frac{1}{n!}\sum_{\sigma\in\mathfrak{S}_n}\Phi_{m\leftarrow n}^\gamma[h_m](z_{\sigma(1)},\cdots,z_{\sigma(n)}).\]
	There is a more explicit formula for. We define $\mathcal{R}_{((s_i,\bar{s}_i),(\kappa_j))}^{m, n,\gamma,\tau} \subset \mathcal{D}^n_\e$ as the set of initial data such that pseudotrajectories $\Zt_n(\cdot)$ has $m$ remaining particles at time $\tau$, $\kappa_j(\tau)=0$ for all $j$ and no local recollision. Then 
	\begin{equation}
	\Phi_{m,n}^{\gamma,\tau}[h_m](Z_n) := \frac{1}{n!}\sum_{\substack{(s_i,\bar{s}_i)_{i\leq n-m}\\(\kappa_j)_j,\in[0,\gamma-1]}} \prod_{i= 1}^{n-m} \bar{s}_i \ind_{\mathcal{R}_{((s_i,\bar{s}_i),(\kappa_j))}^{m, n,\gamma,\tau}}h_m(\Zt_n(\tau)).
	\end{equation}
	
	Finally we want to separate pseudotrajectories without recollision. We define the development along pseudotrajectories without recollision
	\begin{equation}
	\Phi^{0,\tau}_{m, n}[h_m](Z_n) :=\frac{1}{n!} \sum_{(s_i,\bar{s}_i)_{ i\leq n-m}} \prod_{i= 1}^{n-m} \bar{s}_i \ind_{\mathcal{R}_{(s_i,\bar{s}_i)}^{m, n,\tau}}h_m(\Zt_n(\tau)),
	\end{equation}
	where $\mathcal{R}_{(s_i,\bar{s}_i)}^{m, n,\tau}\subset\mathcal{R}_{(s_i,\bar{s}_i),(0)_j}^{m, n,\gamma,\tau}$ such that the pseudotrajectories has no recollision and the part with development along pseudotrajectories with non pathological recollision
	\begin{equation}
	\Phi_{m,n}^{>,\tau}[h_m] := \Phi_{m,n}^{\gamma,\tau}[h_m]-\Phi_{m,n}^{0,\tau}[h_m].
	\end{equation}

	We bring together all these decomposition and obtain on $\Upsilon_\e$,
	\begin{equation}
	\begin{split}
	\sum_{\ui_m}h_m\left(\gr{Z}_{\ui}(\tau)\right)=&\sum_{n\geq m} \sum_{\ui_n}\Phi^{0,\delta}_{m,n}[h_m](\gr{Z}_{\underline{i}}(\tau-\delta))+\sum_{n\geq m} \sum_{\ui_n}\Phi^{>,\delta}_{m,n}[h_m](\gr{Z}_{\underline{i}_n}(\tau-\delta))\\
	&+\sum_{\ui_m}h_m(\gr{Z}_{\ui}(\tau))\mathcal{X}_{\ui_m}(\gr{Z}_{\N}(\tau-\delta))\\
	&-\sum_{n\geq m}\sum_{\ui_n}\Phi^{\gamma,\delta}_{m\leftarrow n}[h_m](\gr{Z}_{\ui_n}(\tau-\delta))\mathcal{X}_{\ui_m}(\gr{Z}_{\N}(\tau-\delta)).
	\end{split}
	\end{equation}
	
	The first term is the expansion with respect to pseudotrajectories with no recollision. It is the main part of the sum. The rest takes into accounts the recollision of the dynamics. 
	
	We iterate this decomposition:
	\begin{equation*}
	\begin{split}
	\sum_{\ui_m}h_m\left(\gr{Z}_{\ui}(\tau)\right)&=\sum_{n\geq m} \sum_{\ui_n}\Phi^{0,\tau}_{m,n}[h_m](\gr{Z}_{\underline{i}}(0))+\sum_{k=0}^{\tau/\delta}\sum_{n\geq m} \sum_{\ui_n}\Phi^{>,k\delta}_{m,n}[h_m](\gr{Z}_{\underline{i}_n}(\tau-k\delta))\zeta_\e^0(g)\ind_{\Upsilon_\e}\\
	&+\sum_{n\geq m}\sum_{\ui_n}\Phi^{0,(k-1)\delta}_{m,n}[h_m](\gr{Z}_{\ui_n}(\tau-(k-1)\delta))\mathcal{X}_{\ui}(\gr{Z}_{\N}(\tau-k\delta))\\
	&-\sum_{n'\geq n\geq m}\sum_{\ui_{n'}}\Phi^{\gamma,\delta}_{n\leftarrow n'}\Big[\Phi_{m,n}^{0,(k-1)\delta}[h_m]\Big](\gr{Z}_{\ui_{n'}}(\tau-k\delta))\mathcal{X}_{\ui_m}(\gr{Z}_{\N}(\tau-k\delta)).
	\end{split}
	\end{equation*}
	
	The final ingredient is to make a second sampling on longer time $\theta \sim |\log\e|^{-1}$ in order to control the growth of the number of collision in during time. We denote $K:=t/\theta\in\mathbb{N}$ and $K':= \theta/\delta\in\mathbb{N}$.
	
	We obtain the following decomposition	
	\begin{equation}
		\mathbb{E}_\e\left[\zeta_\e^t(h)\zeta_\e^0(g)\right]= G_\e^{\text{main}}(t) + G_\e^{\text{clust}}(t) + G_\e^{\text{exp}}(t)+G_\e^{\text{rec},1}(t)+G_\e^{\text{rec},2}(t)
	\end{equation}
	with  $G^{\text{main}}_\e(t)$ the main part:
	\begin{equation}
	G_\e^{\text{main}}(t) := \sum_{\substack{\underline{n}:=(n_j)_{j\leq K}\\0<n_j-n_{j-1}\leq 2^j}} \mathbb{E}_\e\left[\mu_\e^{-1/2}\sum_{\ui_{n_K}}\Phi^0_{\underline{n}}[h]\left(\gr{Z}_{\underline{i}_{n_K}}(0)\right)\zeta^0_\e(g)\right]
	\end{equation}
	where \[\Phi^0_{\underline{n}}[h]:=\Phi^{0,\theta}_{n_{K-1},n_K}\circ\Phi^{0,\theta}_{n_{K-2},n_{K-1}}\cdots\circ\Phi^{0,\theta}_{1,n_{1}}[h]\] is the development of $h$ along pseudotrajectories tree with $n_k$ creation on the time interval $[t-(k-1)\theta,t-k\theta]$ and no recollision, denoting $\underline{n}=(n_1,\cdots,n_K)$,
	\begin{equation}
	G_\e^{\text{clust}}(t) := \mathbb{E}_\e \left[\zeta_\e^t(h) \zeta_\e^0(g) \ind_{\Upsilon^c_\e}\right]- \sum_{\substack{n_1\leq\cdots\leq n_K\\n_j-n_{j-1}\leq 2^j}} \mathbb{E}_\e\left[\mu_\e^{-1/2}\sum_{\ui_{n_K}} \Phi^0_{\underline{n}}[h]\left(\gr{Z}_{\underline{i}_{n_K}}(0)\right)\zeta^0_\e(g)\ind_{\Upsilon^c_\e}\right]
	\end{equation}
	corresponding to the symmetric conditioning,
	\begin{equation}
	G_\e^{\text{exp}}(t) := \sum_{k=1}^K \sum_{\substack{n_1\leq\cdots\leq n_{k-1}\\n_j-n_{j-1}\leq 2^j}}\sum_{n_k>2^k+n_{k-1}} \mathbb{E}_\e\left[\mu_\e^{-1/2}\sum_{\ui_{n_k}} \Phi^0_{\underline{n}}[h]\left(\gr{Z}_{\underline{i}_{n_k}}(t-k\theta))\right)\zeta^0_\e(g)\ind_{\Upsilon_\e}\right]
	\end{equation}
	corresponding with tree with superexponential growth, and the part of (non local) recollision:
	\begin{equation}
	\begin{split}
	&G_\e^{\text{rec},1}(t):=\sum_{\substack{1\leq k\leq K-1\\1\leq k'\leq K'}} \sum_{\substack{n_1\leq\cdots\leq n_k\\n_j-n_{j-1}\leq 2^j}} \sum_{\substack{n''\geq n'\geq n_k}} \mathbb{E}_\e\left[\mu_\e^{-1/2}\sum_{\ui_{n''}}\Phi_{\underline{n},n',n''}^{>,k'}[h]\left(\gr{Z}_{\ui}(t_s)\right)\,\zeta^0_\e(g)\ind_{\Upsilon_\e}\right]
	\end{split}
	\end{equation}
	with \[\Phi_{\underline{n},n',n''}^{>,k'}[h]:=\Phi_{n',n''}^{>,\delta}\circ\Phi_{n_k,n'}^{0,k'\delta}\circ\Phi^0_{\underline{n}}[h],\] and the part of pathological pseudotrajectories $G_\e^{\text{rec},2}(t)$:
	\begin{equation}
	\begin{split}
	\sum_{\substack{1\leq k\leq K-1\\1\leq k'\leq K'}} &\sum_{\substack{n_1\leq\cdots\leq n_k\\n_j-n_{j-1}\leq 2^j}}\left(\sum_{\substack{n'\geq n_k}} \mathbb{E}_\e\left[\mu_\e^{-1/2}\sum_{(i_1,\cdots,i_{n'})}\Phi_{\underline{n},n'}^{0,k'}[h]\left(\gr{Z}_{\ui}(t_s+\delta)\right)\mathcal{X}_{\ui}\big(\gr{Z}_{\N}(t_s)\big)\,\zeta^0_\e(g)\ind_{\Upsilon_\e}\right]\right.\\
	&-\left.\sum_{n''\geq n'\geq n_k} \mathbb{E}_\e\left[\mu_\e^{-1/2}\sum_{(i_1,\cdots,i_{n''})}\Phi_{n'\leftarrow n''}^{\gamma}\left[\Phi_{\underline{n},n'}^{0,k'}\left[h\right]\right]\left(\gr{Z}_{\ui}(t_s)\right)\mathcal{X}_{\ui_{n'}}\big(\gr{Z}_{\N}(t_s)\big)\,\zeta^0_\e(g)\ind_{\Upsilon_\e}\right]\right).
	\end{split}
	\end{equation}
	In the last two terms, we denote $t_s:=t-(k-1)\theta-k'\delta$ (for \emph{stopping time}) and
	\begin{itemize}
		\item $\Phi_{\underline{n},n'}^{0	,k'}[h]:=\Phi_{n_k,n'}^{0,k'\delta}\circ\Phi^0_{\underline{n}}[h],$ the tree development with no recollision, $n'$ annihilations on $[0,(k'-1)\delta]$ and for $j<k$, $n_j$ annihilations on $[(k'-1)\delta+(k-j)\theta,(k'-1)\delta+(k-j+1)\theta]$,
		\item $\Phi_{\underline{n},n',n''}^{>,k'}[h]:=\Phi_{n',n''}^{>,\delta}\circ\Phi_{\underline{n},n'}^{0	,k'}[h],$  the tree development with no recollision on $[\delta,k'\delta+k\theta]$, $n''$ annihilations on $[0,\delta]$, $n'$ annihilations on $[\delta,k'\delta]$ and for $j<k$ and $n_j$ annihilations on $[k'\delta+(k-j)\theta,k'\delta+(k-j-1)\theta]$, and with the least one recollision.
		\end{itemize}
	In addition thanks to the conditioning, every pseudotrajectories have at most $\gamma$ recollision by particles.
	
	The convergence of $G_\e^{\text{main}}$ is treated in section \ref{Treatment of the main part}. The bound $G_\e^{\text{clust}}$, $G_\e^{\text{exp}}$ and $G_\e^{\text{rec},1}$ have already be done in the original paper. The section \ref{$L^2$ estimation of the local recollision part} is dedicated to the 
	
	\section{Quasi-orthogonality estimates}\label{Quasi-orthogonality estimates}

	The different error terms are of the form
	\begin{equation*}
	\mathbb{E}_\e\left[\sum_{\ui_n}\Phi_{n}[h](\gr{Z}_{\ui_n}(t_{\text{stop}}))\zeta_\e^0(g)\ind_{\Upsilon_\e}\right]
	\end{equation*}
	with $\Phi_n$ sum continuous functional $L^\infty(\mathbb{D})\to L^\infty(\mathbb{D}^n)$. In order to bound it we will need an $L^2(\mathbb{P}_\e)$ bound on $\sum_{\ui_n}\Phi_{n}[h](\gr{Z}_{\ui_n})$. Such bound is derived in the following section from estimation on the $\Phi_n[h]$. We use in particular that we can bounding the $\Phi_n[h](Z_n)$ by looking only at the relative positions of particles inside $Z_n$.
	
	In the following we denote for $y\in\Lambda$
	\begin{equation}
	\tau_a:\left\{\begin{array}{r@{~}c@{~}l} \mathbb{D}^n~~&\rightarrow&~~\mathbb{D}^n\\
	(X_n,V_n)&\mapsto& (x_1+a,\cdots,x_n+a,V_n).
	\end{array}\right.
	\end{equation} 
	\begin{theorem}\label{theoreme de quasi orthogonalite}
		Fix $m<n$ two positive integers, and $g_n$, $h_{m}$ two functions on $\mathbb{D}^n$ and $\mathbb{D}^{m}$ such that there exists a finite sequence $(c_0,c_0',c_1,\cdots,c_n)\in\mathbb{R}^{n+1}_+$  bounding $g_n,\,h_{m}$ in the following way:
		\begin{equation}\label{borne sur g_n}
		\int\limits_{x_1=0}\sup_{y\in\Lambda}\big|g_n\big(\tau_yZ_n\big)\big|M^{\otimes n}(V_n)dX_{2,n}dV_n\leq c_0,
		\end{equation}
		\begin{equation}\label{borne sur h_m}
		\int\limits_{x_1=0}\sup_{y\in\Lambda}\big|h_{m}\big(\tau_yZ_{m}\big)\big|M^{\otimes m}(V_m)dX_{2,m}dV_{m}\leq c'_0
		\end{equation}
		and for all $l \in [1,m]$
		\begin{equation}
		\begin{split}
		\int\limits_{x_1=0}\sup_{y\in\Lambda} \big|g_n\big(\tau_yZ_n\big)h_{m}\big(\tau_yZ_{n+1-l,n+m-l}\big)\big|M^ {\otimes (n+m-l)}(V_{n+m-l})dX_{2,n+m-l}dV_{n+m-l}\\
		\leq \frac{\mu_\e^{l-1}}{n^l}c_l.
		\end{split}
		\end{equation}
		There exits a constant $C>0$ depending only on dimension such that
		\begin{equation}\label{borne espérence}
		\big|\mathbb{E}_\e\big[g_n\big]\big|\leq C^n c_0,~\big|\mathbb{E}_\e\big[h_{m}\big]\big|\leq C^{m} c'_0
		\end{equation}
		and denoting 
		\begin{equation}
		g_n\circledast_lh_{m}(Z_{n+m-l})=\frac{1}{(n+m-l)!}\sum_{\sigma\in\mathfrak{S}_{n+m-l}} g_n(Z_{\sigma([1,n])})h_{m}(Z_{\sigma([n+1-l,n+m-l])}),
		\end{equation}
		\begin{equation}\label{quasi covariance}\begin{split}
		\mathbb{E}_\e\Big[\mu_\e\hat{g}_n\hat{h}_{m}\Big]=\sum_{l=1}^{m}\binom{n}{l}\binom{m}{l}\frac{l!}{\mu_\e^{l-1}}\mathbb{E}_\e\big[g_n\circledast_lh_{m}\big]	+\,O\big(C^{n+m} c_0c'_0 \e\,\big).
		\end{split}\end{equation}

		In particular
		\begin{equation}\label{borne variance}
		|\mathbb{E}_\e\big[\mu_\e\hat{g}_n\hat{h}_{m}\big]\leq C^{n+m} \sum_{l=1}^{m}c_l+C^nc_0c_0'\,\e.
		\end{equation}
	\end{theorem}
	
	\begin{proof}[Proof of Theorem \ref{theoreme de quasi orthogonalite}]~
	
		$\bullet~~$We begin by the proof of \eqref{borne espérence}
		
		Using invariance under permutation, 
		\[\begin{split}
		\mathbb{E}_\e[g_n] &\frac{1}{\mu_\e^n\mathcal{Z}_\e}=\sum_{p\geq n}\frac{\mu_\e^p}{p!}\int \sum_{\substack{(i_1,\cdots i_n)\\\forall k, i_k\leq p}}g_n(Z_n)e^{-\mathcal{H}^\e_p(Z_p)}\frac{dZ_p}{(2\pi)^{dp/2}}\\
		&= \frac{1}{\mu_\e^n\mathcal{Z}_\e}\sum_{p\geq n}\frac{\mu_\e^p}{p!}\frac{p!}{(n-p!)}\int g_n(Z_n)e^{-\mathcal{H}^\e_p(Z_p)}\frac{dZ_p}{(2\pi)^{dp/2}}\\
		&= \frac{1}{\mathcal{Z}_\e}\sum_{p\geq 0}\frac{\mu_\e^p}{p!}\int g_n(Z_n)e^{-\mathcal{V}^\e_{n+p}(X_n,\underline{X}_p)}M^{\otimes n}dZ_nd\underline{X}_p.
		\end{split}\]
	
		We denote in the following $\Omega:=\{X_n,\underline{x}_1,\cdots,\underline{x}_p\}$ and for $X,Y\in\Omega$, 
		\[\varphi(X,Y):= -\ind_{d(X,Y)\leq\e}\]
		and we decompose $\exp\left({-\mathcal{V}_{n+1+p}^\e(X_{n+1},\underline{X}_p)}\right)$
		\[\begin{split}
		e^{-\mathcal{V}_{n+1+p}^\e(X_{n+1},\underline{X}_p)}
		&= e^{-\mathcal{V}_{n}^\e(X_{n})}\prod_{\substack{(X,Y)\in\Omega^2\\X\neq Y}}\left(1+\varphi(X,Y)\right)=e^{-\mathcal{V}_{n}^\e(X_{n})}\sum_{G\in \mathcal{G}(\Omega)}\prod_{(X,Y)\in E(G)}\varphi(X,Y)
		\end{split}\]
		where $\mathcal{G}$ is the set of non orientated graph on $\Omega$ and $E(G)$ the set of edges of $G$. We make the partition on the connected components of $X_n$. Denoting $\mathcal{C}(\omega)$ the set of connected graph,
		\begin{equation}\label{decomposition e^V_n+p}\begin{split}
		&\exp\left(-\mathcal{V}_{n+1+p}^\e(X_{n+1},\underline{X}_p)\right)\\
		&=\sum_{\substack{\omega\subset[1,p]}}\Bigg(e^{-\mathcal{V}_{n}^\e(X_{n})}\sum_{\substack{G\in \mathcal{C}(\omega\cup \{X_n\})}}\prod_{(X,Y)\in E(G)}\varphi(X,Y)\sum_{G\in \mathcal{G}(\omega^c)}\prod_{(X,Y)\in E(G)}\varphi(X,Y)\Bigg)\\
		&=\sum_{\substack{\omega\subset[1,p]}}\Bigg(e^{-\mathcal{V}_{n}^\e(X_{n})-\mathcal{V}_{|\omega^c|}^\e(\underline{X}_{\omega^c})}\sum_{\substack{G\in \mathcal{C}(\omega\cup\{X_n\})}}\prod_{(X,Y)\in E(G)}\varphi(X,Y)\Bigg)\\
		&=:\sum_{\substack{\omega\subset[1,p]}}e^{-\mathcal{V}_{|\omega^c|}^\e(\underline{X}_{\omega^c})}\psi_p^n(X_n,\underline{X}_\omega)
		\end{split}\end{equation}
		
		Thus using exchangeability,
		\begin{equation}\label{dévelopement de l'esperence}
		\begin{split}
		\mathbb{E}_\e[g_n] &= \frac{1}{\mathcal{Z}_\e}\sum_{p\geq 0}\sum_{p_1+p_2=p}\frac{\mu_\e^p}{p!}\frac{p!}{p_1!p_2!}\int g_n(Z_n)\psi^n_{p_1}(X_n,\underline{X}_{p_1})e^{-\mathcal{V}^\e_{p_2}(\underline{X}'_{p_2})}M^{\otimes n}dZ_nd\underline{X}_{p_1}d\underline{X}_{p_2}'\\
		&=\left(\frac{1}{\mathcal{Z}_\e}\sum_{p\geq0}\frac{\mu_\e^p}{p!}\int e^{-\mathcal{V}^\e_p(\underline{X}_p)}d\underline{X}_p\right)\left(\sum_{p\geq 0}\frac{\mu_\e^p}{p!}\int g_n(Z_n)\psi^n_{p}(X_n,\underline{X}_{p})M^{\otimes n}\frac{dZ_nd\underline{X}_{p}}{(2\pi)^{dn/2}}\right)\\
		&=\sum_{p\geq 0}\frac{\mu_\e^p}{p!}\int g_n(Z_n)\psi^n_{p}(X_n,\underline{X}_{p})M^{\otimes n}dZ_nd\underline{X}_{p}.
		\end{split}
		\end{equation}
		
		We recall Penrose tree inequality (see \cite{Penrose,BGSS, Jansen}),
		\begin{equation}
		\left|\sum_{C\in\mathcal{C}(\Omega)}\prod_{(X,Y)\in E(C)}\varphi(X,Y)\right|\leq \sum_{T\in\mathcal{T}(\Omega)}\prod_{(X,Y)\in E(T)}|\varphi(X,Y)|
		\end{equation}
		with $\mathcal{T}(\Omega)$ the set of trees (minimally connected graph) on $\Omega$. Fix  $\tau_{-x_1}{X_n}$ (the relative position between particles)for the moment. Integrating a constraints $\varphi(\underline{x}_i,\underline{x}_j)$ provides a factor $\gr{c}_d\e^d$, $\varphi(X_n,\underline{x}_j)$ a factor $n\gr{c}_d\e^d$ (where $\gr{c}_d$ is the volume of a sphere of diameter $1$). As there are
		\[\frac{(p-1)!}{(d_0-1)!(d_1-1)!\cdots(d_{p}-1)!}\]
		trees with specified vertex degrees $d_0,\cdots,d_{p}$ associated to vertices $X_n,\,\underline{x}_1,\cdots,\,\underline{x}_p$ (see \cite{Jansen,BGSS}), we get
		\begin{equation}\label{Borne de l'integrale de psi np}\begin{split}
		\bigg|\int \psi^n_p(&X_n\underline{X}_p)d\underline{X}_p dx_1\bigg|\leq \sum_{\substack{d_1,\cdots,d_{p}\geq 1\\d_0+\cdots+d_{p}=2p}}\frac{(p-1)!}{(d_0-1)!(d_1-1)!\cdots(d_{p}-1)!} n^{d_0}(\gr{c}_{\gr{d}}\e^d)^p\\
		&\leq (p-1)!(\gr{c}_{\gr{d}}\e^d)^p\left(\sum_{d_0\geq 1}\frac{n^{d_0}}{(d_0-1)!}\right)\left(\sum_{d_1\geq 1}\frac{1}{(d_1-1)!}\right)\cdots\left(\sum_{d_{p}\geq 1}\frac{1}{(d_p-1)!}\right)\\
		&\leq (p-1)!\big(\gr{c}_{\gr{d}}\e^d\big)^pne^{n+p}.
		\end{split}\end{equation}
		We can integrate on the rest of parameters using \eqref{borne sur g_n}. Hence		
		\[|\mathbb{E}_\e[g_n]|\leq\sum_{p\geq 0}\frac{(p-1)!\big(\gr{c}_{\gr{d}}e\e^d\big)^pne^n}{p!}\int |g_n(Z_n)|e^{-\frac{\|V_n\|^2}{2}}\frac{dZ_n}{(2\pi)^{dn/2}}\leq c_0\sum_{p\geq0} C^n(C\e)^p\]
		which converges for $\e$ small enough. This concludes the proof of \ref{borne espérence}.~~\\
		
		$\bullet~~$We treat now \eqref{quasi covariance}. Note first that
		\[\mathbb{E}_\e\Big[\mu_\e\hat{g}_n\hat{h}_{m}\Big] = \frac{1}{\mu_\e^{n+m-1}}\mathbb{E}_\e\left[\sum_{\ui_n}g_n(\gr{Z}_{\ui_n})\sum_{\underline{j}_{m}}h_{m}(\gr{Z}_{\underline{j}_{m}})\right]-\mu_\e\mathbb{E}_\e\left[g_n\right]\mathbb{E}_\e\left[h_{m}\right].\]
		Lets count the number of way such that $\ui_n$ and $\underline{j}_{m}$ can intersect on a set of length $l$. We have to choose two set $A\subset[n]$ and $A'\subset[m]$ of length $l$, and a bijection $\sigma:A\to A'$ such that for all indices $k\in A$, $i_k=j_{\sigma{k}}$ and  that $\ui_{A^c}$ does not interesect $\underline{j}_{(A^c)'}$. Thus using the symmetry, 
		
		\[\begin{split}\mathbb{E}_\e\Big[\mu_\e\hat{g}_n\hat{h}_{m}\Big] = &\sum_{l=1}^{m}\binom{n}{l}\binom{m}{l}\frac{l!}{\mu_\e^{l-1}}\mathbb{E}_\e\big[g_n\circledast_lh_{m}\big]\\
		&+\mu_\e\left(\mathbb{E}_\e\left[\frac{1}{\mu_\e^{n+m}}\sum_{\underline{i}_{n+m}}g_n(\gr{Z}_{\ui_{n}})h_{m}(\gr{Z}_{\ui_{n+1,n+m}})\right]-\mathbb{E}_\e\left[g_n\right]\mathbb{E}_\e\left[g\right]\right).
		\end{split}\]
		
		We have to estimates the error term.
		\[\begin{split}
		\mathbb{E}_\e\Bigg[&\frac{1}{\mu_\e^{n+m}}\sum_{\underline{i}_{n+m}}g_n(\gr{Z}_{\ui_{n}})h_{m}(\gr{Z}_{\ui_{n+1,n+m}})\Bigg]\\
		&=\frac{1}{\mathcal{Z}_\e}\sum_{p\geq 0} \frac{\mu_\e^p}{p!}\int g_n(Z_n)h_{m}(Z'_{m})\exp\left(-\mathcal{V}^\e_{n+m+p}(X_{n},X'_{m},\underline{X}_p)\right)M^{\otimes n}dZ_{n}M^{\otimes m}dZ'_{m}d\underline{X}_p.
		\end{split}\]
		
		We denote in the following 
		$\Omega:=\{X_n,X_m',\underline{x}_1,\cdots,\underline{x}_p\}$ .
		\[\begin{split}
		\exp\left(-\mathcal{V}^\e_{n+m+p}(X_{n},X'_{m},\underline{X}_p)\right)
		&= e^{-\mathcal{V}^\e_{n}(X_{n})}e^{-\mathcal{V}^\e_{m}(X_{m})}\prod_{\substack{(X,Y)\Omega^2\\X\neq Y}}\left(1+\varphi(X,Y)\right)\\
		&=e^{-\mathcal{V}^\e_{n}(X_{n})}e^{-\mathcal{V}^\e_{m}(X_{m})}\sum_{G\in \mathcal{G}(\Omega)}\prod_{(X,Y)\in E(G)}\varphi(X,Y)
		\end{split}\]
		where $\mathcal{G}$ is the set of non orientated graph on $\Omega$ and $E(G)$the set of edges of $G$. We make the partition on the connected components of $X_n$ and $X'_{m}$.
		\[\everymath={\displaystyle}\begin{array}{r@{}r@{}r@{}l}
		\exp&\multicolumn{2}{l}{\left(-\mathcal{V}^\e_{n+m+p}(X_{n},X'_{m},\underline{X}_p)\right)}\\
		&\multicolumn{2}{l}{=\sum_{\substack{\omega\subset[1,p]}}\Bigg(\exp\left(-\mathcal{V}_{n}^\e(X_{n})-\mathcal{V}_{m}^\e(X'_{m})-\mathcal{V}_{|\omega^c|}^\e(\underline{X}_{\omega^c})\right)\sum_{\substack{G\in \mathcal{C}(\omega\cup\\ \{X_n,X_{m}'\})}}\prod_{(X,Y)\in E(G)}\varphi(X,Y)\Bigg)}\\[10pt]
		&&+\sum_{\substack{\omega_1,\omega_1\subset[1,p]\\\omega_1\cap\omega_2= \emptyset}}\psi^n_{|\omega_1|}(X_n,\underline{X}_{\omega_1})\psi^{m}_{|\omega_2|}(X'_{m},\underline{X}_{\omega_2})e^{-\mathcal{V}^\e_{|(\omega_1\cup\omega_2)^c|}(\underline{X}_{(\omega_1\cup\omega_2)^c})}\\[25pt]
		&\multicolumn{2}{l}{=:\sum_{\substack{\omega\subset[1,p]}}\psi^{n,m}_{|\omega|}(X_{n},X_{m}',\underline{X}_\omega)e^{-\mathcal{V}_{||\omega^c|}^\e(\underline{X}_{\omega^c})}}\\[-5pt]
		&&+\sum_{\substack{\omega_1,\omega_1\subset[1,p]\\\omega_1\cap\omega_2= \emptyset}}\psi^n_{|\omega_1|}(X_n,\underline{X}_{\omega_1})\psi^{m}_{|\omega_2|}(X'_{m},\underline{X}_{\omega_2})e^{-\mathcal{V}^\e_{|(\omega_1\cup\omega_2)^c|}(\underline{X}_{(\omega_1\cup\omega_2)^c})}.
		\end{array}\]
		Using invariance under permutation and \eqref{dévelopement de l'esperence}
		\[\everymath={\displaystyle}\begin{array}{r@{}l@{}r}
		\frac{1}{\mathcal{Z}_\e}\sum_{p\geq 0} &\multicolumn{2}{l}{\frac{\mu_\e^p}{p!}\int g_n(Z_n)h_{m}(Z'_{m})\sum_{\substack{\omega_1,\omega_1\subset[1,p]\\\omega_1\cap\omega_2= \emptyset}}\psi^n_{|\omega_1|}(X_n,\underline{X}_{\omega_1})\psi^{m}_{|\omega_2|}(X'_{m},\underline{X}_{\omega_2})e^{-\mathcal{V}^\e_{|(\omega_1\cup\omega_2)^c|}(\underline{X}_{(\omega_1\cup\omega_2)^c})}}\\[-5pt]
		&&\times M^{\otimes (n+m)}dZ_{n}dZ'_{m}M^{\otimes n}dZ_{n}M^{\otimes n'}dZ'_{n'}d\underline{X}_p\\[5pt]
		&\multicolumn{2}{l}{=\frac{1}{\mathcal{Z}_\e}\sum_{p\geq 0}\sum_{p_1+p_2+p_3=p} \frac{\mu_\e^p}{p!}\frac{p!}{p_1!p_2!p_3!}\int g_n(Z_n)h_{n'}(Z'_{n'})\psi^{n}_{p_1}(X_{n},\underline{X}_{p_1})\psi^{1}_{p_2}(x_{n+1},\underline{X}'_{p_2})}\\
		&&\times\Big(M^{\otimes n}dZ_{n}d\underline{X}_{p_1}\Big)\Big(M^{\otimes n'}dZ'_{n'}d\underline{X}'_{p_2}\Big)\Big(e^{-\mathcal{V}^\e_{p_3}(\underline{X}''_{p_3})}d\underline{X}''_{p_3}\Big)\\
		&\multicolumn{2}{l}{=\mathbb{E}_\e[g_n]\mathbb{E}_\e[h_{n'}],}
		\end{array}\]
		and in the same way
		\[\everymath={\displaystyle}\begin{array}{r@{}l@{}r}
		\frac{1}{\mathcal{Z}_\e}&\multicolumn{2}{l}{\sum_{p\geq 0} \frac{\mu_\e^p}{p!}\int g_n(Z_n)h_{m}(Z'_{m})\sum_{\substack{\omega\subset[1,p]}}\psi^{n,m}_{|\omega|}(X_{n},X_{m}',\underline{X}_\omega)e^{-\mathcal{V}_{||\omega^c|}^\e(\underline{X}_{\omega^c})}M^{\otimes n}dZ_{n}M^{\otimes m}dZ'_{m}d\underline{X}_p}\\[15pt]
		&\multicolumn{2}{l}{=\frac{1}{\mathcal{Z}_\e}\sum_{p\geq 0}\sum_{p_1+p_2=p} \frac{\mu_\e^p}{p!}\frac{p!}{p_1!p_2!}\int g_n(Z_n)h_{m}(Z'_{m})\psi^{n,m}_{|\omega|}(X_{n},X_{m}',\underline{X}_{p_1})}\\
		&&e^{-\mathcal{V}^\e_{p_2}(\underline{X}'_{p_2})}M^{\otimes(n+m)}dZ_{n}dZ'_{m}d\underline{X}_{p_1}d\underline{X}'_{p_2}\\[3pt]
		&\multicolumn{2}{l}{=\sum_{p_1\geq 0} \frac{\mu_\e^p}{p_1!}\int g_n(Z_n)h_{m}(Z'_{m})\psi^{n,m}_{|\omega|}(X_{n},X_{m}',\underline{X}_{p_1})M^{\otimes(n+m)}dZ_{n}dZ'_{m}d\underline{X}_{p_1}d\underline{X}'_{p_2}.}
		\end{array}\]
		
		Using again Penrose tree inequality,
		\begin{equation}
		\left|\psi^{n,m}_{|\omega|}(X_{n},X_{m}',\underline{X}_{p_1})\right|\leq \sum_{T\in\mathcal{T}(\Omega)}\prod_{(X,Y)\in E(T)}|\varphi(X,Y)|
		\end{equation}
		Fix $\tau_{-x_1}{X_n}$ and $\tau_{-x'_1}{X'_m}$ for the moment. Integrating a constraints $\varphi(\underline{x}_i,\underline{x}_j)$ provides a factor $\gr{c}_d\e^d$, $\varphi(X_n,\underline{x}_j)$ a factor $n\gr{c}_d\e^d$, $\varphi(X'_m,\underline{x}_j)$ a factor $m\gr{c}_d\e^d$ and $\varphi(X_n,X'_m)$ a factor $nm\gr{c}_d\e^d$. Denoting $d_0,d'_0,d_1\cdots,d_{p}$ the degrees of $X_n,\,X_m',\,\underline{x}_1,\cdots,\,\underline{x}_m$ we get
		\begin{equation}\begin{split}
		\bigg|\int \psi^{n,m}_{|\omega|}&(X_{n},X_{m}',\underline{X}_{p_1})d\underline{X}_p dx_1dx'_1\bigg|\\
		&\leq \sum_{\substack{d'_0,d_0,\cdots,d_{p}\geq 1\\d'_0+d_0+\cdots+d_{p}=2p}}\frac{p!}{(d'_0-1)(d_0-1)!\cdots(d_{p}-1)!} n^{d_0}m^{d'_0}(\gr{c}_{\gr{d}}\e^d)^{+1}\\
		&\leq p!\big(\gr{c}_{\gr{d}}\e^d\big)^{p+1}nm\,e^{n+m+p}.
		\end{split}\end{equation}
		We can integrate on the rest of parameters using \eqref{borne sur g_n} and \eqref{borne sur h_m}.
		
		Finally		
		\[\begin{split}
		&\mu_\e\left(\mathbb{E}_\e\left[\frac{1}{\mu_\e^{n+m}}\sum_{\underline{i}_{n+m}}g_n(\gr{Z}_{\ui_{n}})h_{m}(\gr{Z}_{\ui_{n+1,n+m}})\right]-\mathbb{E}_\e\left[g_n\right]\mathbb{E}_\e\left[g\right]\right).\\
		&\leq c_0 c_0' \mu_\e\sum_{p\geq 0} \frac{\mu_\e^p}{p!}p!\big(\gr{c}_{\gr{d}}\e^d\big)^{p+1}nm\,e^{n+m+p} \\
		&\leq \mu_\e\e^dnm(\gr{c}_{\gr{d}}e)^{n+m}c_0 c_0' \sum_{p\geq 0} (e\gr{c}_{\gr{d}}\e)^{p}\\
		&\leq \e C^{n+m+1}\sum_{p\geq 0} (e\gr{c}_{\gr{d}}\e)^{p}
		\end{split}\]
		which converge for $\e$ small enough.
	\end{proof}

	Note also the following bound on $L^p$ norms of the fluctuation.
	\begin{theorem}
		For any $p\in[2,\infty)$, there exists a constant $C_p>0$ such that
		\begin{equation}\label{$L^p$ bound of the fluctuation field}
		\left(\mathbb{E}_\e\left[\zeta^0_\e(g)^p\right]\right)^{1/p}\leq C_p \|g\|_{L^p(M(v)dz)}.
		\end{equation}
	\end{theorem}
	The proof can be found in Appendix A of \cite{BGSS1}.
	
	From these estimations we can deduce the following corollary:
	\begin{corollary}\label{Corollaire utilisant la quasi orthogonalite}
		Let $h_n$ such satisfying conditions of theorem \ref{theoreme de quasi orthogonalite}. Then there exists a constant $C>0$ such that
		\begin{equation}
		\begin{split}
		\Bigg|\mathbb{E}_\e\Bigg[\mu_\e^{-1/2}\sum_{(i_1,\cdots,i_n)}&h_n(\gr{Z}_\ui(t_s))\zeta^0_\e(g)\ind_{\Upsilon_{\e}}\Bigg]\Bigg|\\
		&\leq C^n\mu_e^{n-1} \mathbb{E}_\e \big[\zeta_\e^0(g)^2\big]^{1/2}\left(c_0+\left(\sum_{l=1}^nc_l\right)^{1/2}\right).
		\end{split}
		\end{equation}
		
	\end{corollary}
	\begin{proof}
		\[\begin{split}\mathbb{E}_\e\Bigg[\mu_\e^{-1/2}&\sum_{(i_1,\cdots,i_n)}h_n(\gr{Z}_\ui(t_s))\zeta^0_\e(g)\ind_{\Upsilon_{\e}}\Bigg]=\mu_\e^{n-1}\mathbb{E}_\e\Bigg[\mu_\e^{1/2-n}\sum_{(i_1,\cdots,i_n)}h_n(\gr{Z}_\ui(t_s))\zeta^0_\e(g)\ind_{\Upsilon_{\e}}\Bigg]\\
		&~~~=\mu_\e^{n-1}\bigg(\mathbb{E}_\e\Big[\mu_\e^{1/2}\widehat{h_n}(\gr{Z}_\N(t_s))\,\zeta^0_\e(g)\ind_{\Upsilon_{\e}}\Big]+\mathbb{E}_\e\left[h_n\right]\mathbb{E}_\e\Big[\mu_\e^{1/2}\zeta^0_\e(g)\ind_{\Upsilon_{\e}}\Big]\bigg)\\
		&~~~=\mu_\e^{n-1}\bigg(\mathbb{E}_\e\Big[\mu_\e^{1/2}\widehat{h_n}(\gr{Z}_\N(t_s))\,\zeta^0_\e(g)\ind_{\Upsilon_{\e}}\Big]+\mathbb{E}_\e\left[h_n\right]\mathbb{E}_\e\Big[\zeta^0_\e(g)\mu_\e^{1/2}\left(1-\ind_{\Upsilon^c_{\e}}\right)\Big]\bigg)
		\end{split}\]
		because $\mathbb{E}_\e[\zeta_\e^0(g)]=0$. Using Cauchy-Schwartz inequality
		\[\begin{split}
		\Bigg|\mathbb{E}_\e\Bigg[&\mu_\e^{-1/2}\sum_{(i_1,\cdots,i_n)}h_n(\gr{Z}_\ui(t_s))\zeta^0_\e(g)\ind_{\Upsilon_{\e}}\Bigg]\Bigg|\\
		&\leq\mu_\e^{n-1}\left(\mathbb{E}_\e\left[\mu_\e\left[\widehat{h_n}\right]^2\right]^{\frac{1}{2}}\mathbb{E}_\e\big[\zeta^0_\e(g)^2\big]^{\frac{1}{2}}+\mathbb{E}_\e\left[h_n\right]\mathbb{E}_\e\big[\zeta^0_\e(g)^2\big]^{\frac{1}{2}}\big(\mu_\e\mathbb{P}_\e\big[\Upsilon_{\e}^c\big]\big)^{\frac1 2}.\right)
		\end{split}\]
		We apply now theorem \ref{theoreme de quasi orthogonalite}. The bound on $\mathbb{P}_\e\left[\Upsilon_{\e}^c\right]$ given in section \ref{Conditioning} and the bound on $L^p$ norm of $\zeta^0_\e(g)$ \eqref{$L^p$ bound of the fluctuation field} lead  to the expecting bound.
	\end{proof}
	
	\section{Clustering estimations}\label{Clustering estimations}
	The objective of this section is to this section is to bound $G^{\text{clust}}_\e(t)$ and $G^{\text{exp}}_\e(t)$ defined by
	\begin{equation*}
	G_\e^{\text{clust}}(t) := \mathbb{E}_\e \left[\zeta_\e^t(h) \zeta_\e^0(g) \ind_{\Upsilon^c_\e}\right]- \sum_{\substack{n_1\leq\cdots\leq n_K\\n_j-n_{j-1}\leq 2^j}} \mathbb{E}_\e\left[\mu_\e^{-1/2}\sum_{(i_1,\cdots,i_{n_K})} \Phi^0_{\underline{n}}[h]\left(\gr{Z}_{\underline{i}_{n_K}}(0)\right)\zeta^0_\e(g)\ind_{\Upsilon^c_\e}\right],
	\end{equation*}
	\begin{equation*}
	G_\e^{\text{exp}}(t) := \sum_{k=1}^K \sum_{\substack{n_1\leq\cdots\leq n_{k-1}\\n_j-n_{j-1}\leq 2^j}}\sum_{n_k>2^k+n_{k-1}} \mathbb{E}_\e\left[\mu_\e^{-1/2}\sum_{(i_1,\cdots,i_{n_k})} \Phi^0_{\underline{n}}[h]\left(\gr{Z}_{\underline{i}_{n_k}}(t-k\theta))\right)\zeta^0_\e(g)\ind_{\Upsilon_\e}\right].
	\end{equation*}
	
	\begin{prop}
		For $\e>0$ small enough, 
		\begin{equation}\label{Estimation morceau 1}
		\left|G^{\text{exp}}_\e(t)+G^{\text{rec}}_\e(t)\right| \leq C\|g\|\,\|h\|\,\left(\e^{1/2}2^{(t/\theta)^2}+t\theta^{1/2}\right)
		\end{equation}
	\end{prop}

	We need bound on the development on pseudotrajectories without recollision $\Phi_{\underline{n}}^0[h]$:
	\begin{prop}
		Fix $k\in\mathbb{N}$ and $\underline{n}:=(n_1,\cdots,n_k)\in\mathbb{N}^k$ with $n_1\leq n_2\leq\cdots\leq n_k$. Then 
		\begin{equation}\label{Estimation sans reco 1}
		\int_{x_{1}=0}\sup_{y\in\Lambda}\big|\Phi_{\underline{n}}^0[h](\tau_yZ_{n_k})\big|M^{\otimes n_K}dV_{n_K}dX_{2,n_K}\leq \frac{\|h\|}{\mu_\e^{n_k}}C^{n_k}\theta^{n_k-n_{k-1}}t^{n_{k-1}-1},
		\end{equation}
		for $m\in[1,n_k]$
		\begin{equation}\label{Estimation sans reco 2}
		\begin{split}
		\int_{x_{1}=0}\sup_{y\in\Lambda}\big|\Phi_{\underline{n}}^0[h](\tau_yZ_{n_k})\Phi_{\underline{n}}^0[h](\tau_yZ_{n_k-m+1,2n_k-m})&\big|M^{\otimes (2n_K-m)}dV_{2n_K-m}dX_{2,2n_K-m}\\
		&\leq \frac{\|h\|^2}{n_k^m\mu_\e^{2n_k-m}}C^{n_k}\theta^{n_k-n_{k-1}}t^{m+n_{k-1}-1}.
		\end{split}
		\end{equation}
	\end{prop}
	Using Corollary \ref{Corollaire utilisant la quasi orthogonalite} and the previous estimations, 
	\[\begin{split}
	\Bigg|\mathbb{E}_\e\Bigg[\mu_\e^{-1/2}\sum_{(i_1,\cdots,i_{n_k})}& \Phi^0_{\underline{n}}[h]\left(\gr{Z}_{\underline{i}_{n_k}}(t-k\theta)\right)\zeta^0_\e(g)\ind_{\Upsilon_\e}\Bigg]\Bigg|\\
	&\leq \|g\|\|h\| \left(C^{n_k}\theta^{n_k-n_{k-1}}t^{n_{k-1}-1}+\left(\sum_{m=1}^{n_k}C^{n_k}\theta^{n_k-n_{k-1}}t^{n_{k-1}-1+m}\right)^{1/2}\right)\\
	&\leq \|g\|\|h\|C^{n_k} \theta^{(n_k-n_{k-1})/2}t^{n_{k}-1},
	\end{split}\]
	and in the same way,
	\[\begin{split}
	\mathbb{E}_\e\Bigg[\mu_\e^{-1/2}\sum_{(i_1,\cdots,i_{n_K})}& \Phi^0_{\underline{n}}[h]\left(\gr{Z}_{\underline{i}_{n_K}}(0)\right)\zeta^0_\e(g)\ind_{\Upsilon^c_\e}\Bigg]= O\left(\e^{1/2}\|g\|\|h\|C^{n_k} t^{n_{k}-1}\right).
	\end{split}\]
	
	Summing on all possible $(n_1,\cdots,n_k)$,
	\begin{equation}\begin{split}
	\left|G_\e^{\text{exp}}(t)\right| &\leq \sum_{k=1}^K \sum_{\substack{n_1\leq\cdots\leq n_{k-1}\\n_j-n_{j-1}\leq 2^j}}\sum_{n_k>2^k+n_{k-1}} \|g\|\|h\|C^{n_k} \theta^{(n_k-n_{k-1})/2}t^{n_{k}-1}\\
	&\leq  \|g\|\|h\|\sum_{k=1}^K 2^{k^2}\big(t\theta^{1/2}\big)^{2^k}\leq C\|g\|\|h\|\,t\,\theta^{1/2}
	\end{split}\end{equation}
	because the series converges for $\theta$ small enough, and 
	\begin{equation}\begin{split}
	\left|G_\e^{\text{clust}}(t)\right| &\leq C\|g\|\|h\|\e^{1/2}+ \sum_{\substack{n_1\leq\cdots\leq n_K\\n_j-n_{j-1}\leq 2^j}}\e^{1/2}\|g\|\|h\|C^{n_k} t^{n_{k}-1}\leq C\|g\|\|h\|\e^{1/2}2^{K^2}.
	\end{split}\end{equation}
	This concludes the proof of \eqref{Estimation morceau 1}.

	\begin{proof}[Proof of \eqref{Estimation sans reco 1}]
	We recall that 
	\[\Phi_{\underline{n}}^0[h] = \Phi^{0,\theta}_{n_{K-1},n_K}\circ\Phi^{0,\theta}_{n_{K-2},n_{K-1}}\cdots\circ\Phi^{0,\theta}_{1,n_{1}}[h] = \frac{1}{n_k!} \sum_{(s_i,\bar{s}_i)_{ i\leq n_k-1}}  \bar{s}_i \ind_{\mathcal{R}_{(s_i,\bar{s}_i)}^{\underline{n}}}h(\Zt_n(k\theta))\]
	and thus 
	\begin{equation}\label{|Phi^0_n|}
	\left|\Phi_{\underline{n}}^0[h]\right| \leq \frac{\|h\|}{n_k!} 
	\sum_{(s_i,\bar{s}_i)_{ i\leq n_k-1}}  \ind_{\mathcal{R}_{(s_i,\bar{s}_i)}^{\underline{n}}}
	 \end{equation}
	where $\mathcal{R}_{(s_i,\bar{s}_i)}^{\underline{n}}\subset \mathcal{D}^{n_k}_\e$ the set of initial parameters $Z_{n_k}$ such that pseudotrajectory $\Zt_{n_k}(\tau, (s_i,\bar{s}_i)_i,$ $(0)_j,Z_{n_k})$ has $n_l$ remaining particles at time $(k-l)\theta$. Note that the left member of \eqref{|Phi^0_n|} it is stable under translation. Hence it is sufficient to fix $x_1= 0$ and integrate with respect to $(X_{2,n_k},V_{n_k})$.
	
	We define the the \emph{clustering tree} $T^{>}:=(\nu_i,\bar{\nu}_i)_{1\leq i \leq n_k-1}$ where the $i$-th collision happens between particles $\nu_i$ and $\bar{\nu}_i$ (and $\nu_i<\bar{\nu}_i$). Since in the present section seudotrajectories have no recollision, the clustering tree is just the collision graph where we forget the collisions times (but not there order). It constructs a partition of $\mathcal{R}_{(s_i,\bar{s}_i)_i}^{\underline{n}}$.
	
	Fix the clustering tree. We perform the following change of variables
	\[X_{2,n_k}\mapsto(\hat{x}_1,\cdots,\hat{x}_{n_k-1}),~\forall i \in [1,n_k-1],~\hat{x}_i:=x_{\nu_i}-x_{\bar{\nu}_i}\]
	
	Fix $t_{i+1}$ the time of the $(i+1)$-th collision and relative positions $\hat{x}_1,\cdots,\hat{x}_{i-1}$. We denote $T_i = \theta$ if $i\leq n_k-n_{k-1}$, $t$ else (at least $n_k-n_{k-1}$ clustering collisions happens before time $\theta$) and the $i$-th collision set as
	\[B_{T^>,i}:=\left\{\hat{x}_i\Big|\exists \tau\in(0,T_i\wedge t_{i+1}),~|\text{x}_{\nu_{i}}(\tau)-\text{x}_{\bar{\nu}_i}(\tau)|\leq \e\right\}.\]
	Because particles $\text{x}_{\nu_{i}}(\tau)$ and $\text{x}_{\bar{\nu}_i}(\tau)$ are independent until their first meeting, we can do the change of variable $\hat{x}_i\mapsto(t_i,\eta_i)$ where $t_i$ is the first meeting time and 
	\[\eta_i:=\frac{\text{x}_{\nu_{i}}(t_i)-\text{x}_{\bar{\nu}_i}(t_i)}{\left|\text{x}_{\nu_{i}}(t_i)-\text{x}_{\bar{\nu}_i}(t_i)\right|}.\]
	This send the Lebesgue measure $d\hat{x}_i$ to the measure $\mu_\e^{-1}((\text{v}_{\nu_{i}}(t_i)-\text{v}_{\bar{\nu}_i}(t_i))\cdot\eta_i)_+ d\eta_idt_i$ and 
	\[\int \ind_{B_{T^>,i}}d\hat{x}_i \leq \frac{4\pi}{\mu_\e}\left|\text{v}_{\nu_{i}}(t_i)-\text{v}_{\bar{\nu}_i}(t_i)\right|\int _0^{T_i\wedge t_{i+1}}dt_i.\]
	We sum now on every possible edges $(\nu_{i},\bar{\nu}_i)$:
	\[\sum_{(\nu_{i},\bar{\nu}_i)}\left|\text{v}_{\nu_{i}}(t_i)-\text{v}_{\bar{\nu}_i}(t_i)\right|\leq n_k\sum_{k}\left|\text{v}_{k}(t_i)\right|\leq n_k\left(n_k\sum_{k}\left|\text{v}_{k}(t_i)\right|^2\right)^{1/2}\leq \frac{n_k}{2}\left(n_k+\|V_{n_k}\|^2\right)\]
	using that the kinetic energy is decreasing for the pseudotrajectory, hence
	\[\begin{split}\int \ind_{\mathcal{R}^{\underline{n}}_{(s_i,\bar{s}_i)}}d\hat{x}_1\cdots d\hat{x}_{n_k-1}&\leq \left(\frac{Cn_k}{\mu_\e}\right)^{n_k-1}\left(n_k + \|V_{n_k}\|^2\right)^{n_k-1}\int_0^{T_{n_k-1}}dt_{n_k}\cdots\int_0^{T_1\wedge t_2}dt_1\\
	&\leq \left(\frac{Cn_k}{\mu_\e}\right)^{n_k-1}\left(n_k + \|V_{n_k}\|^2\right)^{n_k-1}\frac{t^{n_{k-1}-1}}{(n_{k-1}-1)!}\frac{\theta^{n_k-n_{k-1}}}{(n_k-n_{k-1})!}\\
	&\leq \left(\frac{6C}{\mu_\e}\right)^{n_k-1}\left(n_k + \|V_{n_k}\|^2\right)^{n_k-1}t^{n_{k-1}-1}\theta^{n_k-n_{k-1}},\end{split}\]
	using the Stirling's formula. For $A,B>0$, $x\in\mathbb{R}$, 
	\[\left(A+x^2\right)^B e^{-\frac{x^2}{4}} = B^B\left(\frac{A+x^2}{B}e^{-\frac{A+x^2}{4B}}\right)^B e^{\frac{A}{4}}\leq\left(\tfrac{4B}{e}\right)^B e^{\frac{A}{4}}.\]
	Thus for some constant $C>0$,
	\[\int \left(n_k + \|V_{n_k}\|^2\right)^{n_k-1}e^{-\frac{\|V_{n_k}\|^2}{2}}dV_{n_k}\leq (Cn_k)^{n_k-1}\int e^{-\frac{\|V_{n_k}\|^2}{4}}dV_{n_k}\leq \big(2^{d/2}Cn_k\big)^{n_k}\]
	and
	\[\begin{split}
	\int \ind_{\mathcal{R}^{\underline{n}}_{(s_i,\bar{s}_i)}} M^{\otimes n_k}&dX_{2,n_k}dV_{n_k}\leq \sum_{T^>}\int \prod_{i=1}^{n_k-1}\ind_{B_{T^>,i}}d\hat{x}_i~M^{\otimes n_k}dV_{n_k}\\
	&\leq C\left(\frac{C}{\mu_e}\right)^{n_k-1} t^{n_{k-1}-1}\theta^{n_k-n_{k-1}}\int \left(n_k + \|V_{n_k}\|^2\right)^{n_k-1}M^{\otimes n_k}dV_{n_k}\\
	&\leq C'\left(\frac{C'}{\mu_e}\right)^{n_k-1} t^{n_{k-1}-1}\theta^{n_k-n_{k-1}}n_k^{n_k-1},
	\end{split}\]
	where we denote $C$ an other constant.
	
	Finally we sum on the $4^{n_k-1}$ possible $(s_i,\bar{s}_i)_i$ and dividing by the remaining $(n_k)!$. This gives the expected estimation.
	\end{proof}

	\begin{proof}[Proof of \eqref{Estimation sans reco 2}]
		We begin as in the previous paragraph
		\begin{align*}\Big|\Phi_{\underline{n}}^0[h](Z_{n_k})\Phi_{\underline{n}}^0[h](&Z_{n_k-m+1,2n_k-m})\Big|\\
		&\quad \leq \frac{\|h\|^2}{(n_k!)^2} 
		\sum_{\substack{(s_i,\bar{s}_i)_{ i\leq n_k-1}\\(s'_i,\bar{s}'_i)_{ i\leq n_k-1}}}  \ind_{\mathcal{R}_{(s_i,\bar{s}_i)}^{\underline{n}}}(Z_{n_k})\ind_{\mathcal{R}_{(s'_i,\bar{s}'_i)}^{\underline{n}}}(Z_{n_k-m+1,2n_k-m}).\end{align*}
		
		We have two pseudotrajectories $\Zt(\tau):=\Zt(\tau,Z_{n_k})$ and $\Zt'(\tau):=\Zt(\tau,Z_{n_k-m+1,2n_k-m})$. Note again that the right member is invariant under translation, so we can fix $x_1= 0$.

		We construct the clustering tree $T^>$ as follows: we merge collision graph of the first and the second pseudo trajectory. Then we look at edges one by one in temporal order, keeping only one which do not create a cycle. This construct a tree which connect all the vertices. 
		\begin{figure}[h]
			\includegraphics[scale=0.22]{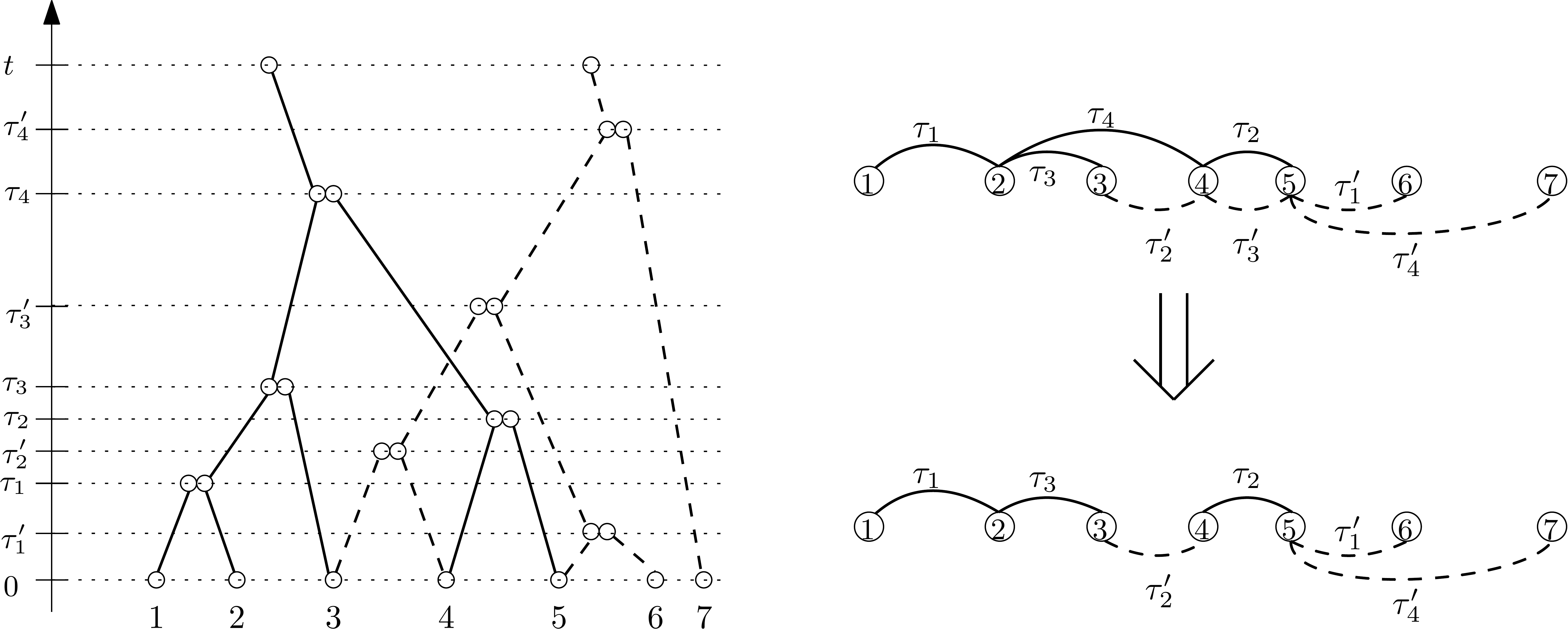}
			\caption{Example of construction of the clustering tree}
		\end{figure}
		
		This form a graph with ordered edges. We remove then the non-clustering collisions. This is the clustering tree $T^>:=(\nu_i,\bar{\nu}_i)$. They induce a partition of \[\left\{Z_{2n_k-m}\in(\Lambda\times\mathbb{R}^d)^{2n_k-m}\vert\,Z_{n_k}\in\mathcal{R}_{(s_i,\bar{s}_i)}^{\underline{n}},\,(Z_{n_k-m+1,2n_k-m})\in\mathcal{R}_{(s'_i,\bar{s}'_i)}^{\underline{n}}\right\}.\]
		
		The rest of the proof is almost the same than in the previous section. Fix the clustering tree. We perform the following change of variable
		\[X_{2,2n_k-m}\mapsto(\hat{x}_1,\cdots,\hat{x}_{2n_k-m-1}),~\forall i \in [1,2n_k-m-1],~\hat{x}_i:=x_{\nu_i}-x_{\bar{\nu}_i}.\]
		
		Fix $t_{i+1}$ the time of the $(i+1)$-th collision and relative positions $\hat{x}_1,\cdots,\hat{x}_{i-1}$. We define the $i$-th collision sets as
		\[B_{T^>,i}:=\left\{\hat{x}_i\Big|\exists \tau\in(0,T_i\wedge t_{i+1}),~|\text{x}_{\nu_{i}}(\tau)-\text{x}_{\bar{\nu}_i}(\tau)|\leq \e \text{\,or\,}|\text{x}'_{\nu_{i}}(\tau)-\text{x}'_{\bar{\nu}_i}(\tau)|\leq \e\right\}\]
		where $T_i = \theta$ for the $(n_k-n_{k-1})$ first collision, $t$ else. Using the same computation than i the previous section, denoting $t_i$ the minimal clustering time,
		\[\int \ind_{B_{T^>,i}}d\hat{x}_i \leq \frac{4\pi}{\mu_\e}\left(\left|\text{v}_{\nu_{i}}(t_i)-\text{v}_{\bar{\nu}_i}(t_i)\right|+\left|\text{v}_{\nu_{i}}(t_i)-\text{v}_{\bar{\nu}_i}(t_i)\right|\right)\int _0^{T_i\wedge t_{i+1}}dt_i.\]		
		Hence as in the previous section
		\[\begin{split}
		\int \ind_{\mathcal{R}_{(s_i,\bar{s}_i)}^{\underline{n}}}(Z_{n_k})\ind_{\mathcal{R}_{(s'_i,\bar{s}'_i)}^{\underline{n}}}&(Z_{n_k-m+1,2n_k-m}) M^{\otimes(2n_k-m)}dX_{2,2n_k-m}dV_{2n_k-m}\\
		&\leq \sum_{T^>}\int \prod_{i=1}^{2n_k-m-1}\ind_{B_{T^>,i}}d\hat{x}_i~ M^{\otimes(2n_k-m)}dV_{2n_k-m}\\
		&\leq C\left(\frac{C}{\mu_e}\right)^{2n_k-m} t^{n_{k-1}+m-1}\theta^{n_k-n_{k-1}}(2n_k-m)^{2n_k-m-1}\\
		&\leq C\frac{\left(2C\right)^{2n_k}}{\mu_e^{2n_k-m-1}} t^{n_{k-1}+m-1}\theta^{n_k-n_{k-1}}n_k^{2n_k-m-1}.
		\end{split}\]
		We sum on the possible $(s_i,\bar{s}_i)_i$ and $(s'_i,\bar{s}'_i)_i$ and
		\begin{align*}
		\int\Big|\Phi_{\underline{n}}^0[h](Z_{n_k})\Phi_{\underline{n}}^0[h](Z_{n_k-m+1,2n_k-m})\Big|M^{\otimes(2n_k-m)}dX_{2,2n_k-m}dV_{2n_k-m} \\
		\leq\|h\|^2 \frac{n_k^{2n_k-m-1}\tilde{C}^{2n_k}}{(n_k!)^2\mu_e^{2n_k-m-1}} t^{n_{k-1}+m-1}\theta^{n_k-n_{k-1}}
		\end{align*}
		Using Stirling formula gives expected estimation.
	\end{proof}

	\section{Estimation of non pathological recollisions}\label{Estimation of the long range recollisions.}
	
	The objective of this section is to bound
	\begin{equation*}
	\begin{split}
	&G_\e^{\text{rec},1}(t):=\sum_{\substack{1\leq k\leq K-1\\1\leq k'\leq K'}} \sum_{\substack{n_1\leq\cdots\leq n_k\\n_j-n_{j-1}\leq 2^j}} \sum_{\substack{n''\geq n'\geq n_k}} \mathbb{E}_\e\left[\mu_\e^{-1/2}\sum_{(i_1,\cdots,i_{n''})}\Phi_{\underline{n},n',n''}^{>,k'}[h]\left(\gr{Z}_{\ui}(t_s)\right)\,\zeta^0_\e(g)\ind_{\Upsilon_\e}\right].
	\end{split}
	\end{equation*}
	
	\begin{prop}
		For $\e$ small enough,
		\begin{equation}\label{Estimation morceau 3}\begin{split}
		\left|G_\e^{\text{rec},1}(t)\right|	&\leq  \|g\|\|h\| \e^{\alpha/2} (C't)^{2^{t/\theta}+2d+6}.
		\end{split}\end{equation}
	\end{prop}
	
	It is sufficient to prove the two following estimations:
	\begin{prop}
		Fix $k\in\mathbb{N}$, $\underline{n}:=(n_1,\cdots,n_k)\in\mathbb{N}^k$  and $(n',n'')\in\mathbb{N}^2$ with $n_1\leq n_2\leq\cdots\leq n_k\leq n'\leq n''$. Then fixing $x_1=0$,
		\begin{equation}\label{Estimation avec reco 1}
		\int\sup_{y\in\Lambda}\big|\Phi_{\underline{n},n',n''}^{>,k'}[h](\tau_y Z_{n''})\big|M^{\otimes n''}dV_{n''}dX_{2,n''}\leq \e^\alpha\frac{\|h\|}{\mu_\e^{n''}}C^{n''}\theta^{(n''-n_{k}-2)_+}\delta^2t^{n_k+2d+4},
		\end{equation}
		for $m\in[1,n'']$
		\begin{equation}\label{Estimation avec reco 2}
		\begin{split}
		\int\sup_{y\in\Lambda}\big|\Phi_{\underline{n},n',n''}^{>,k'}[h](\tau_yZ_{n''})\Phi_{\underline{n},n',n''}^{>,k'}[h](\tau_yZ_{n''-m+1,2n''-m})\big|M^{\otimes (2n''-m)}dV_{2n''-m}dX_{2,2n''-m}\\
		\leq \e^\alpha\frac{\|h\|}{\mu_\e^{n''}}C^{n''}\theta^{(n''-n_{k}-2)_+}\delta^2t^{n_k+2d+4+m}.
		\end{split}
		\end{equation}
	\end{prop}
	
	Using these estimations and corollary \ref{Corollaire utilisant la quasi orthogonalite}, 
	\[\begin{split}
	\Bigg|\mathbb{E}_\e\Bigg[&\mu_\e^{-1/2}\sum_{(i_1,\cdots,i_{n''})} \Phi_{\underline{n},n',n''}^{>,k'}[h]\left(\gr{Z}_{\underline{i}_{n''}}(t-k\theta)\right)\zeta^0_\e(g)\ind_{\Upsilon_\e}\Bigg]\Bigg|\\
	&\leq \|h\| \|g\|\left(\e^{\alpha}C^{n''}\theta^{(n''-n_{k}-2)_+}\delta^2t^{n_k+2d+4}+\left(\sum_{m=1}^{n_k}\e^{\alpha}C^{n''}\theta^{(n''-n_{k}-2)_+}\delta^2t^{n_k+2d+4+m}\right)^{1/2}\right)\\
	&\leq \|g\|\|h\| \delta\e^{\alpha/2}C^{n''} \theta^{(n''-n_{k}-2)_+/2}t^{\frac{n''+n_k}{2}+2d+4}\\[3pt]
	&\leq \|g\|\|h\| \delta\e^{\alpha/2}(Ct)^{n_k+2d+5} (Ct\theta)^{(n''-n_{k}-2)_+/2}.
	\end{split}\]
	
	Thus
	\begin{equation}\begin{split}
	\left|G_\e^{\text{rec},1}(t)\right|&\leq\sum_{\substack{1\leq k\leq K-1\\1\leq k'\leq K'}} \sum_{\substack{n_1\leq\cdots\leq n_k\\n_j-n_{j-1}\leq 2^j}} \sum_{\substack{n''\geq n'\geq n_k}}\|g\|\|h\| \delta\e^{\alpha/2}(Ct)^{n_k+2d+5} (Ct\theta)^{(n''-n_{k}-2)_+/2} \\
	&\leq \|g\|\|h\| K'\delta\e^{\alpha/2} K^{K^2}(Ct)^{2^K+2d+5}\\
	&\leq  \|g\|\|h\| \e^{\alpha/2} (C't)^{2^K+2d+6}
	\end{split}\end{equation}
	using that $K'\delta\leq t$.

	\begin{proof}[Proof of \eqref{Estimation avec reco 1}]
		We recall that 
		\[\begin{split}
		\Phi_{\underline{n},n',n''}^{>,k'}[h](Z_{n''})&=\Phi_{n',n''}^{>,\delta}\circ\Phi_{n_k,n'}^{0,k'\delta}\circ\Phi^0_{\underline{n}}[h](Z_{n''}) \\
		&=
		\frac{1}{n''!}\sum_{\substack{((s_i,\bar{s}_i)_i,(\kappa_j)_j)\\\kappa_j\leq \gamma-1}}\ind_{\mathcal{R}_{\substack{((s_i,\bar{s}_i)_i,(\kappa_j)_j)}}}h\big(\Zt(k\theta+k'\delta,((s_i,\bar{s}_i)_i,(\kappa_j)_j),Z_{n''})\big)
		\end{split}\]
		where $\mathcal{R}_{\substack{((s_i,\bar{s}_i)_i,(\kappa_j)_j)}}$ is the set of initial parameters such that the pseudotrajectory has 
		\begin{itemize}
			\item $n'$ particles at time $\delta$,
			\item $n_l$ particles at time $k'\delta + (k-l)\theta$,
			\item at least one recollision,
			\item no recollision after time $\delta$.
		\end{itemize}
		
		We define the clustering tree $T^>$ as follows: let $\mathcal{G}$ be the collision graph of $\Zt(\tau)$. We look at the collision in temporal order and add only the clustering collision. 
		
		It will not sufficient to categorize initial data. Let $(q,\bar{q})$ (with $q<\bar{q}$) be the first two particles to have a non-clustering collision, $\tau_{\rm cycle}$ the time of this collision and $c\in [1,n''-1]$ such that it happens between time the $c$-th and the $(c+1)$-th clustering collision. 
		
		The data $(T^>,(q,\bar{q},c))$ gives a partition of the set of initial data. Note that the family \[(T^>,(q,\bar{q},c,((s_i,\bar{s}_i)_i),(\kappa_j)_j))\]
		construct the collision graph up to time of non clustering collision. Considering the change of variables 
		\[\forall i \in [1,n''-1],~\hat{x}_i:=x_{\nu_i}-x_{\bar{\nu}_i},~X_{2,n''}\mapsto(\hat{x}_1,\cdots,\hat{x}_{n''-1})\]
		with $T^>=:(\nu_i,\bar{\nu_i})_{i\leq n''-1}$, we can construct as in the previous section a sequence of set $B^i_{T^>,(q,\bar{q},c)}$ depending only on $V_{n''}$ and $\hat{x}_1,\cdots,\hat{x}_{i-1}$ which condition the relative position $\hat{x}_i$. The construction has to take into account the apparition of cycle. We define in the following $(T_i)_i$ by $T_i$ equal $\delta$ if $i$ is smaller than $n''-n'$, $\theta$ if $i$ is between $n''-n'+1$ and $n''-n_k$ and $t$ else (it count the number of clustering collision in $[0,\delta]$, $[\delta,k'\delta]$ and $[k'\delta,k'\delta+k\theta]$).
		
		\begin{figure}[h]
			\centering
			\includegraphics[scale=0.3]{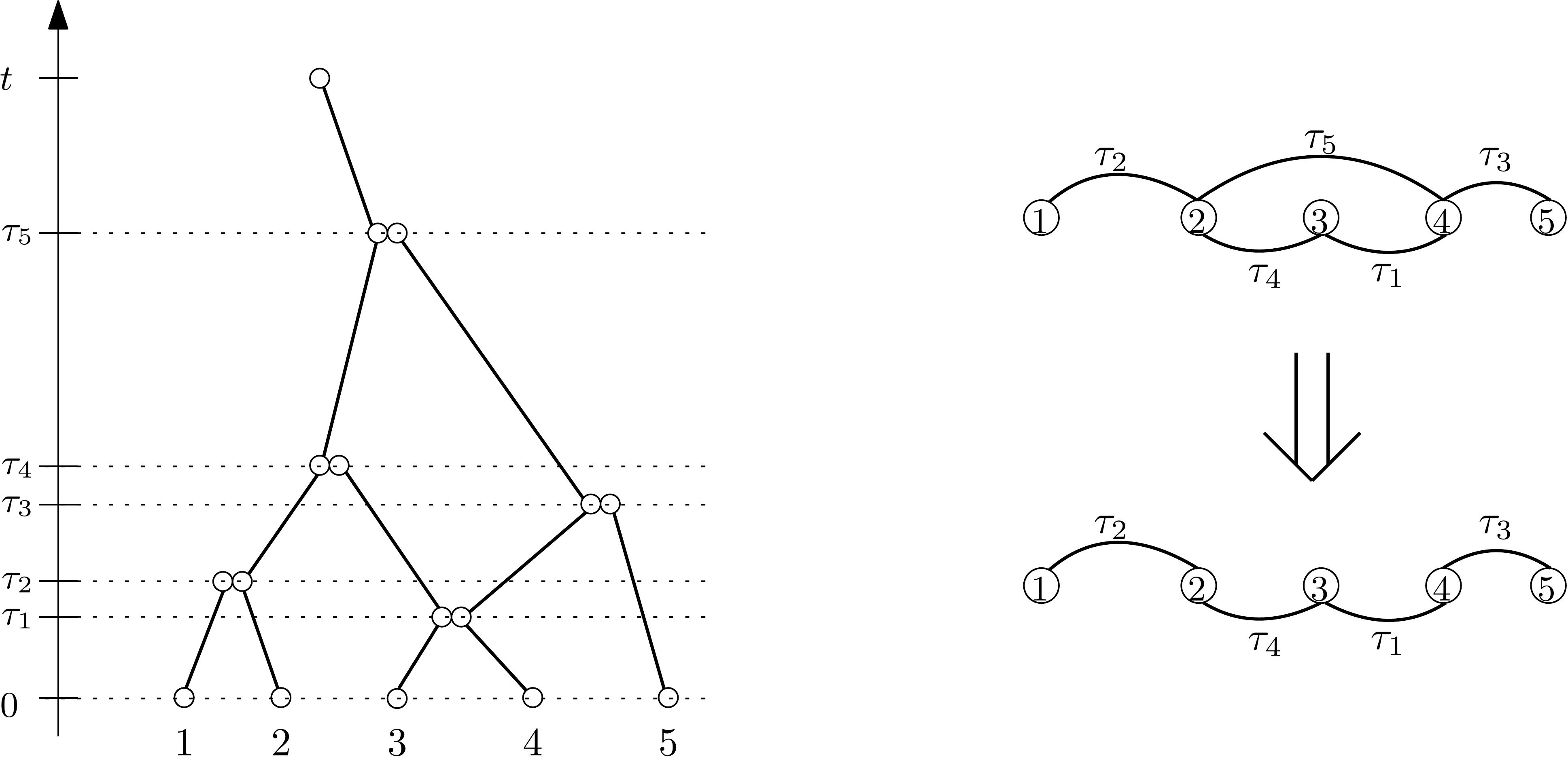}
			\caption{Example of construction of a clustering tree. Here $(q,\bar{q},c)= (2,4,4)$.}
		\end{figure}
		
		We need to characteristize particular collision in $T^>$ which conditions the apparition of the non-clustering collision.

		\begin{definition}\label{def-tutor}
			We call \underbar{parent}~$p$  of a group of particles $(q_k)_{k }$ at time $\tau$ the~$p$-th edge with the largest~$p$ such that one of the particles $(q_k)_{k }$ is deflected at $\tau_p \leq \tau$.      If such a parent does not exist, then we set~$\tau_p := 0$. 
		
		We define the   \underbar{connector}~$k$  of  two particles $(q,\bar q)$  the index of the first edge realizing a connected path between $q$ and $\bar q$.
		
		The \underbar{tutor}~$j$ of two particles $(q, \bar q)$ at time $\tau$  is the largest $j$ with $t_j \leq \tau$ such that $j$ is either the parent  at time $\tau$ or the connector of $(q, \bar q)$.
		\end{definition}
		
		\begin{figure}[h]
			\centering
			\includegraphics[scale=0.15]{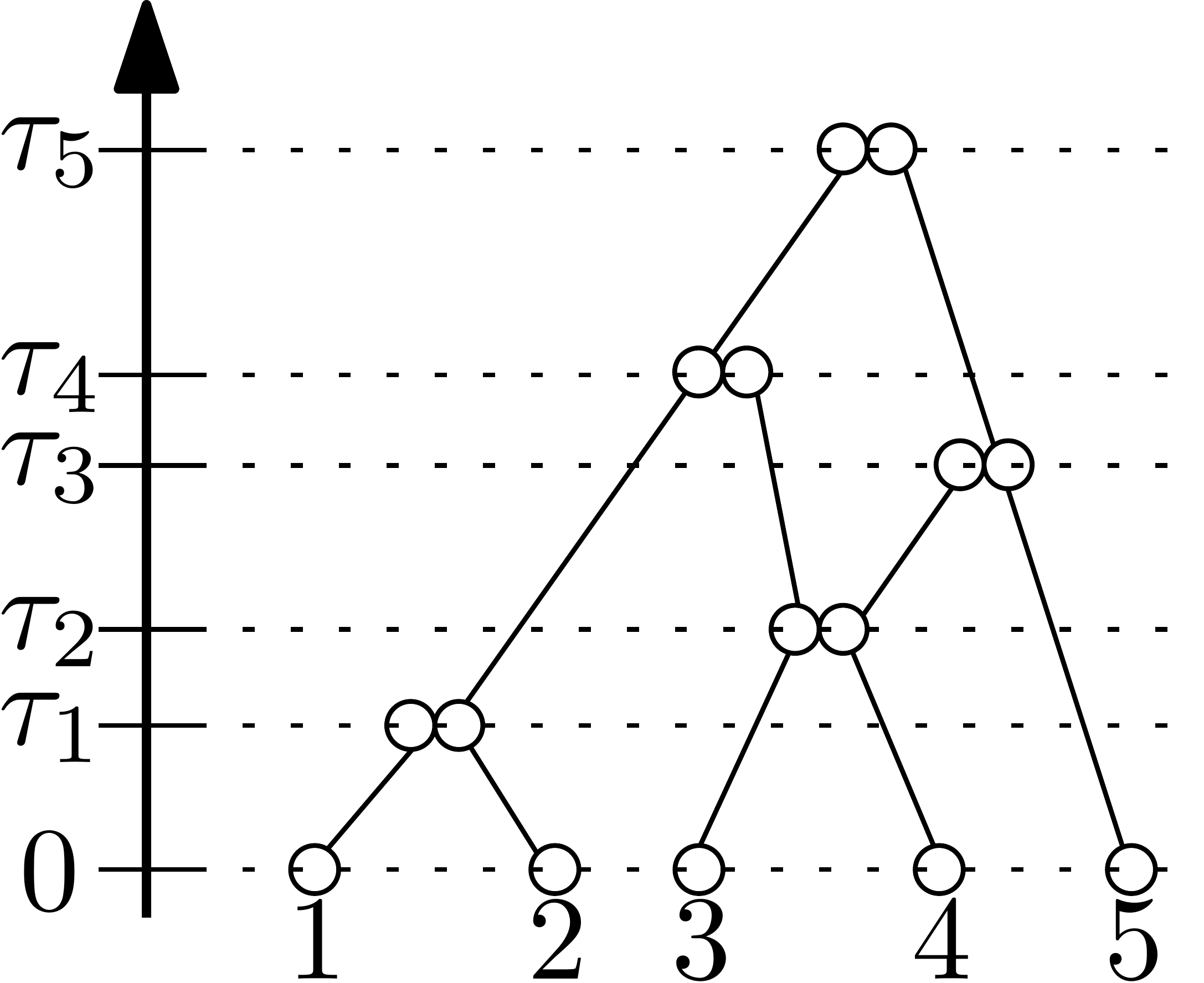}
			\caption{In this pseudotrajectory, the parent of collision between $2$ and $5$ at time $\tau_5$ is the collision between $1$ and $2$ at time $\tau_1$ and the connector the collision between $3$ and $4$ at time $\tau_2$.}
		\end{figure}
		
		The first particular collision is $j$ the tutor of $(q,\bar{q})$ before $\tau_{\rm cycle}$. We define
		\begin{align*}
		B^j_{T^>,(q,\bar{q},c)}:=\Big\{\hat{x}_j\Big|\exists \tau\in(0,T_j\wedge t_{j+1}),~|\text{x}_{\nu_{j}}&(\tau)-\text{x}_{\bar{\nu}_j}(\tau)|\leq \e\, \\
		&{\rm the\,}j{\rm -th\,collision\,is\,the\,tutor\,of\,the\,cycle}\Big\}.
		\end{align*}
		Note that after the clustering time $\tau_j$ particles $q$ and $\bar{q}$ do not change their velocities.
		
		\begin{prop}\label{prop: recollision} 
			Assume that~$d \geq 3$. Then, denoting by~$w_q ,  w_{ \bar q}, w_{q_j}, w_{\bar q_j}$ the velocities of $q, \bar q, q_j, \bar q_j$ at $t_{j-1}^+$ (which are the same than at time $t_{j}^-$), one has if the tutor $j$ is the parent of $(q,\bar{q})$,
			\begin{equation} \label{eq:integralparent1}
			\int  \ind_{B^j_{T^>,(q,\bar{q},c)}} \, d \hat x_j   \leq {C\over \mu_\e}  (\mathbb{V}  t)^d  \times  \left(  {\mathbb{V}\e |\log \e| \ind_{q \neq \bar q_j}\over  |w_q - w_{ \bar q_j }| }+ { \mathbb{V}\e |\log \e| \ind_{\bar q \neq \bar q_j}\over  |w_{\bar q} - w_{ \bar q_j }| } + {\mathbb{V} t \over \mu_\e} \right)  \, , 
			\end{equation}
			and if the tutor is a connector but not a parent 
			\begin{equation} 
			\label{eq:integralparent2} \int \ind_{B^j_{T^>,(q,\bar{q},c)}} \, d \hat x_j   \leq {C\over \mu_\e} ( \mathbb{V} \theta)^{d+1}  \times   \left[ \sum_{\zeta}\ind_{\sin (w_q - w_{\bar q}, \zeta) \leq \e} +(\mathbb{V}  \theta)^{d}   \min \Big(
			1, {\e \ind_{(q, \bar q) \neq (q_j, \bar q_j)} \over \sin\big(
				w_q - w_{\bar q},w_{q_j} -  w_{\bar q_j}\big)}
			\Big)\right]
			\end{equation}
			where the sum runs over $\zeta\in \mathbb{Z}^d\setminus\{0\}$ contained in the ball of radius $  \mathbb{V}\theta $.
		\end{prop} 
		
		The over $B^j_{T^>,(q,\bar{q},c)}$ are defined as in the previous section
		\begin{align*}
		B^i_{T^>,(q,\bar{q},c)}:=\Big\{\hat{x}_i\Big|\exists \tau\in(0,T_i\wedge t_{i+1}),~|\text{x}_{\nu_{i}}&(\tau)-\text{x}_{\bar{\nu}_i}(\tau)|\leq \e\Big\}.
		\end{align*}
		
		The above proposition uses the tutor to gain some smallness from the strong geometric constraint. However,
		the estimates in \eqref{eq:integralparent1}-\eqref{eq:integralparent2} lead to singularities in the relative velocities. Those singularities have to be integrated out either by using available parents  (if any) or by using the Gaussian measure of the velocity distribution at time $0$. The following proposition summarises the different possibilities.
		\begin{prop} \label{prop: singularity}
			(i) Let $q \neq \bar q$ be two particles of velocities $w_q, w_{\bar q}$ with  parent $\ell$. Let~$\zeta\in \mathbb{Z}^d\setminus\{0\}$. Then one has that
			\begin{equation}
			\label{eq:integralgrandparent1}
			\int   \left( {\mathbb{V}\e |\log \e|  \over  |w_q - w_{ \bar q _j}| }\, 
			+\ind_{\sin (w_q - w_{\bar q_j}, \zeta) \leq \e}\right)
			\,   \ind_{B^\ell_{T^>,(q,\bar{q},c)}} \,d\hat x_{\ell} \leq    \frac C {\mu_\e} \,  ^
			\mathbb{V}\e |\log \e| \big( \delta \ind_{\ell = 1} +  t \ind_{\ell \not =  1} \big)\;.
			\end{equation}

			\noindent
			(ii) Let $q , {\bar q}, {q_j}, {\bar q_j}$ be particles with velocities $w_{q}, w_{\bar q}, w_{q_j}, w_{\bar q_j}$ and parent $\ell$ (say deflecting $q$), such that $\left(q,  q_j\right)$ and $\left(\bar q,  \bar q_j\right)$ belong to different connected components of the dynamical graph.
			\begin{equation}\label{eq:integralgrandparent2}
			\begin{split}
			 \int  \min \left(1, {\e  \ind_{\{q, \bar q\} \neq \{q_j, \bar q_j\}}  \over \sin \big(
				w_q - w_{\bar q},w_{q_j} -  w_{\bar q_j}\big)  } \right) \ind_{B^\ell_{T^>,(q,\bar{q},c)}} d\hat x_{\ell} 
			\leq  \frac C {\mu_\e} \,  
			\mathbb{V}\e |\log \e|\left( \delta \ind_{\ell= 1} +  t \ind_{\ell  \neq  1} \right)\quad\\
			\times
			\left( 1 +  {\theta\mathbb{V}\ind_{(q,q_j) \hbox{\rm {\tiny \ encounter at }}\tau_\ell}  \over |  u_{q}  + u_{q_j} - (w_{\bar q_j}+w_{\bar q}) |  }
			+ {t\mathbb{V} \ind_{q = q_j} \ind _{\bar q \neq \bar q_j} \over | w_{\bar q} - w_{\bar q_j}|  }
			\right)\;,
			\end{split}
			\end{equation}
			denoting by $u$ the pre-collisional velocities.
			
			\noindent
			(iii) Let $q , {\bar q}, {q_j}, {\bar q_j}$ be particles with velocities $w_{q}, w_{\bar q}, w_{q_j}, w_{\bar q_j}$ such that $(q), (q_j)$ and $\left(\bar q,  \bar q_j\right)$ belong to different connected components of the dynamical graph. Let $\ell$ be the first parent of~$q , {\bar q}, {q_j}, {\bar q_j}$ deflecting only one particle of the group.
			\begin{equation}
			\label{eq:integralgrandparent14part}
			\int    {\mathbb{V}\e |\log \e|  \over  |  w_q  + w_{q_j} - (w_{\bar q_j}+w_{\bar q}) | }\, \ind_{B^\ell_{T^>,(q,\bar{q},c)}}    \,  d\hat x_{\ell} \leq    \frac C {\mu_\e} \,
			\mathbb{V}\e |\log \e| \big( \delta \ind_{\ell = 1} +  t \ind_{\ell \not =  1} \big)\;.
			\end{equation}
			(iv) For  $q \neq \bar q$,  $\zeta\in \mathbb{Z}^d\setminus\{0\}$
			\begin{equation}
			\label{eq:integralgrandparent1 time 0}
			\begin{aligned}
			\int  M(w_q) M(w_{q_j} )M(w_{\bar q})M(w_{\bar q_j}) \Big({\mathbb{V} \e |\log \e| \over |w_q -  w_{\bar q}|}+ {\mathbb{V}\e |\log \e|  \over  |w_q + w_{q_j}- w_{ \bar q }- w_{\bar q_j}| }+  \ind_{\sin (w_q - w_{\bar q}, \zeta) \leq \e} \quad\\
			+\min \Big(1, {\e \ind_{(q, \bar q) \neq (q_j, \bar q_j)} \over \sin \big(
				w_q - w_{\bar q},w_{q_j} -  w_{\bar q_j}\big) }\Big)  \Big)
			d w_q dw_{q_j} dw_{\bar q} dw_{\bar q_j}\,  \leq     C  \,  
			\mathbb{V} \e | \log \e| \,.
			\end{aligned}
			\end{equation}
		\end{prop} 
		Propositions \ref{prop: recollision} and \ref{prop: singularity} have been proved in \cite{BGSS1}. We can then sum on the $(\hat{x}_1,\cdots,\hat{x}_{n''-1})$ and $(T^>,(q,\bar{q},c))$.		
		\[\sum_{(q,\bar{q},c)} \sum_{T^>}
		\int d\hat x_1 \ind_{B^1_{T^>,(q,\bar{q},c)}} \int d\hat x_{ 2}\,  \dots\int d\hat x_{n_k-1} \ind_{B^{n_k-1}_{T^>,(q,\bar{q},c)}}\]
		
		We integrate the constraints iteratively using successively Propositions \ref{prop: recollision} and \ref{prop: singularity}, one obtain
		\begin{equation*}
		\begin{split}
		\int \ind_{\mathcal{R}_{((s_i,\bar{s}_i),(\kappa_j)_j)}} &M^{\otimes n''}dX_{2,n''}dV_{n''}\\
		&\leq \left(\frac{C}{\mu_\e}\right)^{n''-1}(n'')^{2n''+2}\frac{\delta^{\max(n''-n',1)}}{\max(n''-n',1)!}\frac{\theta^{(n'-n_k-1)_+}}{(n'-n_k-1)_+!}\frac{t^{n_k}}{n_k!}(\mathbb{V}t)^{2d+4}\e|\log\e|\\
		&\leq C'\left (\frac{C'}{\mu_\e}\right)^{n''-1}(n'')^{n''}\delta^{2}\theta^{(n''-n_k-2)_+}t^{n_k+2d+4}\e^\alpha
		\end{split}
		\end{equation*}
		using that $\mathbb{V}:=|\log\e|$.
		
		We obtain the expected result by summing on the \[((s_i,\bar{s}_i)_i,(\kappa_j)_j)\in\{\pm1\}^{2n''-1}\times[0,\gamma-1]^{n''},\]
		and dividing by $n''!$.
	\end{proof}

	\begin{proof}[Proof of \eqref{Estimation avec reco 2}]
		We use first the bound of the previous paragraph 
		\[\begin{split}
		&\Big|\Phi_{\underline{n},n',n''}^{>,k'}[h](Z_{n''})\Phi_{\underline{n},n',n''}^{>,k'}[h](Z_{n''-m+1,2n''-m})\Big|\\
		&\leq \frac{\|h\|^2}{(n''!)^2}
		\sum_{\substack{((s_i,\bar{s}_i)_i,(\kappa_j)_j)\\\kappa_j\leq \gamma-1}}\sum_{\substack{((s'_i,\bar{s}'_i)_i,(\kappa'_j)_j)\\\kappa'_j\leq \gamma-1}}\ind_{\mathcal{R}_{\substack{((s_i,\bar{s}_i)_i,(\kappa_j)_j)}}}(Z_{n''})\ind_{\mathcal{R}_{\substack{((s'_i,\bar{s}'_i)_i,(\kappa'_j)_j)}}}(Z_{n''-m+1,2n''-m}).
		\end{split}\]
		which is invariant under translation. We can fix $x_1=0$ and integrate with respect the over variable. 
		
		\begin{figure}[h]
			\includegraphics[scale=0.28]{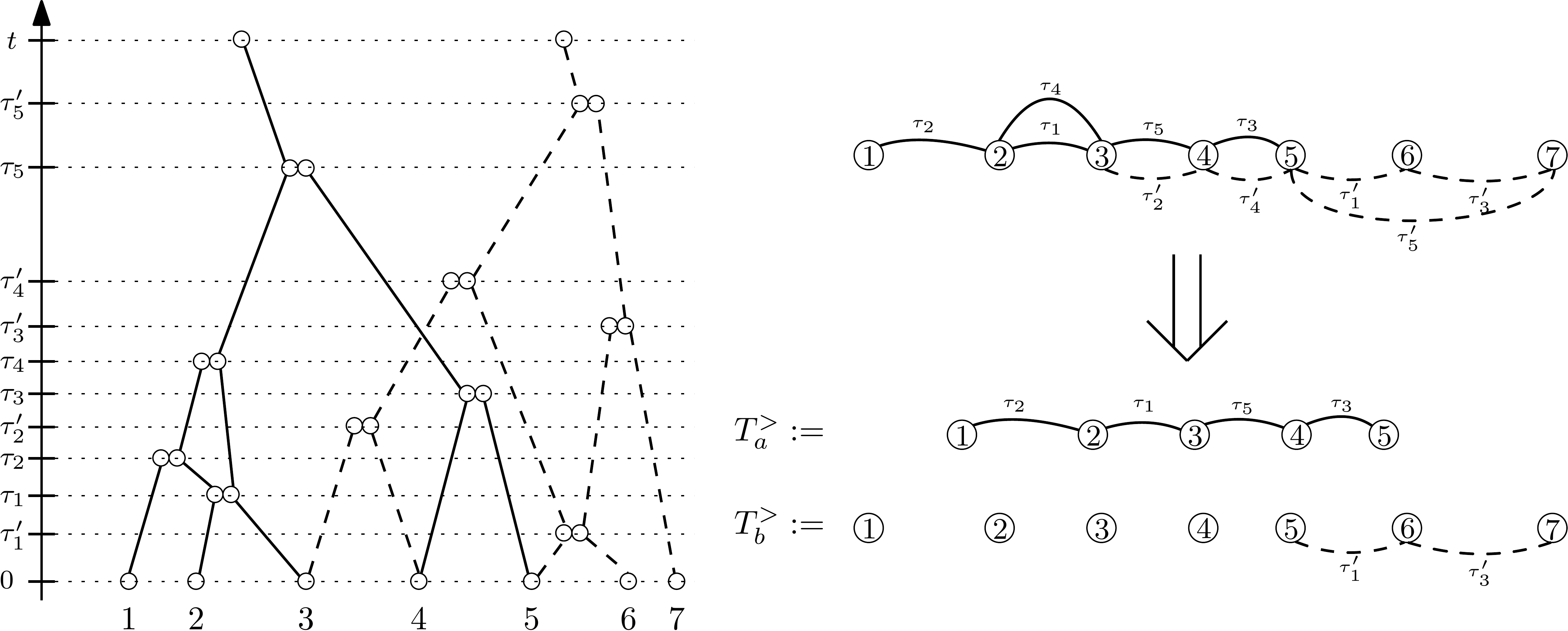}
			\caption{Example of construction of the clustering trees}
		\end{figure}
		
		Fix $((s_i,\bar{s}_i)_i,(\kappa_j)_j)$ and $((s'_i,\bar{s}'_i)_i,(\kappa'_j)_j)$. There are two pseudotrajectories. We construct as in the proof of \eqref{Estimation avec reco 1} the clustering tree $T^>_a$ and the recollisions parameter $(q,\bar{q},c)$ for the first recollision. We construct now the clustering graph $T^>_b$ of $\Zt'(\tau)$ by induction. Let $(\nu_i,\bar{\nu}_i)_{i\leq I}$ be the edges of the collision graph of $\Zt'(\tau)$, with temporal order. We begin by $T_0 = \emptyset$. A the $i$-th step, we add $(\nu_i,\bar{\nu}_i)$ to $T_{i-1}$ only if it does not create a cycle in the graph $T^{>}_a\cup T_{i-1}\cup\{(\nu_i,\bar{\nu}_i)\}$. At the end we have constructed the graph $T_b^>:=T_I$ and $T_a^>\cup T_b^>$ is a simply connected graph which links all the particles. Note that $T_b^>$ has $n''-m$ edges.
		
		We denote $T^>_a:=(\nu_i,\bar{\nu}_i)_{i\in[1,n''-1]}$ and  $T^>_b:=(\nu_i,\bar{\nu}_i)_{i\in[n'',2n''-m-1]}$ (with $\nu_i<\bar{\nu}_i$) and we make the change of variables 
		\[\forall i \in [1,2n''-m-1],~\hat{x}_i:=x_{\nu_i}-x_{\bar{\nu}_i},~X_{2,2n''-m}\mapsto(\hat{x}_1,\cdots,\hat{x}_{2n''-m-1}).\]
		
		We begin by fix $(\hat{x}_1,\cdots,\hat{x}_{n''-1})$ and we construct  construct a sequence of conditioning sets $B_{T^>_b}^i$ as in the proof of \ref{Estimation sans reco 1}. Then we can reproduce the same strategy and
		\[\begin{split}
		\sum_{T_b^>}\int_{B_{T^>_b}^{n''}}&d\hat{x}_{n''}\cdots\int_{B_{T^>_b}^{2n''-m-1}}d\hat{x}_{2n''-m-1}\\
		&\leq \left(\frac{C'(2n''-m)}{\mu_\e}\right)^{n''-m} \left(\|V_{2n''-m}\|^2+2n''-m\right)^{n''-m}\frac{t^{n''-m}}{(n''-m)!}.
		\end{split}\]
		
		In a second time we construct as in the proof of \eqref{Estimation avec reco 1} we construct a sequence of clustering sets $B_{T_a^>,(q,\bar{q},c)}^i$ (for $i\leq n''-1$) of relative position $\hat{x}_i$. Reproducing the same estimations,
		\begin{equation*}\everymath = {\displaystyle}
		\begin{array}{r@{}l@{}r}
		\int &\multicolumn{2}{l}{\ind_{\mathcal{R}_{\substack{((s_i,\bar{s}_i)_i,(\kappa_j)_j)}}}(Z_{n''})\ind_{\mathcal{R}_{\substack{((s'_i,\bar{s}'_i)_i,(\kappa'_j)_j)}}}(Z_{n''-m+1,2n''-m}) M^{\otimes(n''-m)}dX_{2,2n''-m}dV_{2n''-m}}\\[6pt]
		&\multicolumn{2}{l}{\leq \sum_{\substack{(T_a^>,T_b^>)\\(q,\bar{q},c)}}\int M^{\otimes(n''-m)}dV_{2n''-m}}\\[-19pt]
		&~\quad\quad\quad\quad\quad~&\times \int_{B_{T^>,(q,\bar{q},c)}^1}d\hat{x}_{1}\cdots\int_{B_{T^>,(q,\bar{q},c)}^{n''-1}}d\hat{x}_{n''-1}\int_{B_{T^>_b}^{n''}}d\hat{x}_{n''}\cdots\int_{B_{T^>_b}^{2n''-m-1}}d\hat{x}_{2n''-m-1}\\[8pt]
		&\multicolumn{2}{l}{\leq \left(\frac{C}{\mu_\e}\right)^{2n''-m-1}(2n''-m)^{4n''-2m}\frac{\delta^{\max(n''-n',1)}}{\max(n''-n',1)!}\frac{\theta^{(n'-n_k-1)_+}}{(n'-n_k-1)_+!}\frac{t^{n_k+m}}{n_k!(n''-m)!}}\\[10pt]
		&&\times(\mathbb{V}t)^{2d+4}\e|\log\e|\\
		&\multicolumn{2}{l}{\leq C'\left (\frac{C'}{\mu_\e}\right)^{2n''-m-1}(2n''-m)^{2n''-m-1}\delta^{2}\theta^{(n''-n_k-2)_+}t^{n_k+m+2d+4}\e^\alpha.}
		\end{array}
		\end{equation*}
		Where we use that for $(d_1,\cdots,d_k)\in\mathbb{N}^k$,
		\[\frac{1}{d_1!\cdots d_k!}\leq \frac{k^{d_1+\cdots+d_k}}{(d_1+\cdots+d_k)!}\]
		and the Stirling formula. Summing on the $(4\gamma)^{2(n''-1)}$ possible $((s_i,\bar{s}_i)_i,(\kappa_j)_j)$ and $((s'_i,\bar{s}'_i)_i,(\kappa'_j)_j)$ and then dividing by $(n'')!^2$, we obtain the expecting result.
	\end{proof}

	\section{Estimation of pathological recollisions}\label{$L^2$ estimation of the local recollision part}
	
	In the present section is to treat $G_\e^{\text{rec},2}(t)$ defined by
	\begin{equation*}
	\begin{split}
	\sum_{\substack{1\leq k\leq K-1\\1\leq k'\leq K'}} &\sum_{\substack{n_1\leq\cdots\leq n_k\\n_j-n_{j-1}\leq 2^j}}\left(\sum_{\substack{n'\geq n_k}} \mathbb{E}_\e\left[\mu_\e^{-1/2}\sum_{(i_1,\cdots,i_{n'})}\Phi_{\underline{n},n'}^{0,k'}[h]\left(\gr{Z}_{\ui_{n'}}(t_s+\delta)\right)\mathcal{X}_{\ui}\big(\gr{Z}_{\N}(t_s)\big)\,\zeta^0_\e(g)\ind_{\Upsilon_\e}\right]\right.\\
	&-\left.\sum_{n''\geq n'\geq n_k} \mathbb{E}_\e\left[\mu_\e^{-1/2}\sum_{(i_1,\cdots,i_{n''})}\Phi_{n'\leftarrow n''}^{\gamma}\left[\Phi_{\underline{n},n'}^{0,k'}\left[h\right]\right]\left(\gr{Z}_{\ui_{n''}}(t_s)\right)\mathcal{X}_{\ui_{n'}}\big(\gr{Z}_{\N}(t_s)\big)\,\zeta^0_\e(g)\ind_{\Upsilon_\e}\right]\right).
	\end{split}
	\end{equation*}
	
	We will ge the following bound:
	\begin{prop}
		For $\e>0$ small enough, we have
		\begin{equation}\label{Estimation morceau 4}
		\Big|G_\e^{\text{rec},2}(t)\Big|\leq C\|h\| \|g\| \left(K2^{K^2} (Ct)^{2^{K+1}}\right)\e^{a/2}.
		\end{equation}
	\end{prop}
	
	\subsection{Finite parameter expansion}
	In the sums
	\[\sum_{(i_1,\cdots,i_{n'})}\Phi_{\underline{n},n'}^{0,k'}[h]\left(\gr{Z}_{\ui_{n'}}(t_s+\delta)\right)\mathcal{X}_{\ui}\big(\gr{Z}_{\N}(t_s)\]
	and
	\[\sum_{(i_1,\cdots,i_{n''})}\Phi_{n'\leftarrow n''}^{\gamma}\left[\Phi_{\underline{n},n'}^{0,k'}\left[h\right]\right]\left(\gr{Z}_{\ui_{n''}}(t_s)\right)\mathcal{X}_{\ui_{n'}}\big(\gr{Z}_{\N}(t_s)\big),\]
	the indicator function $\mathcal{X}_{\ui_{n'}}(\gr{Z}_{\N})$ depends on all the particles of the system. In addition in the first sum we evaluate a function at time $t_s$ and an other at time $t_s+\delta$. In order to apply usual $L^2$ estimates we have to decompose the terms as sum of  functions evaluated on finitely many parameter

	\subsubsection{Decomposition of $\mathcal{X}_{(i_1,\cdots,i_{n'})}(\gr{Z}_{\N})$}
	We begin by expand  $\mathcal{X}_{(i_1,\cdots,i_{n'})}(\gr{Z}_{\N})$ as a sum of function of finite number of particles and to compute a sum. We can decompose it formally:
	\begin{equation}
	\begin{split}
	\mathcal{X}_{(i_1,\cdots,i_n)}(\gr{Z}_{\N}) &=  1 - \prod_{\substack{\omega\subset\{1,\cdots,\N\}\\\omega\cap\{i_1,\cdots,i_n\}\neq\emptyset}} \big(1-\chi(\gr{Z}_\varpi)\big)\\
	&=-\sum_{p\geq n}\sum_{(i_{n+1},\cdots,i_{p})}\frac{1}{(p-n)!} \sum_{\gr{p}\geq0} \,\sum_{\bar{\varpi}\in\mathcal{Q}^{\gr{p}} _{[1,n],[n+1,p]}}\,\prod_{j=1}^{\gr{p}} \left[-\chi\big(\gr{Z}_{\ui{}_{\varpi_j}}\big)\right]
	\end{split}
	\end{equation}
	where we define for $\omega_1$ and $\omega_2$ two subset of $\mathbb{N}$ with empty intersection 
	\[\mathcal{Q}^{\gr{p}}_{\omega_1,\omega_2} := \left\{(\varpi_1,\cdots,\varpi_{\gr{p}})\Big|\forall i,\,\varpi_i\subset\omega_1\cup\omega_2,\,\varpi_i\cap\omega_1\neq\emptyset\,;\,\omega_2\subset \bigcup_{j= 1}^{\gr{p}} \varpi_j\,;\,\forall i\neq j,~\varpi_i\neq\varpi_j\right\}.\]

	Defining
	\begin{equation}
	\mathfrak{X}_{n,p}\big(\gr{Z}_{\ui_p}\big):=-\frac{1}{(p-n)!} \sum_{\gr{p}\geq0} \,\sum_{\bar{\varpi}\in\mathcal{Q}^{\gr{p}} _{[1,n],[n+1,p]}}\,\prod_{j=1}^{\gr{p}} \left[-\chi\big(\gr{Z}_{\ui{}_{\varpi_j}}\big)\right],
	\end{equation}
	we have for any bounded and measurable function $h_n$
	\[\sum_{(i_1,\cdots,i_n)}h_n(\gr{Z}_{\ui_n}(\tau_1))\mathcal{X}_{\ui_n}(\gr{Z}_{\N}):= \sum_{p\geq n}\sum_{(i_1,\cdots,i_{p})}h_n\big(\gr{Z}_{\ui_n}(\tau_1)\big)\mathfrak{X}_{n,p}\big(\gr{Z}_{\ui_{p}}(\tau_2)\big).\]
	
	For any family  $(\varpi_1,\cdots,\varpi_{\gr{p}})\in \mathcal{Q}_{[1,n],[n+1,p]}^{\gr{p}}$ has all its terms disjoint. Thus ${\gr{p}}$ is smaller than the $|\{\varpi,~\varpi\subset\{1,\cdots,p\}\}| = 2^{p}$ and $|\mathfrak{X}_{n,p}|$ is bounded by $2^{2^{p}}$. This equality holds on $\{\N\leq N\}$ for every $N\in\mathbb{N}$ and he number of particles is bounded on $\Upsilon_{\e}$. So the decomposition is valid on $\Upsilon_{\e}$.

	We use this decomposition:
	\begin{equation}\label{Collision local première partie}
	\begin{split}&\mathbb{E}_\e\left[\mu_\e^{-1/2}\sum_{(i_1,\cdots,i_{n''})}\Phi_{n'\leftarrow n''}^{\gamma}\left[\Phi_{\underline{n},n'}^{0,k'}\left[h\right]\right]\left(\gr{Z}_{\ui_{n''}}(t_s)\right)\mathcal{X}_{\ui_{n'}}\big(\gr{Z}_{\N}(t_s)\big)\,\zeta^0_\e(g)\ind_{\Upsilon_\e}\right]\\
	&=\sum_{p\geq n''}\mathbb{E}_\e\left[\mu_\e^{-1/2}\sum_{(i_1,\cdots,i_{p})}\Phi_{n'\leftarrow n''}^{\gamma}\left[\Phi_{\underline{n},n'}^{0,k'}[h]\right]\left(\gr{Z}_{\ui_{n''}}(t_s)\right)\mathfrak{X}_{n',p}(\gr{Z}_{\ui_{[1,n']\cup[n''+1,p]}}(t_s))\,\zeta^0_\e(g)\ind_{\Upsilon_\e}\right]
	\end{split}\end{equation}
	and
	\begin{equation}\begin{split}
	\mathbb{E}_\e&\left[\mu_\e^{-1/2}\sum_{(i_1,\cdots,i_{n'})}\Phi_{\underline{n},n'}^{0,k'}[h]\left(\gr{Z}_{\ui_{n'}}(t_s+\delta)\right)\mathcal{X}_{\ui_{n'}}\big(\gr{Z}_{\N}(t_s)\big)\,\zeta^0_\e(g)\ind_{\Upsilon_\e}\right]\\
	&=\sum_{p\geq 0}\mathbb{E}_\e\left[\mu_\e^{-1/2}\sum_{(i_1,\cdots,i_{n'+p})}\Phi_{\underline{n},n'}^{0,k'}[h]\left(\gr{Z}_{\ui_{n'}}(t_s+\delta)\right)\mathfrak{X}_{n',p}(\gr{Z}_\ui(t_s))\,\zeta^0_\e(g)\ind_{\Upsilon_\e}\right].
	\end{split}\end{equation}
	
	\subsubsection{Dynamical cluster development}
	In the second member of $G_\e^{\text{rec},2}$ we look at function at time $t_s$ and $t_s-\delta$. To come back to one single evaluation time we have to do some pseudotrajectory development. But tree pseudotrajectories are not adapted since we are precisely where a lot of local recollision happened. So we use an other kind of pseudotrajectory development: dynamical cluster development (see \cite{Sinai} for more details).
	
	We denote $\Zc^\lambda(\tau) = (\mathsf{X}^\lambda(\tau),\mathsf{V}^\lambda(\tau))$ the trajectory of the particles $\lambda$ in hard sphere dynamics -isolated of the other particles- with initial data $Z_\lambda$. For any subset $\lambda'\subset\lambda$, $\Zc_{\lambda'}^\lambda(\tau)$ is the trajectory of particles $\lambda'$ in $\Zc^\lambda(\tau)$.
	
	We say that $\Zc^\lambda(\tau)$ forms a cluster if the collision graph on time interval $[0,\delta]$ is connected and $\varphi_{|\lambda|}(Z_\lambda)$ the indicator function that the trajectory $\Zc^\lambda(\tau)$ form a cluster. In the same way, for $\lambda'\subset\lambda$, $\Zc^\lambda(\tau)$ form a $\lambda'$-cluster if in the collision of $\Zc^\lambda(\tau)$, all the particles are in the same connected components than one of the particles of $\lambda'$. The function $\varphi_{|\lambda|}^{\lambda'}(Z_\lambda)$ is equal to $1$ if $\Zc^\lambda(\tau)$ is a $\lambda'$-cluster, $0$ else.
	
	We say that trajectories $\Zc^{\lambda}(\tau)$ and $\Zc^{\lambda'}(\tau)$ (with $\lambda\cap\lambda'=\emptyset$) have an overlap if there exits a couple of particle $(i,i')\in\lambda\times\lambda'$ and some time $\tau\in[0,\delta]$, $|\mathsf{x}_i^\lambda(\tau)-\mathsf{x}_j^{\lambda'}(\tau)|\leq \e$. Then we denote $\lambda\so\lambda'$.
	
	For $(Z_{\lambda_1},\cdots,Z_{\lambda_{\gr{l}}})\in \prod_{i=1}^{\gr{l}}\mathcal{D}^{|\lambda_i|}_\e$ initial data, we look at the indicator function that for any $i\neq j$, $\Zc^{\lambda_i}(\tau)$ and $\Zc^{\lambda_j}(\tau)$ has no overlap. As in section \ref{Quasi-orthogonality estimates} we can expand the function:
	\begin{equation}\begin{split}
	\prod_{1\leq i<j\leq \gr{l}}\big(1-\ind_{\lambda_i\so \lambda_j}\big)=\sum_{\substack{\omega\subset[1,l]\\1\in\omega}}&\underset{:=\psi_{|\omega|}(Z_{\lambda_1},Z_{\lambda_{\omega(2)}},\cdots,Z_{\lambda_{\omega(|\omega|)}})}{\underbrace{\sum_{C\in\mathcal{C}(\omega)}\prod_{(i,j)\in E(C)}-\ind_{\lambda_i\so \lambda_j}}}\prod_{\substack{(i,j)\in(\omega^c)^2\\i\neq j}}\big(1-\ind_{\lambda_i\so \lambda_j}\big).\\
	&
	\end{split}\end{equation}
	We have defined $(\psi_l)_l$ the cumulants of the overlap indicator. We make a partition of $\mathcal{D}_\e$ depending on particles interacting on the the time interval $[0,\delta]$: fixing $\N\in\mathbb{N}$ and $\ui_m$,
	\[\begin{split}
	&h_m(\gr{Z}_{\ui_m}(\delta))=\sum_{\gr{l}=1}^\N \sum_{\substack{\ui_m\subset\lambda_1\\ (\lambda_2,\cdots,\lambda_{\gr{l}})\in\mathcal{P}_{\lambda_1^c}^{\gr{l}-1}}}h_m(\gr{Z}_{\ui_m}(\delta))\varphi^{\ui_m}_{\lambda_1}(\gr{Z}_{\lambda_1})\prod_{i=2}^{\gr{l}}\varphi_{|\lambda_i|}(\gr{Z}_{\lambda_i})\prod_{1\leq i<j\leq \gr{l}}\big(1-\ind_{\lambda_i\so \lambda_j}\big)\\
	&=\sum_{\gr{l}=1}^\N \sum_{\substack{\ui_m\subset\lambda_1\\ (\lambda_2,\cdots,\lambda_{\gr{l}})\in\mathcal{P}_{\lambda_1^c}^{\gr{l}-1}}}h_m(\gr{Z}_{\ui_m}(\delta))\varphi^{\ui_m}_{\lambda_1}(\gr{Z}_{\lambda_1})\prod_{i=2}^{\gr{l}}\varphi_{|\lambda_i|}(\gr{Z}_{\lambda_i})\sum_{\substack{\omega\subset[1,l]\\1\in\omega}}\psi_{|\omega|}(\gr{Z}_{\underline{\lambda}_\omega})\prod_{\substack{(i,j)\in(\omega^c)^2\\i\neq j}}\big(1-\ind_{\lambda_i\so \lambda_j}\big).
	\end{split}\]
	where we have denoted $\mathcal{P}^r_\omega$ the set of the unordered partitions $(\rho_1,\cdots,\rho_r)$ of the set $\omega$.
	
	We make the change of variable
	\[\left(\gr{l},\left(\lambda_1,\cdots,\lambda_{\gr{l}}\right),\omega\right)\mapsto\left(\rho,\gr{l}_1,\left(\bar{\lambda}_1,\cdots,\bar{\lambda}_{\gr{l}_1}\right),\gr{l}_2,\left(\tilde{\lambda}_1,\cdots,\tilde{\lambda}_{\gr{l}_2}\right)\right)\]
	where
	\[\rho:=\bigcup_{i\in\omega}\lambda_i,~\gr{l}_2:=|\omega|,~\gr{l}_1 := \gr{l}-|\omega|,~\left(\bar{\lambda}_1,\cdots,\bar{\lambda}_{\gr{l}_1}\right):=(\lambda_j)_{j\in \omega^c} {\rm~and~}\left(\tilde{\lambda}_1,\cdots,\tilde{\lambda}_{\gr{l}_2}\right):=(\lambda_j)_{j\in \omega}\]
	
	The set $\rho$ is the set of particle which interact (in the dynamic or {\it via} an overlap) in $\ui_m$. Thus
	\[\begin{split}
	h_m(\gr{Z}_{\ui_m}(\delta))	=\sum_{\ui_m\subset\rho}\sum_{\gr{l}_1=1}^{|\rho|} \sum_{\substack{\ui_m\subset\bar{\lambda}_1\subset \rho\\ (\bar{\lambda}_2,\cdots,\bar{\lambda}_{\gr{l}_1})\in\mathcal{P}_{\bar{\lambda}^{c}_1}^{\gr{l}_1-1}}}h_m\left(\Zc^{\bar{\lambda}_1}_{\ui_m}(\delta)\right)\varphi^{\ui_m}_{\bar{\lambda}_1}\left(\gr{Z}_{\bar{\lambda}_1}\right)\prod_{i=2}^{\gr{l}_1}\varphi_{|\bar{\lambda}_i|}\left(\gr{Z}_{\bar{\lambda}_i}\right)\psi_{\gr{l}_1}\!\left(\gr{Z}_{\bar{\lambda}_1},\cdots,\gr{Z}_{\bar{\lambda}_{\gr{l}_1}}\right)\\[-10pt]
	\times\sum_{\gr{l}_2=1}^{|\rho^c|}\sum_{(\tilde{\lambda}_1,\cdots,\tilde{\lambda}_{\gr{l}_2})\in\mathcal{P}^{\gr{l}_2}_{\rho^c}}\prod_{i=1}^{\gr{l}_2}\varphi_{|\tilde{\lambda}_i|}(\gr{Z}_{\tilde{\lambda}_i})\!\!\prod_{\substack{(i,j)\in(\omega^c)^2\\i\neq j}}\!\!\big(1-\ind_{\tilde{\lambda}_i\so \tilde{\lambda}_j}\big).
	\end{split}\]
	The second line is the sum on all possible partition $(\tilde{\lambda}_1,\cdots,\tilde{\lambda}_{\gr{l}_2})$ of $\rho^c$ of the indicator function that they are effectively the dynamical cluster of the initial data. Hence it is equal to one. Thus defining the $n$-th \emph{dynamical cumulant} as
	\begin{equation}
	\begin{split}
		\mathsf{f}_{m\leftarrow n}[h_m](Z_n):=\frac{1}{(n-m)!}\sum_{{\gr{l}}=1}^n \sum_{\substack{\lambda_1 \subset [1,n]\\ [1,m] \subset \lambda_1}}\sum_{\substack{(\lambda_2,\cdots,\lambda_\gr{l})\\ \in\mathcal{P}^{\gr{l}-1}_{\lambda_1^c}}} h_m(\Zc^{\lambda_1}_{[1,m]}(\delta)) \psi_l(Z_{\lambda_1},\cdots,Z_{\lambda_{\gr{l}}})\\[-20pt]
		\times\varphi_{|\lambda_1|}^{[1,m]}(Z_{\lambda_1})\prod_{i=2}^{\gr{l}}\varphi_{|\lambda_i|}(Z_{\lambda_i}),
		\end{split}
	\end{equation}
	we obtain the dynamical cluster expansion:
		\begin{theorem}
		For almost all $\gr{Z}_\N\in\mathcal{D}_\e$ we have 
		\begin{equation}
		h_m\left((\gr{Z}_{\ui_m}(\delta)\right) = \sum_{n\geq m} \sum_{(i_{m+1},\cdots,i_{n})} \mathsf{f}_{m\leftarrow n}[h_m]\left(\gr{Z}_{\ui_n}(0)\right).
		\end{equation}
	\end{theorem}

	Applying it to \eqref{Collision local première partie}:
	\begin{equation}
	\begin{split}
	\mathbb{E}_\e\Bigg[\mu_\e^{-1/2}&\sum_{(i_1,\cdots,i_{n'})}\Phi_{\underline{n},n'}^{0,k'}[h]\left(\gr{Z}_{\ui_{n'}}(t_s+\delta)\right)\mathcal{X}_{\ui_{n'}}\big(\gr{Z}_{\N}(t_s)\big)\,\zeta^0_\e(g)\ind_{\Upsilon_\e}\Bigg]\\
	&=\sum_{l\geq p\geq n'}\mathbb{E}_\e\left[\mu_\e^{-1/2}\sum_{(i_1,\cdots,i_{l})}\,\mathsf{f}_{p\leftarrow l}\Big[\Phi_{\underline{n},n'}^{0,k'}[h]\Big]\left(\gr{Z}_{\ui_l}(t_s)\right)\mathfrak{X}_{n',p}(\gr{Z}_{\ui_{p}}(t_s))\,\zeta^0_\e(g)\ind_{\Upsilon_\e}\right]
	\end{split}
	\end{equation}
	where $\mathsf{f}_{p\leftarrow l}\big[\Phi_{\underline{n},n'}^{0,k'}\big]$ is the dynamical cumulant of $Z_{p}\mapsto\Phi_{\underline{n},n'}^{0,k'}(Z_{[1,n']})$.
	
	Finally we symmetrize these two functions:
	\begin{equation}\Phi_{\underline{n},n',p,l}^r(Z_{l}) := \frac{1}{l!}\sum_{\sigma\in\mathfrak{S}_{l}}\mathsf{f}_{p\leftarrow l}\big[\Phi_{\underline{n},n'}^{0,k'}[h]\big]\left(Z_{\sigma([1,l])}\right)\mathfrak{X}_{n',p}(Z_{\sigma([1,p])})\end{equation}
	\begin{equation}\Phi_{\underline{n},n',n'',p}^{k'}(Z_{p}) := \frac{1}{p!}\sum_{\sigma\in\mathfrak{S}_{p}}\Phi^{\gamma}_{n'\leftarrow n''}\left[\Phi_{\underline{n},n'}^{0,k'}[h]\right]\left(Z_{\sigma([1,n''])}\right)\mathfrak{X}_{n',p}(Z_{\sigma([1,n']\cup[n''+1,p])})\end{equation}
	
	We have rewrite $G_\e^{\text{rec},2}(t)$ as function evaluated on finitely many variable:
	\begin{equation}
	\begin{split}
	G_\e^{\text{rec},2}(t)	&= \sum_{\substack{1\leq k\leq K-1\\1\leq k'\leq K'}} \sum_{\substack{n_1\leq\cdots\leq n_k\\n_j-n_{j-1}\leq 2^j}}\sum_{n'\geq0}\Bigg(\sum_{\substack{l\geq0\\p\geq0}}\mathbb{E}_\e\Bigg[\mu_\e^{-1/2}\sum_{(i_1,\cdots,i_l)}\Phi_{\underline{n},n',p,l}^{k'}(\gr{Z}_{\ui_l}(t_s))\,\zeta^0_\e(g)\ind_{\Upsilon_\e}\Bigg]\\
	&~~~~~~~~~~~~~~~~~~~~~~~~~~~~~-\sum_{\substack{n''\geq 0\\p\geq0}} \mathbb{E}_\e\Bigg[\mu_\e^{-1/2}\sum_{(i_1,\cdots,i_p)}\Phi_{\underline{n},n',n'',p}^{k'}(\gr{Z}_{\ui_p}(t_s))\,\zeta^0_\e(g)\ind_{\Upsilon_\e}\Bigg]\Bigg).
	\end{split}
	\end{equation}
	\subsection{Geometrical estimation of local recollisions}
	The aim of this part is to prove the following bound on $\Phi^{k'}_{\underline{n},n',n'',p}$ and $\Phi^{k'}_{\underline{n},n',p,l}$:
	\begin{prop}
		Fix $n_1\leq\cdots\leq n_k\leq n'\leq n''<p$, and for $m\in \{1,\cdots p\}$ we have 
		\begin{equation}\label{Estimation reco loc 1}
		\int_{x_1=0} \sup_{y\in\Lambda}\big|\Phi_{\underline{n},n',n'',p}^{k'}(\tau_yZ_p)\big|M^{\otimes p}dX_{2,p}dV_{p}\leq \frac{\|h\|}{\mu_\e^{p-1}} C^{p}\delta^2\e^\alpha \theta^{(p-n_k-2)_+}t^{n_k-1},
		\end{equation}
		\begin{equation}\label{Estimation reco loc 2}
		\begin{split}
		\int_{x_1=0}\sup_{y\in\Lambda} \big|\Phi_{\underline{n},n',n'',p}^{k'}(\tau_yZ_p)\Phi_{\underline{n},n',n'',p}^{k'}(\tau_yZ_{p+1-m,2p-m})\big|M^{\otimes (2p-m)}dX_{2,2p-m}dV_{p-m}\\
		\leq \frac{\|h\|^2}{p^m\mu_\e^{2p-m-1}} C^{p}\delta^2\e^\alpha \theta^{(p-n_k-2)_+}t^{n_k-1+p-m}.
		\end{split}
		\end{equation}
		In the same way if we fix $n_1\leq\cdots\leq n_k\leq n'\leq n''<p\leq l$ and for $m\in \{1,\cdots l\}$ we 
		\begin{equation}\label{Estimation reco loc 3}
		\int_{x_1=0}\sup_{y\in\Lambda} \big|\Phi_{\underline{n},n',p,l}^{k'}(\tau_yZ_{l})\big|M^{\otimes l}dX_{l-1}dV_{l}\leq \frac{\|h\|}{\mu_\e^{l-1}} C^{l}\delta^2\e^\alpha \theta^{(l-n_k-2)_+}t^{n_k-1},
		\end{equation}
		\begin{equation}\label{Estimation reco loc 4}
		\begin{split}
		\int_{x_1=0}\sup_{y\in\Lambda} \big|\Phi_{\underline{n},n',p,l}^{k'}(\tau_yZ_{p})\Phi_{\underline{n},n',p,l}^{k'}(\tau_yZ_{l+1-m,2l-m})\big|M^{\otimes (2l-m)}dX_{2,2l-m}dV_{2l-m}\\
		\leq \frac{\|h\|^2}{l^m\mu_\e^{2l-m-1}} C^{l}\delta^2\e^\alpha \theta^{(l-n_k-2)_+}t^{n_k-1+l-m}.
		\end{split}
		\end{equation}
	\end{prop}
	
	Then using the quasi-orthongonality estimates we obtain:
	\begin{equation}\everymath={\displaystyle}\begin{array}{r@{}r@{}r}
		\Bigg|\mathbb{E}_\e\Bigg[\mu_\e^{-1/2}\sum_{\ui_p}\Phi_{\underline{n},n',n'',p}^{k'}(\gr{Z}_{\ui_p}(t_s))\,\zeta^0_\e(g)\ind_{\Upsilon_\e}\Bigg]\Bigg)\Bigg|&\multicolumn{2}{l}{\leq\|h\| \|g\| C^{p}\Big(\delta^2\e^\alpha\theta^{(p-n_{k}-2)_+}t^{n_k-1}\e^{\frac12}}\\[-8pt]
		&&+\Big(p\delta^2\e^\alpha \theta^{(p-n_k-2)_+}t^{n_k-1+p}\Big)^{\frac12}\Big)\\[8pt]
		&\multicolumn{2}{l}{\leq\delta\e^{\alpha/2}\|h\| \,\|g\| C^{p} (\theta t)^{(p-n_k-2)_+/2}t^{n_k}}
	\end{array}\end{equation}
	and in the same way
	\begin{equation}\Bigg|\mathbb{E}_\e\Bigg[\mu_\e^{-1/2}\sum_{\ui_l}\Phi_{\underline{n},n',p,l}^r(\gr{Z}_\ui(t_s))\,\zeta^0_\e(g)\ind_{\Upsilon_\e}\Bigg]\Bigg)\Bigg|\leq\delta\e^{a/2}\|h\|\|g\| C^{l} (\theta t)^{(l-n_k-2)_+/2}t^{n_k}.\end{equation}
	
	Because $\theta$ tends to $0$ as $\e$ goes to zero, for $\e$ small enough, the two previous series are sommable with respect to respectively $(l,n'',n')$ and $(l,p,n')$. We recall that $K'=\theta/\delta$ and we sum on $k$, $\underline{n}$ and $r$ to obtain that there exists a positive constant $C$ depending only on the dimension and $\gamma$ such that
	\begin{equation}\begin{split}\Big|G_\e^{\text{rec},2}(t)\Big|\leq C  \delta\e^{a/2}\|h\| \|g\|\frac{\theta}{\delta} \sum_{k =1}^K\sum_{\substack{n_1\leq\cdots\leq n_k\\n_j-n_{j-1}\leq 2^j}} (Ct)^{n_k}&\leq C\e^{a/2}\|h\| \|g\| \sum_{k =1}^K2^{k^2}  (Ct)^{2^{k+1}}\\
	 &\leq C\|h\| \|g\| \left(K2^{K^2} (Ct)^{2^{K+1}}\right)\e^{a/2}
	 \end{split}\end{equation}
	 which conclude the proof of \eqref{Estimation morceau 4}.

	We have almost to prove four times the same inequality. However there are some little difference and we will do in detail only the first one and then explain how to adapt it.
	\begin{proof}[Proof of \eqref{Estimation reco loc 1}]
		We recall that
			\begin{equation*}\Phi_{\underline{n},n',n'',p}^{k'}(Z_{p}) := \frac{1}{p!}\sum_{\sigma\in\mathfrak{S}_{p}}\Phi^{\gamma}_{n'\leftarrow n''}\left[\Phi_{\underline{n},n'}^{0,k'}[h]\right]\left(Z_{\sigma([1,n''])}\right)\mathfrak{X}_{n',p}(Z_{\sigma([1,n']\cup[n''+1,p])}).\end{equation*}
		
		In $\Phi_{n'\leftarrow n''}^\gamma\left[\Phi_{\underline{n},n'}^{k'}[h]\right]\left(Z_{n''}\right)\mathfrak{X}_{n',p}(Z_{[1,n']\cup[n''+1,p]})$ we see three sets of indices:
		\begin{itemize}
			\item $[1,n']$ the set of particles in "final" tree pseudotrajectories development,
			\item $[n'+1,n'']$ the particles added in the local tree development,
			\item  $[n''+1,p]$ the particles which make local recollision.
		\end{itemize}
		Any permutation $\sigma$ which sends $[1,n']$, $[n'+1,n'']$ and $[n''+1,p]$ onto themselves stabilizes $\Phi_{n'\leftarrow n''}^\gamma\left[\Phi_{\underline{n},n'}^{k'}[h]\right]\left(Z_{n''}\right)\mathfrak{X}_{n',p}(Z_{[1,n']\cup[n''+1,p]})$ and
		\[\Phi_{\underline{n},n',n'',p}^r(Z_{p}) = \frac{n'!\,(n''-n')! \,(p-n'')!}{p!} \sum_{\substack{\underline{\omega}\in\mathcal{P}^3_{p}\\|\omega_1|=n'\\|\omega_2| = p-n''}}\Phi_{n'\leftarrow n''}^\gamma\left[\Phi_{\underline{n},n'}^{k'}\right]\left(Z_{\omega_1},Z_{\omega_3}\right)\mathfrak{X}_{n',p}(Z_{\omega_1},Z_{\omega_2}).\]
		Let develop $\Phi_{n'\leftarrow n''}^\gamma\left[\Phi_{\underline{n},n'}^{k'}[h]\right]$ and $\mathfrak{X}_{n',l}$. For $((s_i,\bar{s}_i),(\kappa_i))$ a set of recollision parameters, we denote $\mathcal{R}_{((s_i,\bar{s}_i),(\kappa_j))}\subset\mathcal{D}_\e^{n''}$ the set of initial data such that there is 
		\begin{itemize}
			\item $n'$ particles at time $\delta$,
			\item $n_k$ particles at time $(k'+1)\delta$ and
			\item $n_j$ particles at time $(k'+1)\delta + (k-j)\theta$.
		\end{itemize}
		Then
		\begin{equation}\begin{split}\Phi_{\underline{n},n',n'',p}^r(Z_{p}) 
		= \frac{1}{p!}\sum_{((s_i,\bar{s}_i),(\kappa_j))}~\sum_{\substack{ \omega_1\sqcup \omega_2 \sqcup \omega_3= [p]\\|\omega_1|=n'\\|\omega_2| = p-n''}}&~\sum_{\gr{p}\geq1}\sum_{\underline{\varpi}\in\mathcal{Q}^\gr{p}_{\omega_1,\omega_2}}-\ind_{\mathcal{R}_{((s_i,\bar{s}_i),(\kappa_j))}}(Z_{\omega_1\cup\omega_3})\\[-20pt]
		&\times h\left(\Zt_{\omega_1\cup\omega_3}(t-t_s)\right)\prod_{i=1}^{n''-1}\bar{s}_i\prod_{i=1}^\gr{p}\left(-\chi(Z_{\varpi_i})\right).
		\end{split}
		\end{equation}
		where 
		\[\Zt_{\omega_1\cup\omega_3}(t-t_s):= \Zt(t-t_s,((s_i,\bar{s}_i),(\kappa_j)),Z_{\omega_1\cup\omega_3})\]
		and we have the estimation
		\begin{equation}\label{borne 1 sur phi l}\begin{split}
		\Big|\Phi_{\underline{n},n',n'',p}^{k'}\Big|(Z_{p}) 
		=\frac{\|h\|}{p!} \sum_{((s_i,\bar{s}_i),(\kappa_j))}\sum_{\substack{ \omega_1\sqcup \omega_2 \sqcup \omega_3= [p]\\|\omega_1|=n'\\|\omega_2| = p-n''}}~\sum_{\gr{p}\geq 1}\sum_{\underline{\varpi}\in\mathcal{Q}^{\gr{p}}_{\omega_1,\omega_2}}\ind_{\mathcal{R}_{((s_i,\bar{s}_i),(\kappa_j))}}(Z_{\omega_1\cup\omega_3})\\[-25pt]
		\times\prod_{i=1}^{\gr{p}}\chi(Z_{\varpi_i}).\end{split}\end{equation}

		Note that the right hand side is invariant under translation. Thus one can fix $x_1 = 0$ and integrate with respect the other variables.
		
		The set of parameters $\mathcal{Q}^{\gr{p}}_{\omega_1,\omega_2}$ is huge and we need the global conditioning to control the number of acceptable $\underline{\varpi}$.
		
		For parameters $Z_{l}$ we introduce $\underline{\rho}:=(\rho_1,\cdots,\rho_{\gr{r}})$ the $\delta\mathbb{V}$-distance partition: consider the graph $G$ with vertices $[1,p]$ with $(i,j)\in E(G)$ if and only if $|x_i-x_j|\leq 2\delta\mathbb{V}$. The $\rho_i$ are the connected components of $G$. We define $\mathcal{D}^{\underline{\rho}}_\e\subset\mathcal{D}^{p}_\e$ the set such that $\underline{\rho}$ is the distance partition, the $\big(\mathcal{D}^{\underline{\rho}}_\e\big)_{\underline{\rho}}$ form a partition of $\mathcal{D}^{p}_\e$. 
		
		Inside each cluster $\rho_i$, particles  can only interact with  the other particle as long the kinetic energy $\|V_{\rho_i}(\tau)\|^2$ is bounded by $\mathbb{V}^2$. Hence the system $\rho^i$ is isolated on $[0,\delta]$ and for any $\varpi\subset\omega_1\cup\omega_3$, if particles in $Z_\omega$ can have a pseudotrajectory with connected collision graph (and a local recollision), then there some $\rho_i$ containing $\omega$.
		
		We can do now the following parametrisation: for any $\rho_i$, we consider
		\begin{itemize}
			\item $\underline{\omega}^i := (\omega_1^i,\omega_2^i,\omega_3^i)$ the partition of $\rho_i$ defined by $\omega^i_j:=\omega_j\cap \rho_i$,
			\item$\underline{\varpi}^i:=\{\varpi_j {\rm ~such~that~}\varpi_j\subset\rho_i\}$,
			\item  $\mathfrak{p}_i:=(\underline{\omega}^i,\underline{\varpi}^i)$, and $\mathfrak{P}(\rho_i)$ the set of possible $\mathfrak{p}_i$.
		\end{itemize}
	Because $\rho_i$ is of size at most $\gamma$, there exists a constant $C_\gamma$ depending only on $\gamma$ such that $|\mathfrak{P}(\rho_i)|\leq C_\gamma$. Any particles in $\omega_2$ or $\omega_3$ has to be close to a particle in $\omega_1$ because they are in some pseudotrajectories on $[0,\delta]$ implying a particle in $\omega_1$. So for any $\rho_i$, $\omega_1^i$ is not empty. Finally note that if we fix $\underline{\rho}$, the map $(\underline{\omega},\underline{\varpi})\mapsto (\mathfrak{p}_i)_i$ is onto. 
	
	We have now the following bound
	\begin{equation*}\big|\Phi_{\underline{n},n',n'',p}^r(Z_{p})\big|
	\leq \frac{\|h\|}{p!}\sum_{{\gr{r}}=1}^{p}\sum_{\underline{\rho}\in\mathcal{P}^{\gr{r}}_{p}}\sum_{\substack{((s_i,\bar{s}_i),(\kappa_j))\\ \underline{\mathfrak{p}}\in\underset{i}{\prod}\mathfrak{P}(\rho_i)}}~\ind_{\mathcal{R}^{\underline{\rho},\underline{\mathfrak{p}}}_{((s_i,\bar{s}_i),(\kappa_i)}}(Z_{p})\prod_{i=1}^{\gr{r}}\Delta_{\mathfrak{p}_i}(Z_{\rho_i})
	\end{equation*}
	where the function
	\[\Delta_{\mathfrak{p}_i}(Z_{\rho_i}):=\ind_{\!\!\!\!\tiny\begin{array}{l}Z_{\rho_i}\text{\,form\,a}\\ \text{distance\,cluster}\end{array}}\prod_{j=1}^{|\underline{\varpi}^i|}\chi(Z_{\varpi_j^i})\]
	control the local cluster and 
	\[\mathcal{R}^{\underline{\rho},\underline{\mathfrak{p}}}_{((s_i,\bar{s}_i),(\kappa_i)} :=  \Big\{Z_{p}\in\mathcal{D}^{\underline{\rho}}_\e,\,Z_{\omega_1\cup\omega_3}\in\mathcal{R}_{((s_i,\bar{s}_i),(\kappa_i)}\Big\}.\]
	
	We use the same method than in \cite{BGSS} to control the first condition.
	
	For pseudotrajectories $\Zt_{\omega_1\cup\omega_3}(\tau)$, we consider its collision graph $\mathcal{G}_{\omega_1\cup\omega_3}^{[0,t-t_s]}$. Then we construct the graph $G$ by identifying in $\mathcal{G}_{\omega_1\cup\omega_3}^{[0,t-t_s]}$ the particles in a same cluster $\rho_i$. Finally we construct the \emph{clustering trees} $T^>:=(\nu_i,\bar{\nu}_i)_{1\leq i\leq \gr{r}-1}$ where the $i$-th clustering collision in $G$ happens between cluster $\rho_{\nu_i}$ and $\rho_{\bar{\nu_i}}$.
	
	\begin{figure}[h]
		\centering
		\includegraphics[scale=0.1]{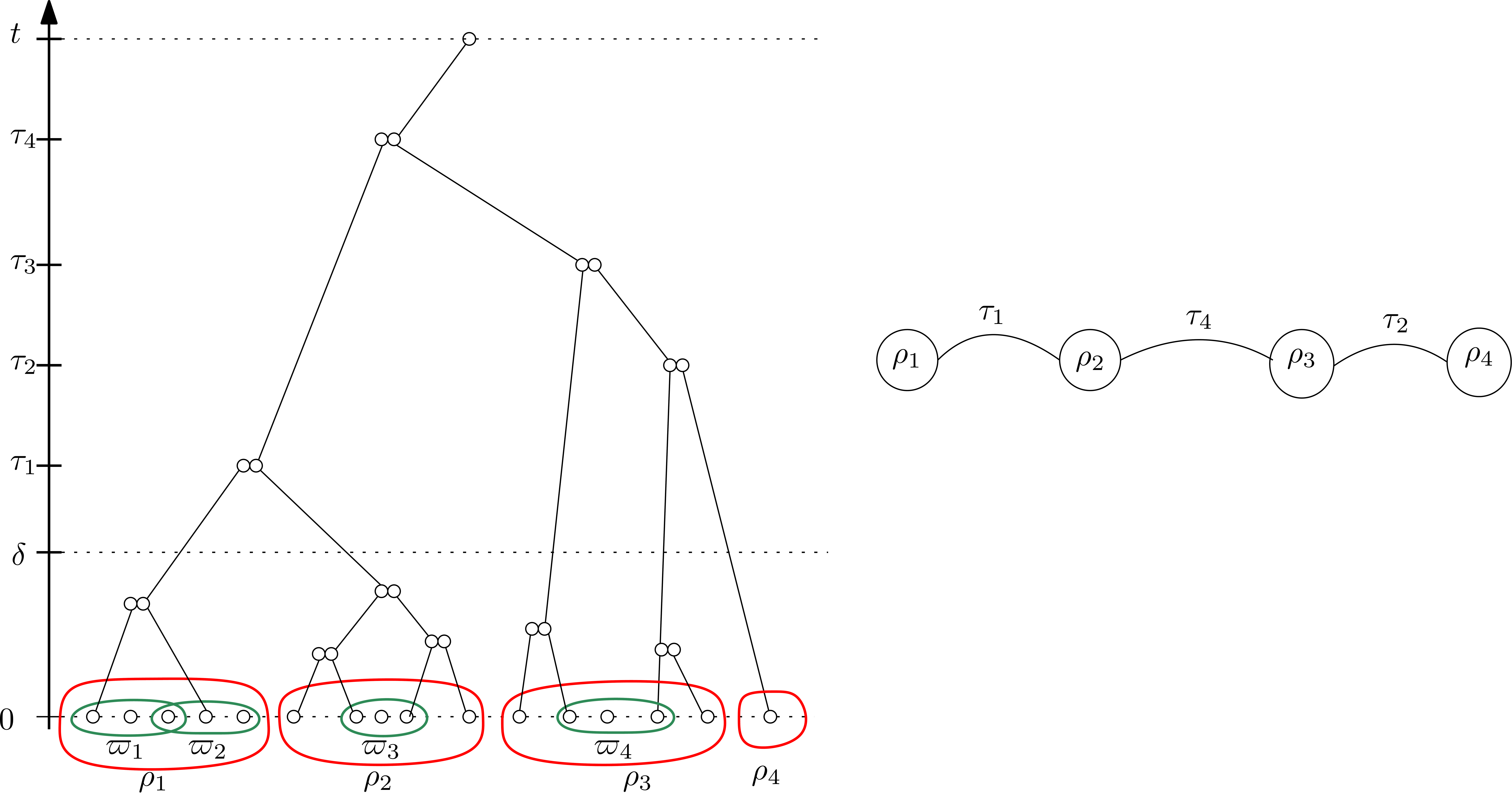}
		\caption{Example of construction of the clustering set.}
	\end{figure}
	
	We need to count the number of clustering collision of $T^>$ happening between time $\delta$ and time $\theta$. If $\gr{r}>n_k$, all the $\gr{r}-1$ collisions in $T^>$ cannot correspond to the $n_k-1$ annihilations of the time interval $[(k'+1)\delta,t-t_s]$. Thus at least $(\gr{r}-n_k)_+$ collision happen in $[\delta,(k'+1)\delta]\subset[0,2\theta]$.
	
	We construct now an other representation of collision graph. Let $L_0$ be equal to $\{\{1\},\cdots,\{\gr{r}\}\}$ and we construct the $L_i$ and $(\nu_{(i)},\bar{\nu}_{(i)})$ sequentially. Suppose that $L_{i-1}=(c_1,\cdots,c_l)$, the $(c_j)$ forming a partition of $[1,r]$. The $i$-th collision happens between cluster $\nu_i\in c_a$ and $\bar{\nu}_i\in c_b$. Then we do the following construction:
	\begin{itemize}
		\item $L_i:=\big(L_{i-1}\setminus\{c_a,c_b\}\big)\cup\{c_a\cup c_b\}$,
		\item $\{\nu_{(i)},\bar{\nu}_{(i)}\}:=\{c_a,c_b\}$ with $\max \nu_{(i)} <\max \bar{\nu}_{(i)}$.
	\end{itemize}
	The $(\nu_{(i)},\bar{\nu}_{(i)})$ define a partition of $\mathcal{T}^>_{\gr{r}}$ (the set of ordered trees on $[1,\gr{r}]$).
	
	\begin{figure}[h]
		\centering
		\includegraphics[scale=0.35]{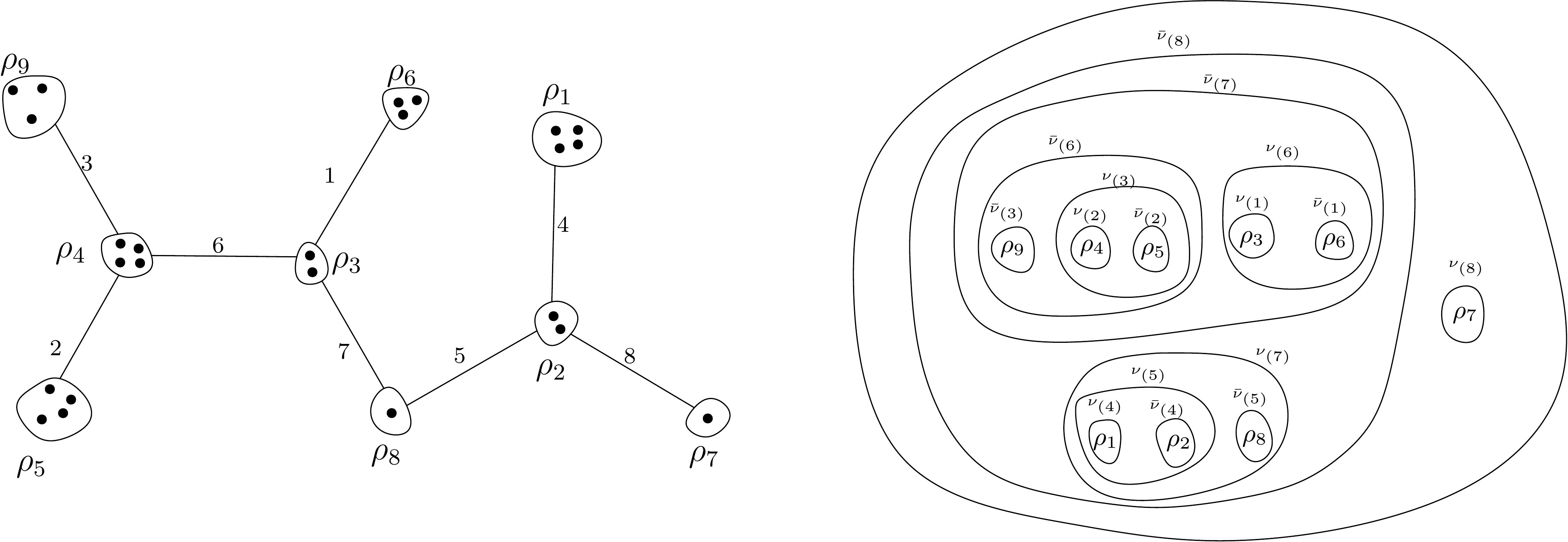}
		\caption{An example of construction of the representation $(\nu_{(i)},\bar{\nu}_{(i)})_i$ from a clustering graph.}
	\end{figure}
	
	We do then the following change of variable: 
	\[\forall i\in\{1,\cdots,\gr{r}-1\},~\hat{x}_i:=x_{\min\nu_{(i)}}-x_{\min\bar{\nu}_{(i)}},~\tilde{X}_i:=(x_j-x_{\min \rho_i})_{j\in\rho_i},\]
	\[X_{2,l}\mapsto(\hat{x}_1\cdots,\hat{x}_{\gr{r}-1},\tilde{X}_1,\cdots,\tilde{X}_{\gr{r}}).\]
	
	We begin  by integrating the condition $\mathcal{R}_{((s_i,\bar{s}_i),(\kappa_i))}$ with respect to $(\hat{x}_1,\cdots,\hat{x}_{\gr{r}-1})$ with the relative positions inside a cluster $\tilde{X}_i$ kept constant. The $(\Delta_{\mathfrak{p}_i})_i$ will be  will be integrating with respect to $(\tilde{X}_i)_i$ later.
	
	Fix $t_{i+1}$ the time of the $(i+1)$-th clustering collision and the relative positions $\hat{x}_{i-1},\cdots,\hat{x}_{1}$. We define the $i$-th clustering set
	\[B_i := \bigcup_{\substack{q\in \bigcup_{j\in\nu_{(i)}}\rho_j\\\bar{q}\in \bigcup_{\bar{\jmath}\in\bar{\nu}_{(i)}}\rho_{\bar{\jmath}}}} B_i^{q,\bar{q}}\]
	with
	\[B_i^{q,\bar{q}}:=\Big\{\hat{x}_i ~ \Big\vert~\exists t_i\in[0,t_{i+1}\wedge T_i],~|\text{x}_{\bar{q}}(t_i)-\text{x}_{\bar{q}}(t_i)|=\e\Big\}\]
	and $T_i:=2\theta$ for the the $(\gr{r}-n_k)_+$ first collisions, $t$ else.
	
	Up to time $t_i$ the curve $\text{x}_q$ and $\text{x}_{\bar{q}}$ are independant. Hence we can do the change of variable $\hat{x}_i\mapsto (t_i,\eta_i)$ with $t_i$ the minimal collision time and
	\[\eta_i = \frac{\text{x}_{\bar{q}}(t_i)-\text{x}_{\bar{q}}(t_i)}{|\text{x}_{\bar{q}}(t_i)-\text{x}_{\bar{q}}(t_i)|}.\]
	The Jacobian of this diffeomorphism is $\mu_\e^{-1}|(\text{v}_{\bar{q}}(t_i)-v_{\bar{q}}(t_i))\cdot\eta_i|dt_id\eta_i$. We integrate and we apply Cauchy-Schwarz inequality , using that kinetic energy associated with cluster $\rho_{\nu_{(i)}}$ is non-increasing (we can only remove particles) up to time $t_i$.
	
	Note that
	\[\begin{split}
	\sum_{\substack{q\in\nu_{(i)}\\\bar{q}\in\bar{\nu}_{(i)}}}|\text{v}_{\bar{q}}(t_i)-\text{v}_{\bar{q}}(t_i)|&\leq \|\text{V}_{\rho_{\nu_{(i)}}}(t_i)\|\,|\rho_{\nu_{(i)}}|^{1/2}|\rho_{\bar{\nu}_{(i)}}|+\|\text{V}_{\rho_{\bar{\nu}_{(i)}}}(t_i)\|\,|\rho_{\bar{\nu}_{(i)}}|^{1/2}|\rho_{\bar{\nu}_{(i)}}|\\[-15pt]
	& \leq\Big(|\rho_{\nu_{(i)}}|+\|V_{\rho_{\nu_{(i)}}}\|^2\Big)\Big(|\rho_{\bar{\nu}_{(i)}}|+\|V_{\rho_{\bar{\nu}_{(i)}}}\|^2\Big)\\
	&\leq\sum_{\substack{\nu_i\in\nu_{(i)}\\\bar{\nu}_{i}\in\bar{\nu}_{(i)}}}\Big(|\rho_{\nu_i}|+\|V_{\rho_{\nu_i}}\|^2\Big)\Big(|\rho_{\bar{\nu}_i}|+\|V_{\rho_{\bar{\nu}_i}}\|^2\Big).
	\end{split}\]
	This gives the following bound on $|B_i|$
	\[\begin{split}
	|B_i|&\leq \frac{C}{\mu_\e}\int^{t_{i+1}\wedge T_i}_0dt_i\sum_{q,\bar{q}}|\text{v}_{\bar{q}}(t_i)-\text{v}_{\bar{q}}(t_i)|\\
	&\leq\frac{C}{\mu_\e}\sum_{\substack{\nu_i\in\nu_{(i)}\\\bar{\nu}_{i}\in\bar{\nu}_{(i)}}}\Big(|\rho_{\nu_i}|+\|V_{\rho_{\nu_i}}\|^2\Big)\Big(|\rho_{\bar{\nu}_i}|+\|V_{\rho_{\bar{\nu}_i}}\|^2\Big)\int^{t_{i+1}\wedge T_i}_0dt_i.
	\end{split}\]
	
	Permuting the product and the sum, 
	\[\begin{split}
	\sum_{(\nu_{(i)},\bar{\nu}_{(i)})}\prod_{i=1}^ {\gr{r}-1}\Big(|\rho_{\nu_{(i)}}|+\|V_{\rho_{\nu_{(i)}}}\|^2\Big)\Big(|&\rho_{\bar{\nu}_{(i)}}|+\|V_{\rho_{\bar{\nu}_{(i)}}}\|^2\Big)\\[-10pt]
	&=\sum_{(\nu_{(i)},\bar{\nu}_{(i)})}\prod_{i=1}^ {\gr{r}-1}\sum_{\substack{\nu_i\in\nu_{(i)}\\\bar{\nu}_{i}\in\bar{\nu}_{(i)}}}\Big(|\rho_{\nu_i}|+\|V_{\rho_{\nu_i}}\|^2\big)\big(|\rho_{\bar{\nu}_i}|+\|V_{\rho_{\bar{\nu}_i}}\|^2\Big)\\[-5pt]
	&=\sum_{(\nu_{i},\bar{\nu}_{i})}\prod_{i=1}^ {\gr{r}-1}\big(|\rho_{\nu_i}|+\|V_{\rho_{\nu_i}}\|^2\big)\big(|\rho_{\bar{\nu}_i}|+\|V_{\rho_{\bar{\nu}_i}}\|^2\big).
	\end{split}\]
	
	Using that
	\[\forall a,b\in\mathbb{N},~\frac{(a+b)!}{a!b!}\leq 2^{a+b},\]
	we have 
	\[\begin{split}
	\int_0^tdt_{\gr{r}-1}\cdots\int_0^{t_2\wedge T_2}dt_1&\leq\frac{t^{n_k\wedge \gr{r}-1}}{(n_k\wedge \gr{r}-1)!}\frac{\theta^{(\gr{r}-n_k)_+}}{((\gr{r}-n_k)_+)!}\leq 2^{\gr{r}-1}\frac{t^{n_k\wedge \gr{r}-1}\theta^{(\gr{r}-n_k)_+}}{(\gr{r}-1)!}.
	\end{split}\]
	
	We can sum now on every clustering collision:
	\[\begin{split}
	\int&\ind_{\mathcal{R}^{\underline{\rho},\underline{\mathfrak{p}}}_{((s_i,\bar{s}_i),(\kappa_i)}}d\hat{x}_1\cdots\hat{x}_{\gr{r}-1}\leq \sum_{(\nu_{(i)},\bar{\nu}_{(i)})}\int d\hat{x}'_1 \ind_{B_1} \int d\hat{x}'_2\cdots\int d\hat{x}_{\gr{r}-1}\ind_{B_{\gr{r}-1}}\\
	&\leq\left(\frac{C}{\mu_\e}\right)^{\gr{r}-1}\int_0^tdt_{\gr{r}-1}\cdots\int_0^{t_2\wedge T_2}dt_1 \sum_{(\nu_{(i)},\bar{\nu}_{(i)})}\prod_{i=1}^ {\gr{r}-1}\Big(|\rho_{\nu_{(i)}}|+\|V_{\rho_{\nu_{(i)}}}\|^2\Big)\Big(|\rho_{\bar{\nu}_{(i)}}|+\|V_{\rho_{\bar{\nu}_{(i)}}}\|^2\Big)\\
	&\leq\left(\frac{2C}{\mu_\e}\right)^{\gr{r}-1}\frac{t^{n_k\wedge \gr{r}-1}\theta^{(\gr{r}-n_k)_+}}{(\gr{r}-1)!}\sum_{(\nu_{i},\bar{\nu}_{i})}\prod_{i=1}^ {\gr{r}-1}\Big(|\rho_{\nu_i}|+\|V_{\rho_{\nu_i}}\|^2\Big)\Big(|\rho_{\bar{\nu}_i}|+\|V_{\rho_{\bar{\nu}_i}}\|^2\Big).
	\end{split}\]
	
	Then denoting $d_i(G)$ the degree of vertices in a graph, $\mathcal{T}_{\gr{r}}$ the set of minimally (not ordinated) connected graph on $[1,\gr{r}]$,
	\[\begin{split}
	\int\ind_{\mathcal{R}^{\underline{\rho},\underline{\mathfrak{p}}}_{((s_i,\bar{s}_i),(\kappa_i)}}d\hat{x}_1\cdots\hat{x}_{\gr{r}-1}&\leq\left(\frac{2C}{\mu_\e}\right)^{\gr{r}-1}\frac{t^{n_k\wedge \gr{r}-1}\theta^{(\gr{r}-n_k)_+}}{(\gr{r}-1)!}\sum_{T^>\in\mathcal{T}^>_{\gr{r}}}\prod_{i=1}^{\gr{r}} \Big(|\rho_{i}|+\|V_{\rho_{i}}\|^2\Big)^{d_i(T^>)}\\
	&\leq\left(\frac{2C}{\mu_\e}\right)^{\gr{r}-1}t^{n_k\wedge \gr{r}-1}\theta^{(\gr{r}-n_k)_+}\sum_{T\in\mathcal{T}_{\gr{r}}}\prod_{i=1}^r \Big(|\rho_{i}|+\|V_{\rho_{i}}\|^2\Big)^{d_i(T)}.
	\end{split}\]
	
	For $A,B>0$, $x\in\mathbb{R}$, there exists a constant $C>0$ such that
	\[\left(A+x^2\right)^B e^{-\frac{x^2}{4}} \leq\left(\tfrac{4B}{e}\right)^B e^{\frac{A}{4}}.\]
	We use this inequality to bound 
	\begin{equation}
	\begin{split}
	\int\ind_{\mathcal{R}^{\underline{\rho},\underline{\mathfrak{p}}}_{((s_i,\bar{s}_i),(\kappa_i)}}&e^{-\frac{1}{4}\|V_p\|^2}d\hat{x}_1\cdots\hat{x}_{\gr{r}-1}\\
	&\leq\left(\frac{C}{\mu_\e}\right)^{\gr{r}-1}t^{n_k\wedge r-1}\theta^{(\gr{r}-n_k)_+}\sum_{T\in\mathcal{T}_{\gr{r}}}\prod_{i=1}^{\gr{r}} \Big(|\rho_{i}|+\|V_{\rho_{i}}\|^2\Big)^{d_i(T)}e^{-\frac{1}{4}\sum_{i=1}^{\gr{r}}\|V_{\rho_i}\|^2}\\
	&\leq \tilde{C}^{l}\frac{t^{n_k\wedge \gr{r}-1}\theta^{(\gr{r}-n_k)_+}}{\mu_\e^{\gr{r}-1}}\sum_{T\in\mathcal{T}_{\gr{r}}}\prod_{i=1}^{\gr{r}} d_i(T)^{d_i(T)}
	\end{split}
	\end{equation}
	
	We use now that for fixed $(d_1,\cdots,d_{\gr{r}})$ such that $\sum_i d_i = 2(n-1)$, 
	\begin{equation}
		\left|\left\{T\in\mathcal{T}_{\gr{r}}\big|\forall i\leq\gr{r},\,d_i(T)=T\right\}\right|=\frac{(\gr{r}-2)!}{(d_1-1)!\cdots(d_{\gr{r}}-1)!}
	\end{equation}
	(see section 2 of \cite{BGSS}), which leads to
	 can use now the following usual estimates:
	\begin{equation}\begin{split}
		\sum_{T\in\mathcal{T}_{\gr{r}}}\prod_{i=1}^{\gr{r}} d_i(T)^{d_i(T)}&=(\gr{r}-2)!\sum_{\substack{d_1,\cdots,d_{\gr{r}}\\ \gr{r}-1\geq d_i\geq 1\\ \sum_i d_i=2(\gr{r}-1)}}\prod_{i=1}^{\gr{r}} \frac{d_i^{d_i}}{(d_i-1)!}\\
		&\leq (\gr{r}-2)!C^{\gr{r}} \sum_{\substack{d_1,\cdots,d_{\gr{r}-1}\\ \gr{r}-1\geq d_i\geq 1\\ r-1\leq \sum_i d_i\leq 2\gr{r}-3}} 1\\
		&\leq C^{\gr{r}}(\gr{r}-2)!\frac{(2\gr{r}-3)^{\gr r-1}}{(\gr r -1)!}\leq \tilde{C}^{l}(\gr{r}-1)!.
	\end{split}\end{equation}
	
	We can integrate now the condition $\Delta_{\mathfrak{p}_i}(Z_{\rho_i})$. The particles in $Z_{\rho_i}$ have to form a distance cluster. Thus every particles in a ball of radius $|\rho_i|\delta\mathbb{V}$ in $\Lambda^{|\rho_i|-1}$ and because clusters are of size at most $\gamma$,
	\[\int_{\Lambda^{|\rho_i|-1}\times(\mathbb{R}^d)^{|\rho_i|}}\Delta_{\mathfrak{p}_i}(Z_{\rho_i}) \tfrac{e^{-\frac{1}{4}\|V_{\rho_i}\|^2}}{(2\pi)^{d|\rho_i|/2}}d\tilde{X}_{i}dV_{\rho_i}\leq C_\gamma \mu_\e^{-|\rho_i|+1}\left(\delta^d\mathbb{V}^d\mu_\e\right)^{|\rho_i|-1}.\]
	In addition, for at least one $\rho_i$, the set family $\underline{\varpi}^i$ is not empty. So we can apply estimate \eqref{estimation recollision} and combining the two estimations
	\[\begin{split}
	\int\ind_{\mathcal{R}^{\underline{\rho},\underline{\mathfrak{p}}}_{((s_i,\bar{s}_i),(\kappa_i)}}&(Z_{lp})\prod_{i=1}^{\gr{r}}\Delta_{\mathfrak{p}_i}(Z_{\rho_i})M^{\otimes p}(V_p)dX_{2,p}dV_{p}\\
	&\leq(\gr{r}-1)! \tilde{C}^{p}\frac{t^{n_k\wedge \gr{r}-1}\theta^{(\gr{r}-n_k)_+}}{\mu_\e^{\gr{r}-1}}\prod_{i=1}^{\gr{r}}\left(\int\Delta_{\mathfrak{p}_i}(Z_{\rho_i}) \tfrac{e^{-\frac{1}{4}\|V_{\rho_i}\|^2}}{(2\pi)^{d|\rho_i|/2}}d\tilde{X}_{i}dV_{\rho_i}\right)\\
	&\leq(\gr{r}-1)! C^{p}\frac{t^{n_k\wedge \gr{r}-1}\theta^{(\gr{r}-n_k)_+}}{\mu_\e^{\gr{r}-1}} \left(\frac{\delta^d\mathbb{V}^d\mu_\e}{\mu_\e}\right)^{\left(\sum_{i=1}^\gr{r}|\rho_{i}|-1\right)-2}\left(\frac{\delta}{\mu_\e}\right)^2\e^\alpha.
	\end{split}\]
	
	Every particles annihilated in the time interval $[0,\delta]$ have a clustering collision in this interval and thus is in a distance interval. Thus $\sum_{i=1}^\gr{r}(|\rho_{i}|-1)$ is bigger than $p-n'$. In addition we have choose $\theta$ bigger than $\delta^d\mathbb{V}^d\mu_\e$ (which is a power of $\e$) and
	\[\begin{split}
	\int&\ind_{\mathcal{R}^{\underline{\rho},\underline{\mathfrak{p}}}_{((s_i,\bar{s}_i),(\kappa_i)}}(Z_{p})\prod_{i=1}^{\gr{r}}\Delta_{\mathfrak{p}_i}(Z_{\rho_i})M^{\otimes p}dX_{2,p}dV_{p}\leq (\gr{r}-1)!\frac{C^{p}}{\mu_\e^{p-1}}t^{n_k-1} \theta^{(p-n_k-2)_+}\delta^2\e^\alpha.
	\end{split}\]
	
	We sum now on the parameters $((s_i,\bar{s}_i),(\kappa_j))$ and $(\mathfrak{p}_i)$. Because size of $\delta\mathbb{V}$-distance clusters are bounded by $\gamma$, the $|\mathfrak{P}(\rho_i)|$ are smaller than some $C_\gamma>0$ depending only on $\gamma$. The conditioning bound also the number of collision parameters  $((s_i,\bar{s}_i),(\kappa_j))$ by $(4\gamma)^{n''}$. Thus
	\[\int\big|\Phi_{\underline{n},n',n'',p}^r(Z_{p})\big|M^{\otimes p}dX_{2,p}dV_{p}\leq \frac{\|h\|(CC_\gamma 4\gamma)^{p}}{p!\mu_\e^{p-1}}t^{n_k-1}\theta^{(p-n_k-2)_+} \delta^2\e^\alpha\sum_{r=1}^{p}\sum_{\underline{\rho}\in\mathcal{P}^r_{p}}(\gr{r}-1)!\]
	
	\[\begin{split}
	\frac{1}{p!}\sum_{r=1}^{p}\sum_{\underline{\rho}\in\mathcal{P}^r_{p}}(\gr{r}-1)! = 	\frac{1}{p!}\sum_{\gr{r}=1}^{p}\sum_{\substack{k_1+\cdots+k_\gr{r}=p\\k_i\geq 1}} \frac{p!}{k_1!\cdots k_\gr{r}!}\frac{(\gr{r}-1)!}{\gr{r}!}\leq\sum_{\gr{r}=1}^{p}\sum_{\substack{k_1+\cdots+k_\gr{r}=p\\k_i\geq 1}} \frac{1}{k_1!\cdots k_\gr{r}!}\leq e^p
	\end{split}\]
	hence
	\[\int\big|\Phi_{\underline{n},n',n'',p}^r(Z_{p})\big|M^{\otimes p}dX_{2,p}dV_{p}\leq \frac{\|h\|p\big(e\tilde{C}\big)^{p}}{\mu_\e^{p-1}} t^{k-1} \theta^{(p-n_k-2)_+}\delta^2\e^\alpha\]
	
	Which ends the proof of the first inequality.\\~
	\end{proof}

	\begin{proof}[Proof of \eqref{Estimation reco loc 2}.] We begin applying \eqref{borne 1 sur phi l} to bound $|\Phi_{\underline{n},n',n'',p}^r(Z_{p})\Phi_{\underline{n},n',n'',p}^r(Z_{m},Z_{p+1,2p-m})\big|$:
	\begin{equation}
	\everymath={\displaystyle}
	\begin{array}{l@{}l}
	|\Phi_{\underline{n},n',n'',p}^{k'}(Z_{p})\Phi_{\underline{n},n',n'',p}^{k'}(Z_{p+1-m,2p-m})\big|\\[5pt]
	\leq\frac{\|h\|^2}{(p!)^2} \sum_{\substack{((s_i,\bar{s}_i),(\kappa_j))\\((s_i,\bar{s}_i),(\kappa_j))}}\sum_{\substack{ \omega_1\sqcup \omega_2 \sqcup \omega_3= [p]\\\substack{ \omega'_1\sqcup \omega'_2 \sqcup \omega'_3= [p+1-m,2p-m]\\|\omega_1|=|\omega_1'|=n'\\|\omega_2| = |\omega'_2|= p-n''}}}&\sum_{\gr{p},\gr{p}'\geq 1}\sum_{\substack{\underline{\varpi}\in\mathcal{Q}^{\gr{p}}_{\omega_1,\omega_2}\\\underline{\varpi}'\in\mathcal{Q}^{\gr{p}'}_{\omega_1',\omega_2'}}}\ind_{\mathcal{R}_{((s_i,\bar{s}_i),(\kappa_j))}}(Z_{\omega_1\cup\omega_3})\\[-8pt]
	&\times\ind_{\mathcal{R}_{((s'_i,\bar{s}'_i),(\kappa_j))}}(Z_{\omega'_1\cup\omega'_3})\prod_{i=1}^{\gr{p}}\chi(Z_{\varpi_i})\prod_{i=1}^{\gr{p}'}\chi(Z_{\varpi'_i}).\end{array}
	\end{equation}
	
	Note that the right hand side is invariant under translation. Thus one can fix $x_1 = 0$ and integrate with respect the other variables.
	
	For a position $Z_{2p-m}$, we consider $\underline{\rho}:=(\rho_1,\cdots,\rho_{\gr{r}})$ the $\delta\mathbb{V}$-cluster. We can then construct the parameters $\mathfrak{p}_i:=(\underline{\omega}^i,\underline{\omega}'{^i},\underline{\varpi}^i,\underline{\varpi}'{^i})$:
	\begin{itemize}
		\item $\underline{\omega}^i:=(\omega^i_1,\omega^i_2,\omega^i_3)$ is a partition of $\rho_i\cap[1,p]$ defined by $\underline{\omega}^i_j:=\omega_j\cap\rho_i$,
		\item $\underline{\omega}'{^i}:=(\omega^i_1,\omega^i_2,\omega^i_3)$ is a partition of $\rho_i\cap[p+1-m,2p+m]$ defined by $\underline{\omega}^i_j:=\omega_j\cap\rho_i$,
		\item$\underline{\varpi}^i:=\{\varpi_j {\rm ~such~that~}\varpi_j\subset\rho_i\}$ and 
		\item  $\underline{\varpi}'{^i}:=\{\varpi'_j {\rm ~such~that~} \varpi'_j\subset\rho_i\}$.
	\end{itemize}
	We denote now $\mathfrak{P}(\rho_i)$ the new set of possible parameter $\mathfrak{p}_i$ (this will not create a conflict with the previous section). Because each cluster $\rho_i$ is of size at most $\gamma$, $|\mathfrak{P}(\rho_i)|$ is bounded by some constant $C_\gamma$ depending only on $\gamma$. We define
	\[\Delta_{\mathfrak{p}_i}(Z_{\rho_i}):=\ind_{Z_{\rho_i}\text{\,form\,a\,distance\,cluster}}\prod_{j=1}^{|\underline{\varpi}^i|}\chi(Z_{\varpi_j^i})\prod_{j=1}^{|\underline{\varpi}'{^i}|}\chi(Z_{\varpi_j'{^i}})\text{~and}\]
	\[\mathcal{R}^{\underline{\rho},\underline{\mathfrak{p}}}_{\substack{((s_i,\bar{s}_i),(\kappa_i))\\((s'_i,\bar{s}'_i),(\kappa'_i))}} :=  \Big\{Z_{2p-m}\in\mathcal{D}^{\underline{\rho}}_\e,\,Z_{\omega_1\cup\omega_3}\in\mathcal{R}_{((s_i,\bar{s}_i),(\kappa_i)},Z_{\omega'_1\cup\omega'_3}\in\mathcal{R}_{((s'_i,\bar{s}'_i),(\kappa'_i)}\Big\}\]
	and we have has in the previous case
	\begin{equation*}
	\begin{split}
	|\Phi_{\underline{n},n',n'',p}^{k'}(Z_{p})&\Phi_{\underline{n},n',n'',p}^{k'}(Z_{m},Z_{p+1,2p-m})\big|
	\\
	&\leq \frac{\|h\|^2}{(p!)^2}\sum_{\gr{r}=1}^{2p-m}\sum_{\underline{\rho}\in\mathcal{P}^r_{p}}\sum_{\substack{((s_i,\bar{s}_i),(\kappa_j))\\((s'_i,\bar{s}'_i),(\kappa'_j))\\ \underline{\mathfrak{p}}\in\underset{i}{\prod}\mathfrak{P}(\rho_i)}}~\ind_{\mathcal{R}^{\underline{\rho},\underline{\mathfrak{p}}}_{\substack{((s_i,\bar{s}_i),(\kappa_i)\\((s'_i,\bar{s}'_i),(\kappa'_i)}}}(Z_{2p-m})\prod_{i=1}^{\gr{r}}\Delta_{\mathfrak{p}_i}(Z_{\rho_i}).
	\end{split}
	\end{equation*}
	Note for at least one $i$, $\underline{\varpi}^i$ is not empty. We construct now a clustering tree in order to estimates $\mathcal{R}^{\underline{\rho},\underline{\mathfrak{p}}}_{\substack{((s_i,\bar{s}_i),(\kappa_i)\\((s'_i,\bar{s}'_i),(\kappa'_i)}}$.
	
	Consider the collision graph associated with the first pseudotrajectory $\mathcal{G}^{[0,t-t_s]}_{\omega_1\cup\omega_3}$  and  the graph associated with second one $\mathcal{G}^{[0,t-t_s]}_{\omega_1'\cup\omega_3'}$. Merge them and identify vertices in a same cluster $\rho_i$. Finally we keep only the first clustering collisions, and we obtain the orientated tree $T^>:=(\nu_i,\bar{\nu}_i)_{1\leq i\leq \gr{r}-1}$. Note that these clustering collisions can happen in the first or in the second pseudotrajectory.
	
	As in the proof of \eqref{Estimation reco loc 1} we have to bound the number of grazing collisions of $T^>$ in the time interval $[0,2\tau]$. There are atmost $(n_k-1+p-m)$ collision in $[(k'+1)\delta,t-t_s]$ ($n_k-1$ for the first pseudotrajctory and we have to connect $p-m$ particles in the second). Thus there are at least $(\gr{r}-(n_k-1+p-m))_+$ clustering collisions in $[\delta,(k'+1)\delta]\subset[0,2\tau]$.
	
	We explain quickly how to estimate the $i$-th collision. As we in the previous paragraph we construct the modified tree parameters $(\nu_{(i)},\bar{\nu}_{(i)})$ and the change of variable 
	\[\forall i\in\{1,\cdots,\gr{r}-1\},~\hat{x}_i:=x_{\min\nu_{(i)}}-x_{\min\bar{\nu}_{(i)}},~\tilde{X}_i:=(x_j-x_{\min \rho_i})_{j\in\rho_i},\]
	\[X_{2,l}\mapsto(\hat{x}_1\cdots,\hat{x}_{\gr{r}-1},\tilde{X}_1,\cdots,\tilde{X}_r),\]
	and we integrate the clustering on the $(\hat{x}_i)$. 
	
	Collision can occur one of the two pseudotrajectories. The clustering set $B_i$ is defined as follows: fix $t_{i+1}$ the time of the $(i+1)$-th clustering collision and the relative positions $\hat{x}_{i-1},\cdots,\hat{x}_{1}$. We define the $i$-th clustering set
	\[B_i := \bigcup_{\substack{q\in \bigcup_{j\in\nu_{(i)}}\rho_j\\\overline{q}\in \bigcup_{\bar{\jmath}\in\bar{\nu}_{(i)}}\rho_{\bar{\jmath}}} }\Big(B_{i}^{q,\bar{q}}\cup B_{i}'^{q,\bar{q}}\Big)\]
	with
	\[B_{i}^{q,\bar{q}}:=\Big\{\hat{x}_i ~ \Big\vert~\exists t_i\in[0,t_{i+1}\wedge T_i],~|\text{x}_{\bar{q}}(t_i)-\text{x}_{\bar{q}}(t_i)|=\e\Big\},\]
	where $\text{x}_i(\tau)$ is the pseudotrajectory with respect to parameters $((s_i,\bar{s}_i)_i,(\kappa_j)_j)$ and $T_i:=2\theta$ for the the $(\gr{r}-n_k)_+$ first collisions, $t$ else, and $B_{i}'^{q,q'}$ is defined in the same way for the other pseudotrajectory. We can apply the estimation of the previous paragraph:
	\[\int \ind_{B_i}d\hat{x}_i\leq\frac{2C}{\mu_\e}\sum_{\substack{\nu_i\in\nu_{(i)}\\\bar{\nu}_i\in\bar{\nu}_{(i)}}}\Big(|\rho_{\nu_i}|+\|V_{\rho_{\nu_i}}\|^2\Big)\Big(|\rho_{\bar{\nu}_i}|+\|V_{\rho_{\bar{\nu}_i}}\|^2\Big)\int_{0}^{t_{i+1}\wedge T_i}dt_i.\]
	
	We finally come back to the situation of the estimation of \ref{Estimation reco loc 1}, and we can apply the same strategy:
	\[\begin{split}
	\int	|\Phi_{\underline{n},n',n'',p}^{k'}(Z_{p})\Phi_{\underline{n},n',n'',p}^{k'}&(Z_{p+1-m,2p-m})\big|M^{\otimes (2p-m)}dX_{2,2p-m}dV_{2p-m}\\
	&\leq \frac{(2p-m)!\|h\|^2}{(p!)^2\mu_\e^{2p-m-1}}C^{p}\delta^2\e^\alpha \tau^{(p-n_k-2)_+}t^{n_k-1+p-m}\\
	&\leq \frac{\|h\|^2}{p^m\mu_\e^{2p-m-1}}\tilde{C}^{p}\delta^2\e^\alpha \tau^{(p-n_k-2)_+}t^{n_k-1+p-m}
	\end{split}\]
	which concludes the proof.
	\end{proof}

	\begin{proof}[Proof of \eqref{Estimation reco loc 3}]
	In $\mathsf{f}_{p\leftarrow l}\big[\Phi_{\underline{n},n'}^{0,k'}[h]\big]\left(Z_{[1,l]}\right)\mathfrak{X}_{n',p}(Z_{[1,p]})$ we have three set of indices:
	\begin{itemize}
		\item  $[1,n']$ the set of particles created in the final pseudotrajectory,
		\item $[n'+1,p]$ the particles added in the treatment of local recollision and
		\item $[p+1,l]$ particles added in the dynamical cluster development.
	\end{itemize}
	
	Any permutation $\sigma$ which sends $[1,n']$, $[n'+1,p]$ and $[p+1,l]$ onto themselves stabilizes \[\mathsf{f}_{p\leftarrow l}\big[\Phi_{\underline{n},n'}^{0,k'}[h]\big]\left(Z_{[1,l]}\right)\mathfrak{X}_{n',p}(Z_{[1,p]})\] and
	
	\[\Phi_{\underline{n},n',p,l}^r(Z_{l}) = \frac{n'!\,(p-n')!\,(l-p)!}{l!}\sum_{\substack{ \omega_1\sqcup \omega_2 \sqcup \omega_3= [l]\\|\omega_1|=n'\\|\omega_2| = l-p}}\mathsf{f}_{p\leftarrow l}\big[\Phi_{\underline{n},n'}^{0,k'}[h]\big]\left(Z_{\omega_1\cup\omega_2},Z_{\omega_3}\right)\mathfrak{X}_{n',p}(Z_{\omega_1},Z_{\omega_2})\]
	
	We develop $\mathsf{f}_{p\leftarrow l}\big[\Phi_{\underline{n},n'}^{0,k'}[h]\big]$: for $\underline{\omega}=(\omega_1,\omega_2,\omega_3)$, $\underline{\lambda}$ two partitions of $[1,l]$  with $\omega_1\cup \omega_2\subset\lambda_1$ and $(s_i,\bar{s}_i)_{1\leq i\leq n'-1}$, we define $\mathcal{R}^{\underline{\omega},\underline{\lambda}}_{(s_i,\bar{s}_i)}\subset\mathcal{D}^l_\e$ the set of initial data such that particles in $\lambda_1$ form a $(\omega_1\cup\omega_2)$-cluster (see the previous part for the definition of cluster pseudotrajectories), and the tree pseudotrajectory $\Zt_{n'}(\tau,(s_i,\bar{s}_i),\Zc^{\lambda_1}_{\omega_1}(\delta))$ with $\sum_{i=1}^j k_i$ particles at time $\tau:= (k'+1)\delta+ (k-j)\tau$. Then we can write:
	\[\begin{split}
	&\Phi_{\underline{n},n',p,l}^{k'}(Z_{l})\\
	&\begin{split}=-\frac{1}{l!}\sum_{\substack{ \omega_1\sqcup \omega_2 \sqcup \omega_3= [l]\\|\omega_1|=n'\\|\omega_2| = p-n'}}\sum_{\gr{l}=1}^{l}\sum_{\substack{\lambda_1 \subset [l]\\ \omega_1\cup \omega_2 \subset \lambda_1}}\sum_{\substack{(\lambda_2,\cdots,\lambda_\gr{l})\\ \in\mathcal{P}^{\gr{l}-1}_{\lambda_1^c}}}\sum_{(s_i,\bar{s}_i)}\sum_{\substack{\gr{p}\geq1\\[1pt]\underline{\varpi}\in\mathcal{Q}^\gr{p}_{\omega_1,\omega_2}}}\ind_{\mathcal{R}^{\underline{\omega},\underline{\lambda}}_{(s_i,\bar{s}_i)}}h(\Zt_{n'}(\tau,(s_i,\bar{s}_i),\Zc^{\lambda_1}_{\omega_1}(\delta)))\prod_{i=1}^{\gr{p}}\chi\big(Z_{\varpi_i}\big)\\[-20pt]
	\times\prod_{i=1}^{n}\bar{s}_i\prod_{i=2}^{\gr{l}}\varphi_{|\lambda_i|}\big(Z_{\lambda_i}\big) \psi_{\gr{l}}\big(Z_{\lambda_1},\cdots,Z_{\lambda_{\gr{l}}}\big).
	\end{split}
	\end{split}\]
	We recall the Penrose's tree inequality (see for example the second section of \cite{BGSS} for a proof)
	\[\Big|\psi_{\gr{l}}\big(Z_{\lambda_1},\cdots,Z_{\lambda_{\gr{l}}}\big)\Big|=\left|\sum_{C\in\mathcal{C}(\omega)}\prod_{(i,j)\in E(C)}-\ind_{\lambda_i\so \lambda_j}\right|\leq\sum_{T\in\mathcal{T}_{\gr{l}}}\prod_{(i,j)\in E(T)}\ind_{\lambda_i\so\lambda_j}.\]
	Hence we obtain the following bound on $\Phi_{\underline{n},n',p,l}^{k'}$, invariant under translation:
	\begin{equation}
	\begin{split}
	\Big|\Phi_{\underline{n},n',p,l}^{k'}\Big|(Z_{l})\leq\frac{\|h\|}{l!}\sum_{\substack{\underline{\omega}\in\mathcal{P}^3_{l}\\|\omega_1|=n'\\|\omega_2|=p}}\sum_{\gr{l}=1}^{l}\sum_{\substack{\underline{\lambda}\in\mathcal{P}^{\gr{l}}_{l}\\\omega_1\cup\omega_2\subset\lambda_1}}\sum_{(s_i,\bar{s}_i)}\sum_{\gr{p}\geq 1}\ind_{\mathcal{R}^{\underline{\omega},\underline{\lambda}}_{(s_i,\bar{s}_i)}}\prod_{i=1}^{\gr{p}}\chi\big(Z_{\varpi_i}\big)\prod_{i=2}^{\gr{l}}\varphi_{|\lambda_i|}\big(Z_{\lambda_i}\big) \\[-20pt]
	\times\sum_{T\in\mathcal{T}_{\gr{l}}}\prod_{(i,j)\in E(T)}\ind_{\lambda_i\so\lambda_j}.
	\end{split}
	\end{equation}
	
	We will use again the distance cluster to control relation between particles in the time interval $[0,\delta]$. Let $\underline{\rho}:=(\rho_1,\cdots,\rho_{\gr{r}})$ the distance partition of $Z_{l}$. For each $\rho_i$, we construct the collision parameter $\mathfrak{p}:=(\underline{\omega}^i,\underline{\lambda}^i,\underline{\varpi}^i)$ with:
	\begin{itemize}
	\item $\underline{\omega}^i:=(\omega^i_1,\omega^i_2,\omega^i_3)$ is a partition of $\rho_i\cap[1,p]$ defined by $\underline{\omega}^i_j:=\omega_j\cap\rho_i$,
	\item $\underline{\lambda}^i:=\{\lambda_1^i:=\lambda_1\cap\rho_i\}\cup\{\lambda_j {\rm~for~}j\geq 2 {\rm~with~}\lambda_j\subset\rho_i\}$ a partition of $\rho^i$ and
	\item $\underline{\varpi}^i:=\{\varpi_j {\rm ~such~that~}\varpi_j\subset\rho_i\}$,
	\end{itemize}
	and we denote $\mathfrak{P}(\rho_i)$ the set of possible $\mathfrak{p}_i$.
	
	The global conditioning bound velocities so particles which make a collisionnal cluster have to be in a same distance cluster. Thus for each $\varpi_j$ and $\lambda_k,\,k\geq2$ there exists a $\rho_i$ containing $\lambda_k$ or $\varpi_j$. In addition for $i\neq i'$, particles in $\lambda_1^i$ do not interact with particles of $\lambda_1^{i'}$. The overlap are also contained in the distance cluster: if two dynamical clusters $\lambda_j$ and $\lambda_{j'}$ with $j,j'\geq 2$, there exists a $\rho_i$ containing the both, and if $\lambda_j\subset\rho_i$ has an overlap with $\lambda_1$, then $\lambda_j$ has an overlap with $\lambda_1^i$. This last property allows us to rewrite the overlap cumulant: on $\mathcal{D}^{\underline{\rho}}_\e$,
	\[\Big|\psi_{\gr{l}}\big(Z_{\lambda_1},\cdots,Z_{\lambda_{\gr{l}}}\big)\Big|\leq\sum_{T\in\mathcal{T}_{\gr{l}}}\prod_{(i,j)\in E(T)}\ind_{\lambda_i\so\lambda_j}\leq \prod_{i= 1}^{\gr{r}}\sum_{T_i\in\mathcal{T}_{|\rho_i|}}\prod_{(j,j')\in E(T_i)}\ind_{\lambda_j^i\so\lambda^i_{j'}}\leq\prod_{i= 1}^{\gr{l}}\Big|\mathcal{T}(\rho_i)\Big|.\]
	
	We have now the following bound
	\begin{equation*}\big|\Phi_{\underline{n},n',p,l}^{k'}(Z_{l})\big|
	\leq \frac{\|h\|}{l!}\sum_{\gr{r}=1}^{l}\sum_{\underline{\rho}\in\mathcal{P}^{\gr{r}}_{l}}\sum_{\substack{(s_i,\bar{s}_i)\\\underline{\mathfrak{p}}\in\underset{i}{\prod}\mathfrak{P}(\rho_i)}}~\ind_{\mathcal{R}^{\underline{\rho},\underline{\mathfrak{p}}}_{(s_i,\bar{s}_i)}}\prod_{i=1}^{\gr{r}}\Delta_{\mathfrak{p}_i}(Z_{\rho_i})
	\end{equation*}
	with 
	\[\Delta_{\mathfrak{p}_i}(Z_{\rho_i}):=\Big|\mathcal{T}(\rho_i)\Big|\ind_{Z_{\rho_i}\text{\,form\,a\,distance\,cluster}}\prod_{j=1}^{|\underline{\varpi}^i|}\chi(Z_{\varpi_j^i})\text{~and}\]
	\[\mathcal{R}^{\underline{\rho},\underline{\mathfrak{p}}}_{(s_i,\bar{s}_i)} :=  \Big\{Z_{l}\in\mathcal{D}^{\underline{\rho}}_\e,\,Z_{l}\in\mathcal{R}_{(s_i,\bar{s}_i)}^{\underline{\omega},\underline{\lambda}}\Big\}.\]
	
	Finally we have to construct a clustering tree : we consider the collision graph of the particles $\omega_1$ on the time interval $[\delta,t-t_s]$. Then we identify vertices in a same cluster $\rho_i$ and we keep only the first clustering collision. This constructs an ordered tree $T^>\in\mathcal{T}^>_{r}$. As in the previous cases, respecting the collision history $T^>$ depends only on the relative position at time $\delta$ which are the same than at time $0$ (cluster do not interact).
	We can apply the same method than in the estimation \eqref{Estimation reco loc 1} and we obtain the expected bound.
	\end{proof}

	\begin{proof}[Proof of \eqref{Estimation reco loc 4}]
	We have to adapt the proof of \eqref{Estimation reco loc 2} with the parametrisation of the previous part.
	\end{proof}

	\section{Treatment of the principal part}\label{Treatment of the main part}
	\subsection{Duality formula}
	We recall that
	\begin{equation*}
	G_\e^{\text{main}}(t) = \sum_{\substack{n_1\leq\cdots\leq n_K\\n_j-n_{j-1}\leq 2^j}} \mathbb{E}_\e\left[\mu_\e^{-1/2}\sum_{(i_1,\cdots,i_{n_K})}\Phi^0_{\underline{n}}[h]\left(\gr{Z}_{\underline{i}}(0)\right)\zeta^0_\e(g)\right]= \sum_{\substack{n_1\leq\cdots\leq n_K\\n_j-n_{j-1}\leq 2^j}} \mathbb{E}_\e\left[\mu_\e^{n_k}\,\hat{\Phi}^0_{\underline{n}}[h]\,\hat{g}\right]
	\end{equation*}
	where $\Phi^0_{\underline{n}}[h]$ is the development of $h(z_i(t))$ along pseudotrajectories tree with $n_k$ particles at time $t-n_k\theta$ and no recollision.
	
	We denotes 
	\begin{equation}
	g^\e_n(Z_n):= \left(\sum_{k=1}^n g(z_k)\right)~ \frac{1}{\mathcal{Z}_\e}\sum_{p\geq0}\frac{\mu_\e^p}{p!}\int \,e^{-\mathcal{V}_{n+p}(X_{n+p})}dX_{n+1,n+p}.
	\end{equation}
	Then using the equality \eqref{quasi covariance} and $L^1$ estimations on $\Phi^0_{\underline{n}}$ of \ref{Clustering estimations}, we have for $h$ and $g$ in $L^\infty$
	\begin{equation*}
	\begin{split}
	&G_\e^{\text{main}}(t) = \sum_{\substack{n_1\leq\cdots\leq n_K\\n_j-n_{j-1}\leq 2^j}} \mathbb{E}_\e\left[\mu_\e^{-1}\sum_{\ui_{n_K}}\Phi^0_{\underline{n}}[h]\left(\gr{Z}_{\ui_{n_k}}(0)\right)\sum_{j=1}^{n_K}g(\gr{z}_{i_j}(0))\right]+O\left(\e \sum_{\substack{\underline{n}}} (Ct)^{n_k}\|h\|\|g\|\right)\\
	&=\sum_{\substack{\underline{n}}} \int\mu_\e^{n_K-1}\Phi^0_{\underline{n}}[h]\left(Z_{n_K}\right)g_{n_K}^\e(Z_{n_K})\,\frac{e^{-\mathcal{H}_{n_K}(Z_{n_K})}dZ_{n_K}}{(2\pi)^\frac{n_Kd}{2}}+O\left(\e \left(K2^{K^2} (Ct)^{2^{K+1}}\|h\|\|g\|\right)\right).
	\end{split}
	\end{equation*}
	
	We want to compute the asymptotic of each terms in the sum.
	\[\begin{split}\int&\mu_\e^{n_K-1}\Phi^0_{\underline{n}}[h]\left(Z_{n_k}(0)\right)g_{n_k}^\e(Z_{n_k})\,\frac{e^{-\mathcal{H}_{n_K}(Z_{n_K})}dZ_{n_K}}{(2\pi)^\frac{n_Kd}{2}}\\
	&=\frac{\mu_\e^{n_K-1}}{n_K!}\int\sum_{(s_i,\bar{s}_i)_i} \prod_{i= 1}^{n_K-1} \bar{s}_i \ind_{\mathcal{R}_{(s_i,\bar{s}_i)}^{\underline{n}}}h\big(\Zt_{n_k}(t)\big)g_{n_K}^\e(Z_{n_K})\,M^{\otimes n_K}dZ_{n_K}\end{split}\]
	where $\mathcal{R}_{(s_i,\bar{s}_i)}^{\underline{n}}\subset\mathcal{D}^{n_K}_\e$ is the set of initial parameters such that for each $k\in[0,K]$, the pseudotrajectory $\Zt_{n_K}(\tau,(s_i,\bar{s}_i))$ has $n_k$ particles at time $t-k\theta$ and no recollision.
	
	We order now the annihilation. We fix an initial position $Z_{n_K}$. Given collision parameters $(s_i,\bar{s}_i)_i$, we can construct a collision tree $(a_i,b_i)$ where the $i$-th removed particle is $b_i$, after a collision with $a_i$. We have a one to one correspondence between the admissible $(a_i,b_i)_i$ and the $(s_i)_i$, thus we can change the collision parameters to $(a_i,b_i,\bar{s}_i)$. The $(b_i)_{1\leq i\leq n_K}$ is the annihilation order. Due to the symmetry of $g_{n_K}^\e$, we can reorder particle such $b_i = n_K-i+1$. Denoting $\tilde{a}_i:= a_{n_K-i+1}$, $\tilde{s}_i:= \bar{s}_{n_K-i+1}$ and $\mathcal{R}^{\underline{n}}_{(\tilde{a}_i,\tilde{s}_i)_i}$ the set of initial parameters respecting collision parameters $(\tilde{a}_i,\tilde{s}_i)_{2\leq i\leq n_K}$, 
	\[\begin{split}\int\mu_\e^{n_K-1}&\Phi^0_{\underline{n}}[h]\left(Z_{n_K}\right)g_{n_K}^\e(Z_{n_K})^{\otimes n_K}dZ_{n_K}\\
	&=\mu_\e^{n_K-1}\sum_{(\tilde{a}_i,\tilde{s}_i)_i} \prod_{i= 2}^{n_K} \tilde{s}_i\int_{\mathcal{R}^{\underline{n}}_{(\tilde{a}_i,\tilde{s}_i)_i}}h(\Zt_{n_K}(t)))g_{n_K}^\e(Z_{n_K})M^{\otimes n_K}dZ_{n_K}.\end{split}\]
	Note that the admissible $(a_i)_{2\leq i\leq n_K}$ verifies $a_i\in[1,i-1]$.
	
	We define now the \emph{backward speudocharacteristic}
	\[\xi^\e_{n_K}(\tau,(\tilde{a}_i,\tilde{s}_i)_i,z_1,(t_i,{\bar{v}}_i,\eta_i)_{2\leq i\leq n_K})\] 
	with a final point $z_1$ and parameters $(\tau_i,\bar{v}_i,\eta_i)_{2\leq i\leq n_K}\in (\mathbb{R}^+\times\mathbb{R}^d\times\mathbb{S}^{d-1})^{n_K-1}$ with $t>t_2>\cdots>t_{n_K}>0$. We construct sequentially the speudotrajectory on each $[t_{i+1},t_i]$. We begin at time $t$ with particle $1$ at $z_1$.  the coordinate of the  pseudocharacteristic at time $\tau\in(t_{i},t_{i-1})$. In the interval $(t_{i+1},t_i)$, there is $i-1$ particles $\xi^\e_{n_K}(\tau)=(z^\e_1(\tau),\cdots,z^\e_{i-1}(\tau))$ which move along straight  line (backwardly). At time $t_i^+$, we add particle $i$ at position $(x^\e_{\tilde{a}_i}(\tau)+\e\eta_i,\bar{v}_i)$. If $\tilde{s}_i=1$ we apply the scattering between particles $\tilde{a}_i$ and $i$, else the particles do not interact.
	
	Note that the $v_i^\e(\tau)$ does not depends on $\e$.
	
	We denote $\mathbb{G}^{\underline{n},0}_{(\tilde{a}_i,\tilde{s}_i)_i}(z_1)$ and $\mathbb{G}^{\underline{n},\e}_{(\tilde{a}_i,\tilde{s}_i)_i}(z_1)$ the definition set of pseudocharecteristics: for $z_1\in\Lambda\times\mathbb{R}^d$
	\[\begin{split}
	\mathbb{G}^{\underline{n},0}_{(\tilde{a}_i,\tilde{s}_i)_i}(z_1):=\Big\{&(t_i,{\bar{v}}_i,\eta_i)_{i}\in(\mathbb{R}\times\mathbb{R}^d\times\mathbb{S}^{d-1})^{n_K-1}\Big|\,t>t_2>\cdots>t_{n_K}>0,\\
	 &\forall i\in[n_{j-1}+1,n_j],t_i\in(t-j\theta,t-(j-1)\theta),~ (v^\e_i(t_i^+)-\bar{v}_i)\cdot\eta_i<0\Big\},
	 \end{split}\]
	and $\mathbb{G}^{\underline{n},\e}_{(\tilde{a}_i,\tilde{s}_i)_i}(z_1)$ the subset of $\mathbb{G}^{\underline{n},0}_{(\tilde{a}_i,\tilde{s}_i)_i}(z_1)$ such that distances between particles are bigger than $\e$ expect when a particle is created (the trajectories without \emph{overlap})
	
	Then we can do the change of variable
	\begin{equation}
	\begin{array}{c}
	\underset{z_1\in\mathbb{D}}\bigcup\{z_1\}\times\mathbb{G}^{\underline{n},0}_{(\tilde{a}_i,\tilde{s}_i)_i}(z_1)\longrightarrow \mathcal{R}^{\underline{n}}_{(\tilde{a}_i,\tilde{s}_i)_i}\\[6pt]
	(z_1,(t_i,{\bar{v}}_i,\eta_i)_{i})\longmapsto\xi^\e_{n_K}(\tau=0).
	\end{array}
	\end{equation}
	Because we have removed all the recollision, this map is a bijection. It is a local diffeomorphism hence a diffeomorphism. It sends measure
	\begin{equation}
	M(v_1)dz_1d\Lambda^{\underline{n}}_{(\tilde{a}_i,\tilde{s}_i)_i}:=M(v_1)dz_1~\prod_{i=2}^{n_K}\big((v^\e_{a(i)}(t_i^+)-\bar{v}_i)\cdot\eta_i\big)_+M(v_i)d\bar{v}_id\eta_idt_i\end{equation}
	onto $\mu_\e^{n_k-1}M^{\otimes n_K}dZ_{n_K}$. We will denote with a little abuse of notation:
	\[\mathbb{D}\times \mathbb{G}^{\underline{n},\e}_{(\tilde{a}_i,\tilde{s}_i)_i} := \bigcup_{z_1\in\mathbb{D}}\{z_1\}\times\mathbb{G}^{\underline{n},\e}_{(\tilde{a}_i,\tilde{s}_i)_i}(z_1).\]
	
	Finally we can write the following duality formula 
	\begin{equation}\begin{split}
	\mu_\e^{n_K-1}\int&\Phi^0_{\underline{n}}[h]\left(Z_{n_K}\right)g_{n_K}^\e(Z_{n_K})M^{\otimes n_K}(V_{n_K})dZ_{n_K}\\
	&=\sum_{(\tilde{a}_i,\tilde{s}_i)_i} \prod_{i= 1}^{n_K-1} \tilde{s}_i\int_{\mathbb{D}\times \mathbb{G}^{\underline{n},\e}_{(\tilde{a}_i,\tilde{s}_i)_i} }h(z_1)g_{n_K}^\e(\xi^\e_{n_K}(0))M(v_1)dz_1d\Lambda^{\underline{n}}_{(\tilde{a}_i,\tilde{s}_i)_i}.
	\end{split}\end{equation}
	
	Denoting
	\[g_{n_K}(Z_{n_K}):=\sum_{i=1}^{n_K}g(z_i),\]
	we have formally
	\[\begin{split}\int\mu_\e^{n_K-1}&\Phi^0_{\underline{n}}\left(Z_{n_K}\right)g_{n_K}^\e(Z_{n_K})M^{\otimes n_K}(V_{n_K})dZ_{n_K}\\
	&\underset{\e\rightarrow 0}{\longrightarrow}~\sum_{(\tilde{a}_i,\tilde{s}_i)_i} \prod_{i= 1}^{n_K-1} \tilde{s}_i\int_{\mathbb{D}\times \mathbb{G}^{\underline{n},0}_{(\tilde{a}_i,\tilde{s}_i)_i}}h(z_1)g_{n_K}(\xi^0_{n_K})M^{\otimes n_K}(V_{n_K})dZ_{n_K}d\Lambda^{\underline{n}}_{(\tilde{a}_i,\tilde{s}_i)_i}.\end{split}\]
	
	In order to have an explicit rates of convergence we decompose the error in three parts:
	\begin{equation}\begin{split}
	&\int\mu_\e^{n_K-1}\,\Phi^0_{\underline{n}}\,g_{n_K}^\e M^{\otimes n_K}(V_{n_K})dZ_{n_K}\\
	&=\sum_{(\tilde{a}_i,\tilde{s}_i)_i} \prod_{i= 1}^{n_K-1} \tilde{s}_i\int_{\mathbb{D}\times \mathbb{G}^{\underline{n},0}_{(\tilde{a}_i,\tilde{s}_i)_i}}h(z_1)g_{n_K}(\xi^0_{n_K}(0))M(v_1)dz_1d\Lambda^{\underline{n}}_{(\tilde{a}_i,\tilde{s}_i)_i}+R_1+R_2+R_3
	\end{split}\end{equation}
	
	\[R_1=\sum_{(\tilde{a}_i,\tilde{s}_i)_i} \prod_{i= 1}^{n_K-1} \tilde{s}_i\int_{\mathbb{D}\times \mathbb{G}^{\underline{n},0}_{(\tilde{a}_i,\tilde{s}_i)_i}}h(z_1)\left(g_{n_K}(\xi^\e_{n_K}(0))-g_{n_K}(\xi^0_{n_K}(0))\right)M(v_1)dz_1d\Lambda^{\underline{n}}_{(\tilde{a}_i,\tilde{s}_i)_i}\]
	\[R_2=-\sum_{(\tilde{a}_i,\tilde{s}_i)_i} \prod_{i= 1}^{n_K-1} \tilde{s}_i\int_{\mathbb{D}\times \mathbb{G}^{\underline{n},0}_{(\tilde{a}_i,\tilde{s}_i)_i}}h(z_1)g_{n_K}(\xi^\e_{n_K}(0))\left(1-\ind_{\mathbb{G}^{\underline{n},\e}_{(\tilde{a}_i,\tilde{s}_i)_i}(z_1)}\right)M(v_1)dz_1d\Lambda^{\underline{n}}_{(\tilde{a}_i,\tilde{s}_i)_i}\]
	\[R_3=\sum_{(\tilde{a}_i,\tilde{s}_i)_i} \prod_{i= 1}^{n_K-1} \tilde{s}_i\int_{\mathbb{D}\times \mathbb{G}^{\underline{n},\e}_{(\tilde{a}_i,\tilde{s}_i)_i}}h(z_1)\left(g^\e_{n_K}(\xi^\e_{n_K}(0))-g_{n_K}(\xi^\e_{n_K}(0))\right)M(v_1)dz_1d\Lambda^{\underline{n}}_{(\tilde{a}_i,\tilde{s}_i)_i}.\]

	They are estimated using the following usual estimations:
	\begin{lemma}\label{Borne sur la taille des parrametres d'arbre}
		Fix $\bar{n}:=(n_1,\cdots,n_k)$ and for any $\e>0$ sufficiently small, we have for $p\in[1,2]$ and $z_1\in \mathbb{D}$
		\begin{equation}\begin{split}
		\sum_{(\tilde{a}_i,\tilde{s}_i)_i} \int_{ \mathbb{G}^{\underline{n},0}_{(\tilde{a}_i,\tilde{s}_i)_i}(z_1)}\left(\prod_{i=2}^{n_K}\big\|v^\e_{\tilde{a}_i}(t_i^+)-\bar{v}_i\big\|^p\,M(\bar{v}_i)d\bar{v}_id\eta_idt_i\right)\frac{e^{-\frac{1}{2}\|v_1\|^2}}{(2\pi)^{d/2}}\\
		\leq  (C(K-1)\theta)^{n_{K-1}}(C\theta)^{n_K-n_{K-1}} e^{-\frac{\|v_1\|^2}{4}},	
		\end{split}\end{equation}
	\end{lemma}
	\begin{proof}
		We use the same proof than the Lemma 4.2 in \cite{SP2}.
		
		For $i\in[2,n_K]$ we forget parameters $(\tilde{a}_j)_{i<j\leq n_K}$ and $(t_j,{\bar{v}}_j,\eta_j)_{i<j\leq n_K}$.
		\begin{align*}
		\sum_{\tilde{a}_i=1}^{i-1}&\big\|v^\e_{\tilde{a}_i}(t_i^+)-\bar{v}_i\big\|^pe^{-\frac{\|v_1\|^2+\sum_{j= 2}^{i-1}\|\bar{v}_j\|^2}{8n_K}-\frac{\|\bar{v}_i\|^2}{8}}\\
		&\leq 2^{p-1}\left[\sum_{j= 1}^{i-1}\|v_j(t_i^+)\|^p+(i-1)\|\bar{v}_i\|^p\right]e^{-\frac{\|v_1\|^2+\sum_{j= 2}^{i-1}\|\bar{v}_j\|^2}{8n_K}+\frac{\|\bar{v}_i\|^2}{8}}\\
		&\leq 2^{p-1} \left[\left(\sum_{j= 1}^{i-1}\|v_j(t_i^+)\|^2\right)^{p/2}(i-1)^{1-p/2}+(i-1)\|\bar{v}_i\|^p\right]e^{-\frac{\|v_1\|^2+\sum_{j= 2}^{i-1}\|\bar{v}_j\|^2}{8n_K}+\frac{\|\bar{v}_i\|^2}{8}}
		\end{align*}
		\begin{align*}
		\sum_{\tilde{a}_i=1}^{i-1}&\big\|v^\e_{\tilde{a}_i}(t_i^+)-\bar{v}_i\big\|^pe^{-\frac{\|v_1\|^2+\sum_{j= 2}^{i-1}\|\bar{v}_j\|^2}{8n_K}-\frac{\|\bar{v}_i\|^2}{8}}\\
		&\leq 2^{p-1} \left[\left(\|v_1\|^2+\sum_{j= 2}^{i-1}\|\bar{v}_j\|^2\right)^{p/2}(i-1)^{1-p/2}+(i-1)\|\bar{v}_i\|^p\right]e^{-\frac{\|v_1\|^2+\sum_{j= 2}^{i-1}\|\bar{v}_j\|^2}{8n_K}+\frac{\|\bar{v}_i\|^2}{8}}\\
		&\leq C \left[n_K^{p/2}(i-1)^{1-p/2}+(i-1)\right]\leq C n_K.
		\end{align*}
		Thus
		\[\begin{split}
		&\sum_{(\tilde{a}_i,\tilde{s}_i)_i}\int_{ \mathbb{G}^{\underline{n},0}_{(\tilde{a}_i,\tilde{s}_i)_i}(z_1)}\left(\prod_{i=2}^{n_K}\big\|v^\e_{\tilde{a}_i}(t_i^+)-\bar{v}_i\big\|^p\,M(\bar{v}_i)d\bar{v}_id\eta_idt_i\right)\frac{e^{-\frac{1}{2}\|v_1\|^2}}{(2\pi)^{d/2}}\\
		&\leq \int_{\theta}^{K\theta}dt_2\cdots\int_\theta^{t_{n_{K-1}-1}}dt_{n_{K-1}}\int_0^\theta dt_{n_{K-1}+1}\cdots\int_0^{t_{n_K}-1} dt_{n_K}\\
		&~~~~~~~~~~~~~~~~~~~~~~~~~~\times\int_{(\mathbb{S}^{d-1}\times\mathbb{R}^d)^{n_K-1}} C^{n_K-1}n_{K}^{n_K-1}e^{-\frac{\|v_1\|^2+\sum_{j= 2}^{i-1}\|\bar{v}_j\|^2}{4}}\left(\prod_{i=2}^{n_K}\frac{d\bar{v}_id\eta_i}{(2\pi)^{d/2}}\right)\\
		&\leq \frac{C(C(K-1) (n_K-1)\theta)^{n_{K-1}-1}}{(n_{K-1}-1)!}\frac{(C\theta n_{K})^{n_{K}-n_{K-1}}}{(n_{K}-n_{K-1})!}\leq (\tilde{C}t)^{n_K}.
		\end{split}\]
	\end{proof}
	
	\begin{lemma}\label{Borne sur les chevauchements}
		Fix $\bar{n}:=(n_1,\cdots,n_K)$ and for any $\e>0$ sufficiently small, we have
		\begin{equation}
		\sum_{(\tilde{a}_i,\tilde{s}_i)_i} \int_{\mathbb{D}\times \mathbb{G}^{\underline{n},0}_{(\tilde{a}_i,\tilde{s}_i)_i}}\Big|1-\ind_{\mathbb{G}^{\underline{n},\e}_{(\tilde{a}_i,\tilde{s}_i)_i}(z_1)}\Big|\,M(v_1)dz_1d\Lambda^{\underline{n}}_{(\tilde{a}_i,\tilde{s}_i)_i}\leq  (Ct)^{n_K}\e^\alpha,	
		\end{equation}
	\end{lemma}

	The estimations \eqref{Borne sur les chevauchements} is an estimation of the set of parameter leading to an overlap. It can be in the same than the estimation of recollision of Section \ref{Quasi-orthogonality estimates}.
	
	From Lemma \ref{Borne sur les chevauchements} we deduce
	\[|R_1|\leq C(Ct)^{n_K}\e^\alpha\|g\|\,\|h\|.\]
	
	\begin{lemma}\label{Borne sur les chevauchements bis}
		Fix $\bar{n}:=(n_1,\cdots,n_k)$, $\e>0$ sufficiently small, and $X_{n_K}\in\Lambda^{n_K}$ such that for  $i\neq j$ \[|x_i-x_j|>\e.\]Then
		\begin{equation}
		~\Bigg| 1 - \frac{1}{\mathcal{Z}_\e}\sum_{p\geq0}\frac{\mu_\e^p}{p!}\int e^{-\mathcal{V}_{n_K+p}(X_{n_K},\bar{X}_p)}d\bar{X}_{p}\Bigg|\leq C^{n_K}\e
		\end{equation}
	\end{lemma}

	\begin{proof}
		Using the formula \eqref{decomposition e^V_n+p}, for any $X_{n_K}\in\Lambda^{n_K}$ with $|x_i-x_j|>\e$ for  $i\neq j$,
		\[\begin{split}
		\exp\left(-\mathcal{V}_{n_K+p}^\e(X_{n_K},\underline{X}_p)\right)&=\sum_{\substack{\omega\subset[1,p]}}e^{-\mathcal{V}_{n_K}^\e(X_{n_K})-\mathcal{V}_{|\omega^c|}^\e(\underline{X}_{\omega^c})}\psi_p^{n_K}(X_{n_K},\underline{X}_\omega)\\
		&=\sum_{\substack{\omega\subset[1,p]}}e^{-\mathcal{V}_{|\omega^c|}^\e(\underline{X}_{\omega^c})}\psi_p^{n_K}(X_{n_K},\underline{X}_\omega).
		\end{split}\]
		Then
		\[\begin{split}
		\sum_{p\geq0}\frac{\mu_\e^p}{p!}\int e^{-\mathcal{V}_{n_K+p}(X_{n_K},\bar{X}_p)}d\bar{X}_{p} &= \sum_{p\geq0}\sum_{p_1+p_2=p}\frac{\mu_\e^p}{p!}\frac{p!}{p_1!p_2!}\int e^{-\mathcal{V}_{p_2}^\e(\underline{X}'_{p_2})}\psi_{p_1}^{n_K}(X_{n_K},\underline{X}_{p_1})d\bar{X}_{p_1}d\bar{X}'_{p_2}\\
		&=\mathcal{Z}_\e \sum_{p\geq0}\frac{\mu_\e^p}{p!}\int \psi_p^{n_K}(X_{n_K},\underline{X}_{p})d\bar{X}_{p}\\
		&=\mathcal{Z}_\e\bigg(1+\sum_{p\geq1}\frac{\mu_\e^p}{p!}\int \psi_p^{n_K}(X_{n_K},\underline{X}_{p})d\bar{X}_{p}\bigg).
		\end{split}\]
		Using the estimation  \eqref{Borne de l'integrale de psi np}, 
		\[\sum_{p\geq1}\frac{\mu_\e^p}{p!}\int \psi_p^{n_K}(X_{n_K},\underline{X}_{p})d\bar{X}_{p}\leq\sum_{p\geq1}\frac{\mu_\e^p}{p!}  (p-1)!\big(Ce\e^d\big)^pn_Ke^{n_K}\leq \sum_{p\geq1}\big(C'\e\big)^pn_Ke^{n_K}\leq 2\e n_Ke^{n_K}\]
		for $\e$ small enough. This conclude the proof.
	\end{proof}
	Using Lemmata \ref{Borne sur la taille des parrametres d'arbre} and \ref{Borne sur les chevauchements bis} we obtain
	\[|R_3| = C(Ct)^{n_K}\e \|g\|\,\|h\|.\]
	
	\begin{lemma}\label{Borne entre la trajectoire limite et la trajectoire d'enskog}
		Fix $\bar{n}:=(n_1,\cdots,n_k)$, $\e>0$ and $(z_1,(t_i,{\bar{v}}_i,\eta_i)_{i})\in\mathbb{D}\times\mathbb{G}^{\underline{n},\e}_{ {(\tilde{a}_i,\tilde{s}_i)_i}}$, we have
		\begin{equation}
		\big|\xi^\e_{n_K}(0)-\xi^0_{n_K}(0)\big|\leq n_K^{3/2}\e
		\end{equation}
	\end{lemma}
	\begin{proof}
		We recall first that the two trajectories $\xi^\e_{n_K}(\tau)$ and $\xi^0_{n_K}(\tau)$ have same velocities and at each creation of a particle a there is a new shift of size $\e$. Thus for any $i$ bigger than $1$, $\|x_i^\e(\tau)-x_i^0(\tau)\|\leq (i-1)\e$ and summing it, \[\|\xi^\e_{n_K}(\tau)-\xi^0_{n_K}(\tau)\|^2 \leq n_K^3\e^2.\]
	\end{proof}
	For if $g$ is uniformly Lipschitz, we can applied Lemmata \ref{Borne sur la taille des parrametres d'arbre} and \ref{Borne entre la trajectoire limite et la trajectoire d'enskog},
	\[|R_1| = C(Ct)^{n_K}\e \|\nabla g\|\,\|h\|.\]
	
	Finally we gets for $h$ and $g$ Lipschitz
	\[\begin{split}
	\int\mu_\e^{n_K-1}\,\Phi^0_{\underline{n}}\,g_{n_K}^\e M^{\otimes n_K}dZ_{n_K}	=\sum_{(\tilde{a}_i,\tilde{s}_i)_i} \prod_{i= 1}^{n_K-1} \tilde{s}_i\int_{\mathbb{D}\times \mathbb{G}^{\underline{n},0}_{(\tilde{a}_i,\tilde{s}_i)_i}}h(z_1)g_{n_K}(\xi^0_{n_K}(0))M(v_1)dz_1d\Lambda^{\underline{n}}_{(\tilde{a}_i,\tilde{s}_i)_i}\\
	+O\bigg(\e^\alpha (Ct)^{n_K}\|h\|\big(\|g\|+ \|\nabla g\|\big)\bigg). \end{split}\]
	and summing it,
	\begin{equation}\label{Estimation morceau 5}
	\begin{split}
	G_\e^{\text{main}}(t) =\sum_{\substack{n_1\leq\cdots\leq n_K\\n_j-n_{j-1}\leq 2^j}}\sum_{(\tilde{a}_i,\tilde{s}_i)_i} \prod_{i= 1}^{n_K-1} \tilde{s}_i\int_{\mathbb{D}\times \mathbb{G}^{\underline{n},0}_{(\tilde{a}_i,\tilde{s}_i)_i}}h(z_1)g_{n_K}(\xi^0_{n_K})M(v_1)dz_1d\Lambda^{\underline{n}}_{(\tilde{a}_i,\tilde{s}_i)_i}\\
	+\,O\left(\e^\alpha K2^{K^2} (Ct)^{2^{K+1}}\|h\|\big(\|g\|+\|\nabla g\|\big)\right).
	\end{split}
	\end{equation}
	
	\subsection{Linearized Boltzmann equation}
	Let $\gr{g}(t)$ be the solution of the linearized Boltzmann equation:
	\begin{equation*}
	\begin{array}{c}
	\partial_t \gr{g}(t) + v\cdot\nabla_x f(t) = \mathcal{L} f(t) \text{~for~}(t,x,v)\in[0,\infty)\times\mathbb{D}\\[6pt]
	\gr{g}(t=0)=g \text{~on~}\mathbb{D},
	\end{array}
	\end{equation*}
	and $\mathcal{L}$ is the linearized Boltzmann operator:
	\begin{equation*}
	\mathcal{L} g(v) := \int_{\mathbb{S}^{d-1}\times\mathbb{R}^d} \big(g(v')+g(\bar{v}')-g(v)-g(\bar{v})\big)((v-\bar{v})\cdot\eta)_+M(\bar{v})d\eta\,d\bar{v}
	\end{equation*}
	and $(v',\bar{v}')$ defined by the scattering \eqref{scatering}.
	
	We can this equation in the Duhamel form: denoting $S(\tau)$ the semigroup associated with $v\cdot\nabla_x$,
	\[\gr{g}(t)=S(t)g +\int_0^tS(t-\tau_1)\mathcal{L}f(\tau_1)d\tau_1.\]
	We want to iterate this formula, whiles steel cutting trees with surexponential growth of number of annihilation time (as in the hard sphere system): defining 
	\[Q_{m,n}(\tau)[g] = \int_0^\tau dt_{m+1}\int_0^{t_{m+1}}\cdots\int_0^{t_{n-1}}dt_n S(t-t_{m+1})\mathcal{L}S(t_{m+1}-t_{m+2})\cdots\mathcal{L}S(t_n)g,\]
	for $\underline{n}:=(n_1,\cdots,n_k)$ with $1\leq n_1\leq \cdots\leq n_k$,
	\[Q_{\underline{n}}(\tau)g=Q_{1,n_1}(\tfrac{\tau}{k})Q_{n_1,n_2}(\tfrac{\tau}{k})\cdots Q_{n_{k-1},n_k}(\tfrac{\tau}{k})[g],\]
	we have
	\begin{equation}\label{decompo g(t)}
	\begin{split}
	\gr{g}(t)=\sum_{\substack{n_1\leq\cdots\leq n_K\\n_j-n_{j-1}\leq 2^j}} Q_{\underline{n}}(t) [g]+\sum_{k=1}^K \sum_{\substack{n_1\leq\cdots\leq n_{k-1}\\n_j-n_{j-1}\leq 2^j}}\sum_{n_k>2^k}Q_{\underline{n}}(k\tau) [\gr{g}(t-k\theta)].
	\end{split}
	\end{equation}
	
	If $g$ is continuous and bounded, we have the following characteristic formula for $Q_{1,2}(t)[g]$
	\[\begin{split}
	Q_{1,2}(\tau)&[g](x,v ) = \int_0^\tau d\tau_2 S(\tau-\tau_2) \mathcal{L}S(\tau_2)[g](x,v)\\
	&=\int_0^\tau d\tau_2 \int_{\mathbb{S}^2\times\mathbb{R}^3} \Big(S(\tau_2)[g](x-(t-\tau_2)v,v')+S(\tau_2)[g](x-(t-\tau_2)v,\bar{v}_2')\\
	&-S(\tau_2)[g](x-(t-\tau_2)v,\bar{v})-S(\tau_2)[g](x-(t-\tau_2)v,v)\Big) \big((v-\bar{v})\cdot\eta\big)_+M(v_*)d\eta d\bar{v}d\tau_2\\
	&=\int_{\mathbb{G}^{(2),0}_{(1,1)}}g_2(\xi_2^0)d\Lambda^{(2)}_{(1,1)}-\int_{\mathbb{G}^{(2),0}_{(1,-1)}}g_2(\xi_2^0)d\Lambda^{(2)}_{(1,-1)}
	\end{split}\]
	where we denote as in the previous paragraph
	\[g_n(\theta,Z_n):=\sum_{i= 1}^ng(t,z_i).\]
	We can iterate this construction:
	\begin{equation}
	Q_{\underline{n}}(t)[g](z_1) =\sum_{(\tilde{a}_i,\tilde{s}_i)_i} \prod_{i= 1}^{n_K-1} \tilde{s}_i\int_{\mathbb{G}^{\underline{n},0}_{(\tilde{a}_i,\tilde{s}_i)_i}(z_1)}g_{n_K}(\xi^0_{n_K})~d\Lambda^{\underline{n}}_{(\tilde{a}_i,\tilde{s}_i)_i}.
	\end{equation}
	
	This formula gives to things: first term of \eqref{decompo g(t)} correspond to the main part of $\mathbb{E}_\e\big[\zeta_\e^t(h)\zeta_\e^0(g)\big]$ in \eqref{Estimation morceau 6}. Second it give the following $L^2$ estimation:
	\begin{prop}
		There exists a constant $C$ such that for any $g\in L^2(M(v)dz)$, and $\underline{n}:=(n_1,\cdots,n_k)$,
		\begin{equation}
		\big\|Q_{\underline{n}}(k\theta)g \big\|_{L^2(M^2(v)dz)}\leq \big(C(k-1)\theta\big)^{\frac{n_{k-1}}{2}}\big(C\theta\big)^{\frac{n_k-n_{k-1}}{2}}\,\|g\|_{L^2(M(v)dz)}.
		\end{equation}
	\end{prop}
	
	\begin{proof}
		The proof is given in section 4.4 of \cite{BGS}. We suppose that $g$ is continuous in order to use the pseudocharacteristic formula and we conclude by density.
		
		Using Cauchy-Schwartz inequality,
		\[\begin{split}
		&\big\|Q_{\underline{n}}(k\theta)g \big\|^2_{L^2(M^2(v)dz)}   \\
		&=\int_{\mathbb{D}} \Bigg(\sum_{(\tilde{a}_i,\tilde{s}_i)_i} \prod_{i= 1}^{n_K-1} \tilde{s}_i\int_{\mathbb{G}^{\underline{n},0}_{(\tilde{a}_i,\tilde{s}_i)_i}(z_1)}g_{n_K}(\xi^0_{n_K})~d\Lambda^{\underline{n}}_{(\tilde{a}_i,\tilde{s}_i)_i}\Bigg)^2 M^2(v_1)dz_1\\
		&\leq\int_{\mathbb{D}} \Bigg(M(z_1)\sum_{(\tilde{a}_i,\tilde{s}_i)_i} \int_{\mathbb{G}^{\underline{n},0}_{(\tilde{a}_i,\tilde{s}_i)_i}(z_1)}\,d\Lambda^{q,\underline{n}}_{(\tilde{a}_i,\tilde{s}_i)_i}\Bigg)\sum_{(\tilde{a}_i,\tilde{s}_i)_i} \int_{\mathbb{G}^{\underline{n},0}_{(\tilde{a}_i,\tilde{s}_i)_i}(z_1)}g_{n_K}^2(\xi^0_{n_K})~d\Lambda^{b,\underline{n}}_{(\tilde{a}_i,\tilde{s}_i)_i}M(v_1)dz_1
		\end{split}\]
		where
		\[d\Lambda^{b,\underline{n}}_{(\tilde{a}_i,\tilde{s}_i)_i}:=M(v_1)dz_1~\prod_{i=2}^{n_K}\frac{\big((v^\e_{a(i)}(t_i^+)-\bar{v}_i)\cdot\eta_i\big)_+}{1+\big\|v^\e_{a(i)}(t_i^+)-\bar{v}_i\big\|}M(\bar{v}_i)d\bar{v}_id\eta_idt_i,\]
		\[d\Lambda^{q,\underline{n}}_{(\tilde{a}_i,\tilde{s}_i)_i}:=M(v_1)dz_1~\prod_{i=2}^{n_K}{\big((v^\e_{a(i)}(t_i^+)-\bar{v}_i)\cdot\eta_i\big)_+}\left(1+\big\|v^\e_{a(i)}(t_i^+)-\bar{v}_i\big\|\right) M(\bar{v}_i)d\bar{v}_id\eta_idt_i.\]
		
		From \eqref{Borne sur les chevauchements} we have the bound 
		\[\Bigg(M(z_1)\sum_{(\tilde{a}_i,\tilde{s}_i)_i} \int_{\mathbb{G}^{\underline{n},0}_{(\tilde{a}_i,\tilde{s}_i)_i}(z_1)}\,d\Lambda^{q,\underline{n}}_{(\tilde{a}_i,\tilde{s}_i)_i}\Bigg)\leq \big(C(k-1)\theta\big)^{n_{k-1}}\big(C\theta\big)^{n_k-n_{k-1}}.\]
		
		On the other hand, using the representation formula in the reverse sens,
		\[\begin{split}
		&\sum_{(\tilde{a}_i,\tilde{s}_i)_i} \int_{\mathbb{G}^{\underline{n},0}_{(\tilde{a}_i,\tilde{s}_i)_i}(z_1)}g_{n_K}^2(\xi^0_{n_K})~d\Lambda^{b,\underline{n}}_{(\tilde{a}_i,\tilde{s}_i)_i}\\
		&\leq n_K \sum_{(\tilde{a}_i,\tilde{s}_i)_i} \int_{\mathbb{G}^{\underline{n},0}_{(\tilde{a}_i,\tilde{s}_i)_i}(z_1)}\left(g^2\right)_{n_K}(\xi^0_{n_K})~d\Lambda^{b,\underline{n}}_{(\tilde{a}_i,\tilde{s}_i)_i}\\
		&\leq n_K\int_\theta^{k\theta}dt_2\cdots\int_{\theta}^{t_{n_{k-1}-1}}dt_{n_K}\int_0^\theta dt_{n_{k-1}+1}\cdots\int_0^{t_{n_k-1}}dt_{n_k}S(t-t_{2})|L^b|\cdots|L^b|S(t_n)g^2
		\end{split}\]
		with
		\[|L^b|g(v):=\int_{\mathbb{S}^{d-1}\times\mathbb{R}^d} \big(g(v')+g(\bar{v}')+g(v)+g(\bar{v})\big)\frac{((v-\bar{v})\cdot\eta)_+}{1+\|v-\bar{v}\|}M(\bar{v})d\eta\,d\bar{v}\]
		and
		\[(g^2)_{n_K}(Z_{n_K}):=\sum_{i= 1}^{n_k}g^2(z_i).\]
		
		\begin{lemma}
			The operator $|L^b|: L^1(M(v)dz)\to L^1(M(v)dz)$ is a bounded.
		\end{lemma}
		\begin{proof}
			For $f\in L^1(M(v)dz)$, using that the change of variables $(v,\bar{v},\eta)\mapsto(v',\bar{v}',\eta)$ sending $(v-\bar{v})\cdot\eta)_+dv\,d\bar{v}\,d\eta\to (v'-\bar{v}')\cdot\eta)_-dv'\,d\bar{v}'\,d\eta$,
			
			\[\int_{\mathbb{D}} |L^b|f(z)M(v)dz= 4\int\int_{\mathbb{D}\times\mathbb{S}^2\times\mathbb{R}^3} f(z)M(v)M(\bar{v})dzd\eta d\bar{v}\leq 16\pi\|f\|_{L^1(M(v)dz)}.\]
		\end{proof}
		We use now that $S(t)$ conserves the $L^1(M(v)dz)$ norm, and integrating the times variables. Hence
		\[\begin{split}\int_{\mathbb{D}} \sum_{(\tilde{a}_i,\tilde{s}_i)_i} \int_{\mathbb{G}^{\underline{n},0}_{(\tilde{a}_i,\tilde{s}_i)_i}(z_1)}&g_{n_K}^2(\xi^0_{n_K})~d\Lambda^{b,\underline{n}}_{(\tilde{a}_i,\tilde{s}_i)_i}M(v_1)dz_1\\
		&\leq \frac{(C(k-1)\theta)^{n_{k-1}}(C\theta)^{n_k-n_{k-1}}}{n_{k-1}!(n_k-n_{k-1})!}\|g\|_{L^2(M(v)dz)}\end{split}\]
		This conclude the proof of the proposition.
		\end{proof}
	
	Because $\|\gr{g}(t)\|_{L^2(M(z)dz)}$ is decreasing, we have for $\|h\|<\infty$ (we use here the weight of the norm $\|h\| \approx\sup \left| M^{-1} g\right|$).
	\begin{equation}\label{Estimation morceau 6}
	\begin{split}
	\Bigg|\Bigg<h,\sum_{k=1}^K \sum_{\substack{(n_j)_{j\leq k-1}\\n_j\leq 2^j}}\sum_{n_k>2^k}Q_{\underline{n}}(k\theta) \gr{g}(t-k\theta)\Bigg>_{L^2(M(v)dz)}\Bigg|&\leq \sum_{k= 1}^K (C^2 t \theta)^{k/2} \|h\| \|g\|_{L^2(M(v)dz)}\\
	&\leq C
	t^{1/2} \theta^{1/2} \|h\| \|g\|.
	\end{split}
	\end{equation}
	
	Using all the estimations \eqref{Estimation morceau 1}, \eqref{Estimation morceau 3}, \eqref{Estimation morceau 4}, \eqref{Estimation morceau 5} and  \eqref{Estimation morceau 6}, we gets that 
	\begin{equation}
	\mathbb{E}_\e\Big[\zeta^t_\e(h)\zeta^0_\e(g)\Big] = \Big<h,\,\gr {g}(t)\Big>_{L^2(M(v)dz)} +O\Big( \Big( C t \theta^{1/2} + (Ct)^{2^{t/\theta}}\e^{\alpha/2}\Big)\|h\|\big(\|g\|+\|\nabla g\|\big)\Big).
	\end{equation}
	
	This conclude the proof of the main theorem.

	\bibliographystyle{abbrv}
	
\end{document}